\documentclass[11pt,reqno]{amsart}
\usepackage[left=1.3in, right=1.2in, top=1in, bottom=1in, includefoot]{geometry}

\usepackage{graphicx}
\usepackage{mathrsfs} 
\usepackage{amssymb}   
\usepackage{amsthm}  
\usepackage{bbm}
\usepackage[shortlabels]{enumitem}
\usepackage{xcolor}
\usepackage{soul}

\newtheorem{thm}{Theorem}[section]

\newtheorem{prop}[thm]{Proposition}
\newtheorem{cor}[thm]{Corollary}
\newtheorem{Def}[thm]{Definition}
\newtheorem{lemma}[thm]{Lemma}
\newtheorem{remark}[thm]{Remark}

\newtheorem{obs}[thm]{Observation}
\newtheorem{nots}[thm]{Notations}
\newtheorem{DefProp}[thm]{Definition/Proposition}
\newtheorem{AssDef}[thm]{Assumption/Definition}

\newcommand{\ve}{{\varepsilon}}
\newcommand{\N}{{\mathbb{N}}}
\newcommand{\R}{{\mathbb{R}}}
\numberwithin{equation}{section}
\title{Rigorous derivation of weakly dispersive shallow water models with large amplitude topography variations} 

\author[L. Emerald]{Louis Emerald}
\author[M. Oen Paulsen]{Martin Oen Paulsen}
 
\address{Department of Mathematics, Nazarbayev University, Astana 010000, Kazakhstan}
\email{Louisemerald76@gmail.com}
\address{Department of Mathematics\\ University of Bergen\\ Postbox 7800\\ 5020 Bergen\\ Norway}
\email{Martin.Paulsen@UiB.no}

\keywords{Rigorous derivation, shallow water models, multi-scale expansion, Dirichlet-Neumann operator, pseudo-differential operators}
\subjclass[2010]{Primary: 76B15; 35Q35; 35C20}
\begin{document}

    \begin{abstract}  
        We derive rigorously from the water waves equations new irrotational shallow water models for the propagation of surface waves in the case of uneven topography in horizontal dimensions one and two. The systems are made to capture the possible change in the waves' propagation, which can occur in the case of large amplitude topography. The main contribution of this work is the construction of new multi-scale shallow water approximations of the Dirichlet-Neumann operator. We prove that the precision of these approximations is given at the order $O(\mu \ve)$, $O(\mu\varepsilon  +\mu^2\beta^2)$ and $O(\mu^2\varepsilon+\mu \ve \beta+ \mu^2\beta^2)$. Here $\mu$, $\varepsilon$, and $\beta$ denote respectively the shallow water parameter, the nonlinear parameter, and the bathymetry parameter. From these approximations, we derive models with the same precision as the ones above. The model with precision $O(\mu \ve)$ is coupled with an elliptic problem, while the other models do not present this inconvenience.

        



		
    \end{abstract}
    \maketitle

    \section{Introduction}

    \subsection{Motivations}

     The general model of surface waves in coastal oceanography is often considered too complex to be used in practical situations. As a result, the simplification of the water waves equations in specific asymptotic regimes has been a subject of active research. In the derivation of asymptotic models, one considers characteristic quantities of the system under study. In our paper, we will denote by $H_0, L, a_{\mathrm{surf}}$ and $a_{\mathrm{bott}}$, the characteristic water depth, the characteristic wavelength in the longitudinal direction, the characteristic surface amplitude and the characteristic amplitude of the bathymetry of the system. From these characteristic quantities, we define the following non-dimensional parameters
     \begin{equation*}
         \mu := \frac{H_0^2}{L^2}, \quad \ve : = \frac{a_{\mathrm{surf}}}{H_0}, \quad \beta := \frac{a_{\mathrm{bott}}}{H_0}.
     \end{equation*}
     We will focus on the shallow water regime defined by $\mu  \ll 1$ and the weakly nonlinear regime $\ve \ll 1$, in the case of uneven topography. 
    
    Numerous shallow water models were derived in the literature, giving approximations of the solutions of the water waves system in the shallow water regime or long wave regime $\mu \sim \varepsilon \ll 1$, at the order of precision $O(\mu^k)$ with $k = 1, 2$ or $3$. However, it is not clear that these classical models capture the change in the propagation of the waves which can happen in the case of large amplitude topographies. Such occurrences have been studied in the Dingemans experiments \cite{Dingemans94}. In these experiments, the authors investigate a long wave passing over a submerged obstacle. They observed that waves tend to steepen due to a compression effect from the bottom, where high harmonics generated by topography-induced nonlinear interactions are freely released behind the obstacle. This last phenomenon makes it natural to improve the frequency dispersion of the classical shallow water models. The ideal precision of the resulting model would be of the form $O(\mu^k \varepsilon^l)$, with $k, l \geq 1$, in order to capture both the shallow water and the weakly non-linear regimes. 
    
    One way to improve the frequency dispersion is to consider multi-parameters Boussinesq or Green-Naghdi models; see \cite{WWP, WeiKirbyGrilliEtAl95} for a comparison between the classical Boussinesq and Green-Naghdi models with their multi-parameters versions in the case of the Dingemans experiments. The improved frequency dispersion allowed the authors to describe strongly dispersive waves with uneven bathymetry accurately. However, the order of precision of these multi-parameters systems is the same as the classical ones. 

    In the flat bottom case, another way to improve the frequency dispersion of the classical shallow water models is to consider full dispersion models, for which the dispersion relation is the same as the one of the water waves equations:
    \begin{align*}
        \omega_{\mathrm{WW}}(\xi)^2 = \frac{\tanh{(\sqrt{\mu}|\xi|)}}{\sqrt{\mu}|\xi|}|\xi|^2.
    \end{align*}
    In \cite{Emerald21}, the author rigorously derived these models at the order of precision $O(\mu\varepsilon)$ and $O(\mu^2\varepsilon)$. To obtain this non-trivial order of precision, it is fundamental to keep the exact dispersion relation. In comparison, for the classical Boussinesq and Green-Naghdi model, the dispersion relation is 
    \begin{align*}
        \omega_{\mathrm{B}}(\xi)^2 =  \big(1 - \frac{\mu}{3}|\xi|^2\big)|\xi|^2, \quad \omega_{\mathrm{GN}}(\xi)^2 = \frac{1}{1 + \frac{\mu}{3}|\xi|^2}|\xi|^2,
    \end{align*}
    so that by a Taylor expansion, one makes errors of order $O(\mu^2)$ from the approximation of the dispersion relation of the water waves equations.
    
    In \cite{DucheneMMWW21}, the author extended the work in \cite{Emerald21} in the case of variable bottom. He derived models with a precision of order $O(\mu\varepsilon + \mu\beta)$ when compared to the water waves equations for a class of weakly dispersive Boussinesq system, and a precision order $O(\mu^2\varepsilon+\mu^2\beta)$ with respect to the water waves equations for a class of weakly dispersive Green-Naghdi systems.  
    
    The first result of this paper is the rigorous derivation of an extension of the full dispersion models in the case of uneven bathymetries at the order of precision $O(\mu\varepsilon)$. This model reads
    \begin{align}\label{system mu eps motivations}
            \begin{cases}
                \partial_t \zeta - \frac{1}{\mu}\mathcal{G}_b \psi= 0\\
                \partial_t \psi + \zeta + \dfrac{\varepsilon}{2}|\nabla_X \psi|^2 = 0,
            \end{cases}
    \end{align}
    where $\zeta(t,X) \in \mathbb{R}$ represents the water surface elevation, $b(X) \in \mathbb{R}$ represents the bottom elevation, $\psi(t,X) \in \mathbb{R}$ is the trace at $z=0$, where $z$ is the transversal variable, of the potential $\phi$, solving
    \begin{align}\label{elliptic problem motivations}
            \begin{cases}
                \Delta_{X,z}^{\mu} \phi = 0 \ \ \mathrm{in} \ \ \R^d \times [-1+\beta b, 0], \\
                \phi|_{z=0} = \psi, \ \ \big[\partial_z\phi - \mu\beta \nabla_X b \cdot \nabla_X \phi \big]\big|_{z=-1+\beta b}  = 0,
            \end{cases}
    \end{align}
    and where $\mathcal{G}_b$ is an operator given by
    \begin{equation}\label{operator Gb motivations}
            \frac{1}{\mu}\mathcal{G}_b \psi = -\nabla_{X} \cdot \Big{(} \frac{1 + \varepsilon\zeta - \beta b}{1 - \beta b} \int_{-1 + \beta b}^0 \nabla_X \phi \: \mathrm{d}z\Big{)}.
    \end{equation}
    The model \eqref{system mu eps motivations} can be viewed as a simplified version of the water waves model, where the elliptic problem is given on a fixed domain independent of time. The precision of the model is $O(\mu\varepsilon)$ and makes it an ideal extension of the full dispersion models in the case of a variable bottom with which one can capture the change of behavior in the propagation of the wave during the aforementioned Dingemans experiments. A drawback from a numerical point of view would be that one would need to solve an elliptic problem at each time step when computing the solutions of the model.

    To simplify the model even further, one can construct explicit approximations of the solutions of the elliptic problem \eqref{elliptic problem motivations}. From these approximations, one deduce expansions of the Dirichlet-Neumman operator at the same order of precision. In \cite{CraigSulemGuyenneNicholls05}, the authors derived such approximations of the Dirichlet-Neumann operator in the long wave regime/ small amplitude waves regime using an explicit formula for the solution of the elliptic problem \eqref{elliptic problem motivations}, where the formula depends on the inversion of a pseudo-differential operator. They make use of these approximations to derive a Boussinesq type model. Their result is formal and holds only in horizontal dimension one. The extension in the variable bottom case of full dispersion models was considered in \cite{CarterDinvayKalish21} in the case of horizontal dimension one. The authors used the result in \cite{CraigSulemGuyenneNicholls05} to formally derive three models in the shallow water regime with order of precision $O(\mu\varepsilon + \varepsilon^2)$. A drawback is that their models depend on the inversion of a pseudo-differential operator and consequently seem to create instabilities in the simulations. Moreover, if one inverts a pseudo-differential operator, it is not clear how one could quantify the error of approximation in the Sobolev spaces uniformly in the parameters $\mu$, $\ve$ and $\beta$ and then make the derivation rigorous with the correct order of precision. 

    In \cite{Chazel07}, the author derived rigorously, from the water waves equations, a classical type Boussinesq system in the long wave regime with an order of precision $O(\mu^2)$ when $\beta = O(1)$. Translated in the shallow water regime, the precision is $O(\mu^2 + \mu\varepsilon + \mu^2\beta^2)$. One should also note that the bathymetry related terms in the aforementioned system are of higher order when compared to the linear terms. Therefore, the well-posedness of such a system is not clear.
    
    In the present work, we construct new shallow water approximations of the Dirichlet-Neumann operator at the order of precision $O(\mu \ve)$, $O(\mu\varepsilon + \mu^2\beta^2)$ and $O(\mu^2\varepsilon+ \mu \ve \beta+ \mu^2\beta^2)$. We also quantify the error in the Sobolev spaces uniformly in $\mu, \ve$ and $\beta$. With these approximations, we prove that system \eqref{system mu eps motivations} is consistent with the water waves equations at order $O(\mu\ve)$. Then we derive new weakly-dispersive Boussinesq type systems with the order of precision $O(\mu\varepsilon + \mu^2\beta^2)$, with respect to the water waves equations. In addition, we derive new weakly-dispersive Green-Naghdi type systems with the order of precision $O(\mu^2\varepsilon+ \mu \ve \beta+ \mu^2\beta^2)$. We emphasize the fact that the orders of precision are non-trivial in terms of the bathymetry parameter. Contrary to the models presented in\cite{Chazel07}, the contribution of the bathymetry terms does not contain higher order derivatives when compared to the linear terms. Moreover, they have a similar quasi-linear hyperbolic structure as the full dispersion models in the flat bottom case. We expect then, in light of the recent works of \cite{Benoit_WW_17, Benoit_BP_17}, to be able to prove a long time well-posedness result for these models. This will be an objective for future work.  Lastly, we discuss the derivation of extensions that have a Hamiltonian structure.

    \begin{nots} \color{white} space \color{black}
    \begin{itemize}
	    \item   Let $\mathrm{Id}$ be the $d\times d$ identity matrix, and take $\mathbf{0} = (0,0)^T $ if $d=2$, $\mathbf{0} = 0$ if $d=1$. Then we define the $(d+1)\times(d+1)$ matrix $I^{\mu}$ by
		\begin{equation*}
			I^{\mu} 
			=
			\begin{pmatrix}
				\sqrt{\mu}  \mathrm{Id} &  \mathbf{0} \\
				\mathbf{0}^T  & 1
			\end{pmatrix}.
		\end{equation*}
	
		\item We define the $d$-dimensional Laplace operator by
		\begin{equation*}
			\Delta_X 	=
			\begin{cases}
				\partial_x^2 \quad \hspace{1.7cm} \text{when} \quad d=1
				\\
				\partial_x^2 + \partial_y^2\hspace{1.25cm} \text{when} \quad d=2.
			\end{cases} 
		\end{equation*}
		\item We define the $(d+1)$-dimensional scaled gradient by
		\begin{equation*}
			\nabla^{\mu}_{X,z} = I^{\mu}\nabla_{X,z} 
			=
			\begin{cases}
				(\sqrt{\mu} \partial_x, \partial_z)^T \quad \hspace{1.7cm} \text{when} \quad d=1
				\\
				(\sqrt{\mu} \partial_x,\sqrt{\mu} \partial_y , \partial_z)^T \hspace{1cm} \text{when} \quad d=2,
			\end{cases} 
		\end{equation*}
		and we introduce the scaled Laplace operator
		\begin{equation*}
			\Delta^{\mu}_{X,z} = \nabla^{\mu}_{X,z}\cdot \nabla^{\mu}_{X,z} = \mu \Delta_X + \partial_z^2.
		\end{equation*}
            \item Let $f:\mathbb{R}^d \to \mathbb{R}$ be a tempered distribution, let $\hat{f}$ or $\mathcal{F}f$ be its Fourier transform and $\mathcal{F}^{-1}f$ be its inverse Fourier transform.
            \item For any $s \in \mathbb{R}$ we call the multiplier $\mathrm{J}^s = (1+|\mathrm{D}|^2)^{\frac{s}{2}} = \langle \mathrm{D} \rangle^s$ the Bessel potential of order $-s$. 
            \item The Sobolev space $H^s(\mathbb{R}^d)$  is equivalent to the weighted $L^2-$space with $|f|_{H^s} = |\mathrm{J}^s f|_{L^2}$. 
            \item For any $s\geq 1$ we will denote $\dot{H}^s(\mathbb{R}^d)$ the Beppo-Levi space  with $|f|_{\dot{H}^s} = |\mathrm{J}^{s-1}\nabla_X f|_{L^2}$. 
            \item Let $\Omega \subset \R^{d+1}$. For any $k\in \N$, we define the space $H^{k,0}(\Omega)$ with norm 
            $$\|f\|_{H^{k,0}(\Omega)}^2
             = 
             \sum \limits_{|\gamma|\leq k}\int_{\Omega} |\partial_X^\gamma f(X,z)|^2  \: \mathrm{d}z\mathrm{d}X,$$
            and similarly, for $l \in \N$ such that $l\leq k$, we define the space $H^{k,l}(\Omega)$ with norm $\|f\|_{H^{k,l}(\Omega)} = \sum \limits_{j=0}^l\|\partial_z^j f\|_{H^{k-j,0}(\Omega)}$.
            \item We say that $f$ is a  Schwartz function $\mathscr{S}(\mathbb{R}^d)$, if $f \in C^{\infty}(\mathbb{R}^d)$ and satisfies for all $\alpha,\beta  \in \mathbb{N}^d$,
		\begin{equation*}
			\sup \limits_{X \in \R^d} |X^{\alpha} \partial_X^{\beta} f | < \infty.
		\end{equation*}
		\item If  $A$ and $B$ are two operators, then we denote the commutator between them to be $[A,B] = AB - BA$.
            \item We let $c$ denote a positive constant independent of $\mu, \ve, \beta$ that may change from line to line. Also, as a shorthand, we use the notation $a \lesssim b$ to mean $a \leq c\: b$.
            \item Let $t_0> \frac{d}{2}$, $s\geq 0$, $h_{\min}, h_{b,\min} \in (0,1)$. Then for $\zeta, b,\nabla_X \psi$ sufficiently regular and $C(\cdot)$ a positive, non-decreasing function of its argument, we define the constants
            \begin{align*}
                &M_0  =
                C(\frac{1}{h_{\min}}, \frac{1}{h_{b, \min}}, |\zeta|_{H^{t_0}},|b|_{H^{t_0}})
                \\
                &M(s)  = C(M_0, |\zeta|_{H^{\max\{t_0 +2, s\}}},|b|_{H^{\max\{t_0 +2, s\}}})
                \\
                & N(s)  = C(M(s),|\nabla_X \psi |_{H^{s}} ).
            \end{align*}
        \end{itemize}
    \end{nots}
    
    \subsection{The consistency problem and main results}

    Throughout this paper, $d$  will be the dimension of the horizontal variable, denoted $X \in \R^d$. The reference model of our study is the water waves equations, written under the Zakharov-Craig-Sulem formulation:
    \begin{equation}\label{Water wave equations}
		\begin{cases}
			\partial_t \zeta - \frac{1}{\mu}\mathcal{G}^{\mu}[\varepsilon \zeta, \beta b]\psi = 0
			\\
			\partial_t \psi + \zeta + \frac{\varepsilon}{2}|\nabla_{X} \psi|^2 - \frac{\mu \varepsilon }{2} \frac{(\frac{1}{\mu}\mathcal{G}^{ \mu}[\varepsilon \zeta, \beta b] \psi + \varepsilon \nabla_X \zeta \cdot \nabla_X \psi)^2}{1+ \varepsilon^2  \mu |\nabla_X \zeta|^2} = 0.
		\end{cases}
	\end{equation}
    Here the free surface elevation is the graph of $\zeta(t,X)$, which is a function of time $t$ and horizontal space $X \in \R^d$. The bottom elevation is the graph of $b(X)$, which is a time-independent function. The function $\psi(t,X)$ is the trace at the surface of the velocity potential, and $\mathcal{G}^{\mu}$ is the Dirichlet-to-Neumann operator defined later in Definition \ref{Def Original Laplace}.
    Moreover, every variable and function in \eqref{Water wave equations} is compared with physical characteristic parameters of the same dimension $H_0, a_{\mathrm{surf}}, a_{\mathrm{bott}}$ or $L$.

    Throughout the paper, we  will always make the following fundamental assumption:
    \begin{Def}[Non-cavitation condition] \label{Def: non-cavitation} Let $\ve \in [0,1]$, $\beta \in [0,1]$ and $s >\frac{d}{2}$. Let also $b \in C^{\infty}_c(\R^d)$ be a smooth function with compact support, and take $\zeta \in H^{s}(\R^d)$.  We say $\zeta$ and $b$ satisfies the \lq\lq non-cavitation condition\rq\rq\: if there exists $h_{\text{min}}\in(0,1)$ such that 
		\begin{equation}\label{non-cavitation 1}
				h := 1+\ve\zeta(X) - \beta b(X) \geq h_{\text{min}}, \quad \text{for all} \: \: \: X\in \mathbb{R}^d.
		\end{equation}
    \end{Def}


    \noindent
    Under the non-cavitation condition, we may  define the Dirichlet-Neumann operator by \cite{WWP}:
    
    \begin{Def}\label{Def Original Laplace} Let $t_0>\frac{d}{2}$, $\psi\in \dot{H}^{\frac{3}{2}}(\R^d)$, $b \in C^{\infty}_c(\R^d)$, and $\zeta \in H^{t_0 +1}(\R^d)$ be such that \eqref{non-cavitation 1} is satisfied. Let $\Phi$ be the unique solution in $\dot{H}^2(\Omega_t)$ of the boundary value problem 
    \begin{equation}\label{Laplace pb}
		\begin{cases}
			\Delta^{\mu}_{X,z} \Phi = 0 \hspace{1.15cm}\qquad \text{in} \quad \Omega_t := \{(X,z) \in \R^{d+1}, -1 + \beta b(X) < z < \ve \zeta(X) \}
			\\
			\partial_{n_{ b}} \Phi  = 0 \hspace{2.2cm} \text{on} \quad  z = -1 + \beta  b(X)
			\\
			\Phi = \psi \hspace{2.7cm} \text{on} \quad  z = \varepsilon \zeta(t,X),
		\end{cases}
	\end{equation}	
	where
	\begin{align*}
		\partial_{n_{b}} =  \mathbf{n}_{b} \cdot  I^{\mu} \nabla^{\mu}_{X,z}, \qquad \mathbf{n}_b = \frac{1}{\sqrt{1+\beta^2 |\nabla_X b|^2}} 
		\begin{pmatrix}
			- \beta \nabla_X b 
			\\
			1
		\end{pmatrix}, 
	\end{align*}
	then $\mathcal{G}^{\mu}[\ve \zeta, \beta b]\psi \in H^{\frac{1}{2}}(\R^d)$ is defined by
	\begin{equation}\label{D-N operator}
		\mathcal{G}^{\mu}[\varepsilon \zeta, \beta b]\psi = (\partial_z \Phi 
		- 
		\mu \varepsilon 
		\nabla_X\zeta \cdot \nabla_X \Phi)_{|_{z = \varepsilon \zeta}}.
	\end{equation}
	\end{Def}
	For convenience, it is easier to work with the vertical average of the horizontal component of the velocity. We make the following definition using Proposition $3.35$ in \cite{WWP}.
	\begin{Def}\label{Def V bar} Let $t_0>\frac{d}{2}$, $\psi\in\dot{H}^{ \frac{3}{2}}(\R^{d})$, $b \in C^{\infty}_c(\R^d)$, and $\zeta \in H^{t_0+1}(\R^{d})$ such that \eqref{non-cavitation 1} is satisfied. Let $\Phi\in  \dot{H}^{2}(\Omega_t)$ be the solution of \eqref{Laplace pb}, then we define the operator:
		\begin{equation}\label{V bar}
			\overline{V}^{\mu}[\ve \zeta, \beta b] \psi = \dfrac{1}{h}\int_{-1+\beta b}^{\varepsilon\zeta} \nabla_X\Phi  \: \mathrm{d}z,
		\end{equation}
		and the following relation holds,
		\begin{equation}\label{Relation: D-N op}
			\mathcal{G}^{\mu}[\ve \zeta,\beta b] \psi = - \mu \nabla_X \cdot (h\overline{V}^{\mu}[\ve \zeta, \beta b]\psi).
		\end{equation}
        Throughout this paper, we will denote $\overline{V}^{\mu}[\ve \zeta, \beta b] \psi$ by $\overline{V}$ when no confusion is possible.
	\end{Def}
    In order to write the main results of this paper, we need to define two types of differential operators. The first type is the Fourier multipliers.
    \begin{Def} 
        Let $u:\mathbb{R}^d \to \mathbb{R}^d$ be a tempered distribution, and let $\widehat{u}$ be its Fourier transform. Let $F:\mathbb{R}^d\to\mathbb{R}$ be a smooth function with polynomial decay. Then the Fourier multiplier associated with $F(\xi)$ is denoted $\mathrm{F}(\mathrm{D})$ (denoted $\mathrm{F}$ when no confusion is possible) and defined by the formula:
        \begin{align*}
            \widehat{\mathrm{F}(\mathrm{D})u}(\xi) = F(\xi)\widehat{u}(\xi).
        \end{align*}
    \end{Def}
    \begin{Def}\label{Fourier mult}
       Let $\mathrm{F}_0$ be a Fourier multiplier depending on the transverse variable:
        \begin{equation*}
            \mathrm{F}_0u(X) = \mathcal{F}^{-1}\Big( \frac{\cosh((z+1)\sqrt{\mu}|\xi|)}{\cosh(\sqrt{\mu}|\xi|)} \hat{u}(\xi)\Big)(X),
        \end{equation*}
        for $z\in[-2,0]$.  We also define the four Fourier multipliers $\mathrm{F}_1$, $\mathrm{F}_2$, $\mathrm{F}_3$ and $\mathrm{F}_4$ by the expressions:
        \begin{align*}
            \mathrm{F}_1 = \dfrac{\tanh{(\sqrt{\mu}|\mathrm{D}|)}}{\sqrt{\mu}|\mathrm{D}|},\quad \mathrm{F}_2 = \frac{3}{\mu |\mathrm{D}|^2}(1- \mathrm{F}_1), \quad  \mathrm{F}_3 =\mathrm{sech}(\sqrt{\mu}|D|), \quad \mathrm{F}_4 =  \frac{2}{\mu |\mathrm{D}|^2}(1- \mathrm{F}_3).
        \end{align*}
    \end{Def}
    Next, we would like to define operators of the form 
    \begin{equation}\label{pseudo diff op}
        \mathcal{L}[X,D]u(X) :=  \mathcal{F}^{-1} \big{(}L(X,\xi) \hat{u}(\xi)\big{)}(X),
    \end{equation}
    where $L$ is a smooth function in a particular symbol class given in the next definition. 
    \begin{Def}\label{Pseudo diff}
        Let $d=1,2$ and $m\in\R$. We say $L \in S^m$ is a symbol of order $m$ if $L(X,\xi)$ is $C^{\infty}(\R^d \times \R^d)$ and satisfies
        \begin{equation*}
            \quad\forall \alpha \in \mathbb{N}^d, \quad \forall \gamma \in \mathbb{N}^d, \quad \langle \xi \rangle^{-(m -|\gamma|)}|\partial_X^{\alpha} \partial_{\xi}^{\gamma} L(X,\xi)| < \infty.
        \end{equation*}
        We also introduce the seminorm
        \begin{equation}\label{Semi norm}
            \mathcal{M}_m(L)  
            =
            \sup\limits_{|\alpha|\leq \lceil \frac{d}{2} \rceil + 1}\sup \limits_{|\gamma|\leq \lceil \frac{d}{2} \rceil + 1}   \sup \limits_{(X,\xi)\in \R^d\times \R^d } \Big{\{} \langle \xi \rangle^{-(m -|\gamma|)}|\partial_X^{\alpha} \partial_{\xi}^{\gamma} L(X,\xi)| \Big{\}}.
        \end{equation}
    \end{Def}
    The next result allows us to justify the formula \eqref{pseudo diff op} for functions $u$ in Sobolev spaces. 
    \begin{thm}\label{C-V thm}
        Let $d = 1,2$, $s \geq 0$, and $L \in S^m$. Then formula \eqref{pseudo diff op} defines a bounded pseudo-differential operator of order $m$ from $H^{s+m}(\R^d)$ to $H^{s}(\R)$ and satisfies
        \begin{equation}\label{est Stein}
            |\mathcal{L}[X,D]u|_{H^{s}} \leq \mathcal{M}_{m}(L)  |u|_{H^{s+m}}.
        \end{equation}
    \end{thm}
    \noindent
    We refer to  \cite{AlinhacGerard07} for this result, where the constant is given implicitly in the proof (see also \cite{Metivier2008, Alazard21}). We will define operators of interest under the assumption:
    \begin{AssDef} Let $d=1,2$ and $\beta \in [0,1]$. Throughout this paper, we will always assume that the bathymetry $\beta b \in C^{\infty}_c(\R^d)$ satisfies the following: There exists $b_{\mathrm{max}}  \in (0,1)$ such that
    \begin{equation}\label{non-cavitation 2}
        \beta |b(X)| \leq  b_{\mathrm{max}}  < 1, \quad \text{for all} \: \: \: X \in \R^d.
    \end{equation}
    We also define the water depth at the rest state $h_b := 1- \beta b(X)$. As a consequence of \eqref{non-cavitation 2}, there exists a constant $h_{b,\min} \in (0,1)$ such that 
    \begin{align}\label{non cav 3}
        0<h_{b,\min}\leq h_b.
    \end{align}
    \end{AssDef}
    \noindent
    Through \eqref{non cav 3}, we suppose the bottom topography is submerged under the still water level. We may now define the pseudo-differential operators that will play an important role in deriving new models that allow for large amplitude topography variations. 
    \begin{DefProp}\label{Prop op}
        Let $\mu,\beta \in [0,1]$, $d=1,2$,  $s \geq 0$ and $b \in C^{\infty}_c(\R^d)$ such that \eqref{non-cavitation 2} is satisfied. We define the following pseudo-differential operators of order zero, bounded uniformly with respect to $\mu$ and $\beta$ in $H^s(\R^d)$:
        \begin{align*} 
            \mathcal{L}_{1}^{\mu}[\beta b]  
            & =
            -\frac{1}{\beta}\sinh{(\beta b(X) \sqrt{\mu}|\mathrm{D}|)}\mathrm{sech}(\sqrt{\mu}|\mathrm{D}|)\dfrac{1}{\sqrt{\mu}|\mathrm{D}|}
            \\
            \mathcal{L}_{2}^{\mu}[\beta b]  
            & = -(\mathcal{L}_{1}^{\mu}[\beta b] + b) \frac{1}{\mu |\mathrm{D}|^2}
            \\
            \mathcal{L}_{3}^{\mu}[\beta b]  & = 
            - \big{(}\cosh(\beta b(X)\sqrt{\mu}|\mathrm{D}|)\mathrm{sech}(\sqrt{\mu}|\mathrm{D}|)-1\big{)}\frac{1}{\mu |\mathrm{D}|^2}.
        \end{align*}
        Moreover, for $u\in \mathscr{S}(\R^d)$ we have the following estimates
        \begin{align}
            \label{L1 est}
            |\mathcal{L}_1^{\mu}[\beta b] u|_{H^s} & \leq M(s)  |u|_{H^{s}}.
            \\ \label{id op 3}
            |\mathcal{L}_2^{\mu}[\beta b] u|_{H^s} & \leq M(s) |u|_{H^{s}}
            \\
            \label{L3}
            |\mathcal{L}_3^{\mu}[\beta b] u|_{H^s} & \leq M(s) |u|_{H^{s}}
            \\
            \label{L approx}
            |\mathcal{L}_1^{\mu}[\beta b] u +  b u|_{H^s} & \leq \mu M(s) |u|_{H^{s+2}}
            %
            %
            %
            %
            %
            %
            %
            %
            %
            \\
            \label{L1 approx next order}
            |\mathcal{L}_1^{\mu}[\beta b] u -  (-b  - \frac{\mu\beta^2}{6}b^3|\mathrm{D}|^2)\mathrm{F}_3 u|_{H^s} & \leq \mu^2\beta^4 M(s) |u|_{H^{s+4}}
            \\
            \label{L2 next order}
            |\mathcal{L}_2^{\mu}[\beta b] u -  (-\frac{1}{2}b\mathrm{F}_4  + \frac{\beta^2}{6}b^3\mathrm{F}_3) u|_{H^s} & \leq \mu\beta^4  M(s) |u|_{H^{s+2}}.
        \end{align} 
        
    \end{DefProp}

    \begin{remark}
        Under assumption \eqref{non-cavitation 2} the operators $\mathcal{L}_1^{\mu}$, $\mathcal{L}_2^{\mu}$, and $\mathcal{L}_3^{\mu}$ are \lq\lq classical pseudo-differential operators of order zero\rq\rq. We will share the details of the proof in Appendix A, Subsection \ref{A1}. 
    \end{remark} 
    \begin{prop}\label{Prop G[0,b]} Let $t_0>\frac{d}{2}$, $\psi\in\dot{H}^{\frac{3}{2}}(\R^{d})$, $b \in C^{\infty}_c(\R^d)$, and $\zeta \in H^{t_0+1}(\R^{d})$ such that \eqref{non-cavitation 1} is satisfied. Let $\phi\in  \dot{H}^{2}(\R^d \times [-1+\beta b, 0])$ be the solution of
    \begin{align}\label{phi pb}
            \begin{cases}
                \Delta_{X,z}^{\mu} \phi = 0 \ \ \mathrm{in} \ \ \mathcal{S}_b, \\
                \phi|_{z=0} = \psi, \ \ \big[\partial_z\phi - \mu\beta \nabla_X b \cdot \nabla_X \phi \big]\big|_{z=-1+\beta b}  = 0.
            \end{cases}
    \end{align}
    Then we can define
    \begin{equation}\label{def G[0,b]}
            \frac{1}{\mu}\mathcal{G}_b \psi = -\nabla_{X} \cdot \Big{(} \frac{h}{h_b} \int_{-h_b}^0 \nabla_X \phi \: \mathrm{d}z\Big{)}.
        \end{equation}
        Moreover, for $\psi \in \dot{H}^{s+5}(\R^d)$ and $\zeta \in H^{\max\{t_0+2,s+3\}}(\R^d)$ we have the estimate
        \begin{equation*}\label{est G[0,b]}
            \dfrac{1}{\mu}|\mathcal{G}^{\mu}\psi 
         - \mathcal{G}_b \psi  |_{H^{s}} \leq \mu\ve M(s+3)|\nabla_X \psi|_{H^{s+4}}.
        \end{equation*}

    \end{prop}

    \begin{remark}
        The operator $\mathcal{G}_b$ contains terms of order $\ve \zeta$ and is different from $\mathcal{G}^{\mu}[0,\beta b]$ defined by \eqref{Relation: D-N op}. To be precise, we can relate the two operators by expanding $\frac{h}{h_b}$:
        \begin{equation*}
             \frac{1}{\mu}\mathcal{G}_b \psi = \frac{1}{\mu}\mathcal{G}^{\mu}[0,\beta b]\psi 
             - \ve \nabla_{X} \cdot \Big{(} \frac{\zeta}{h_b}\int_{-h_b}^0 \nabla_X \phi \: \mathrm{d}z\Big{)}.
        \end{equation*}
    \end{remark}

    \begin{prop}\label{Prop G}  
    Let $d=1,2$, $t_0>\frac{d}{2}$ and $s\geq 0$.  Also let $b \in C^{\infty}_c(\R^d)$ and $\zeta \in H^{\max\{t_0+2,s+3\}}(\R^d)$ such that \eqref{non-cavitation 1} and \eqref{non-cavitation 2} are satisfied.    From the previously defined operators, we have the following approximations of the Dirichlet-Neumann operator:
    \begin{align*}
        \frac{1}{\mu}\mathcal{G}_0\psi & = 
        -
        \mathrm{F}_1\Delta_X \psi  
        -\beta(1+
        \frac{\mu}{2}\mathrm{F}_4 \Delta_X) 
        \nabla_X\cdot \big(\mathcal{L}_{1}^{\mu}[\beta b]\nabla_X \psi \big) 
        -
        \varepsilon \nabla_X\cdot \big( \zeta\mathrm{F}_1 \nabla_X\psi \big)
        \\
        &
        \hspace{0.5cm}
        +
        \frac{\mu  \beta^2}{2}\nabla_X\cdot \big(\mathcal{B}[\beta b] \nabla_X \psi\big),
    \end{align*}
    and
    \begin{align*}
        \frac{1}{\mu}\mathcal{G}_1 \psi & = -\nabla_X\cdot (h \nabla_X \psi) 
         -
         \frac{\mu}{3} \Delta_X\Big{(}\frac{h^3}{h_b^3}\mathrm{F}_2 \Delta_X \psi \Big{)} 
         -
         \mu \beta \Delta_X \big{(}   \mathcal{L}_{2}^{\mu}[\beta b]\Delta_X \psi\big{)}
         \\ 
         & 
         \hspace{0.5cm}
         -
        \frac{\mu \beta}{2} \mathrm{F}_4 \Delta_X \nabla_X \cdot \big(\mathcal{L}_1^{\mu}[\beta b] \nabla_X \psi\big)
        +
        \frac{\mu  \beta^2}{2} \nabla_X\cdot\big(\mathcal{B}[\beta b] \nabla_X \psi\big),
    \end{align*}
    where
    \begin{align}\label{B bathymetry}
        \mathcal{B}[\beta b]\nabla_X \psi &= 
            b \mathrm{F}_4\nabla_X(\nabla_X\cdot(b\nabla_X \psi))
            \\ 
            & \hspace{0.5cm} \notag
            +
            h_b
             \nabla_X \big{(}b\mathrm{F}_4\nabla_X\cdot(b\nabla_X \psi)\big{)}
             +
             2h_b(\nabla_Xb)\mathrm{F}_1\nabla_X\cdot(b\nabla_X \psi).
    \end{align}
    Moreover, we have the following estimates on the Dirichlet-Neumann operator 
    \begin{align}\label{G0}
         & \dfrac{1}{\mu}|\mathcal{G}^{\mu}\psi 
         - \mathcal{G}_0\psi  |_{H^{s}} \leq (\mu\ve + \mu^2 \beta^2) M(s+3)|\nabla_X \psi|_{H^{s+5}}
         \\\label{G1}
         & \dfrac{1}{\mu}|\mathcal{G}^{\mu}\psi 
         - \mathcal{G}_1\psi |_{H^{s}} \leq (\mu^2\ve +\mu \ve \beta + \mu^2 \beta^2) M(s+3)|\nabla_X \psi|_{H^{s+5}}.
    \end{align}
    \end{prop}

    Proposition \ref{Prop G} is the key result from which we will derive our new models. However, before presenting these models, we need to define the notion of consistency of the water waves equations \eqref{Water wave equations} with a given asymptotic model.
    \begin{Def}[Consistency]\label{Consistency}
    Let $\mu,\ve,\beta \in [0,1]$. We denote by $(\mathrm{A})$ an asymptotic model of the following form:
    \begin{align*}
        (\mathrm{A}) \ \ \begin{cases}
                    \partial_t \zeta + \mathcal{N}_1(\zeta,b,\psi) = 0\\
                    \partial_t (\mathcal{T}[\zeta, b]\psi)
                    + \mathcal{N}_2(\zeta,b,\psi) = 0,
                 \end{cases}
    \end{align*}
    where $\mathcal{T}$ is a linear operator with respect to $\psi$ and possibly nonlinear with respect to $\zeta$ and $b$. While $\mathcal{N}_1$ and $\mathcal{N}_2$ are possibly nonlinear operators. 
    
    We say that the water waves equations are consistent at order $O(\sum \mu^k \varepsilon ^l\beta^m)$ with $(\mathrm{A})$ if there exists $n\in \N$ and a universal constant $T >0$ such that for any $s\geq0$ and every solution $(\zeta,\psi) \in C([0,\frac{T}{ \varepsilon }]; H^{s+n}(\mathbb{R}^d)\times \dot{H}^{s+n}(\mathbb{R}^d))$ to the water waves equations \eqref{Water wave equations}, one has for all $t \in [0,\frac{T}{\varepsilon}]$,
    \begin{equation*} 
                \begin{cases}
                    \partial_t \zeta + \mathcal{N}_1(\zeta,b,\psi) = \big(\sum \mu^k \varepsilon ^l\beta^m\big) R_1\\
                    \partial_t (\mathcal{T}[\zeta, b]\psi)+ \mathcal{N}_2(\zeta,b,\psi) = \big(\sum \mu^k \varepsilon ^l\beta^m\big) R_2,
                \end{cases}
        \end{equation*}
        where $|R_i|_{H^s}\leq N(s+n)$ for all $t\in [0,\frac{T}{ \varepsilon }]$ with $i=1,2$. 
    \end{Def}  
    We should note that the existence time for solutions of the water waves equations is proved to be on the scale $O(\frac{1}{\max\{\ve,\beta\}})$ and uniformly with respect to $\mu$ (see \cite{Alvarez-SamaniegoLannes08b}). However, it was proved that when one includes surface tension with a strength of the same order as the shallow water parameter $\mu$, then the time existence is improved and becomes of order $O(\frac{1}{\ve})$ \cite{Benoit_WW_17}. For the sake of clarity, we will omit the surface tension in this paper. But one could easily add it to every model of this work without changing the results. With this in mind, we may now state our consistency results. 
     \begin{thm}\label{thm Whitham Boussinesq mu ve}
        Let $\mathcal{G}_b$ be defined by \eqref{def G[0,b]}. Then for any $\mu \in (0, 1]$, $\ve \in [0,1]$, and $\beta\in [0,1]$ the water waves equations \eqref{Water wave equations} are consistent, in the sense of Definition \ref{Consistency} with $n= 5$, at order $O(\mu\ve)$ with the Boussinesq-type system:
        \begin{align}\label{Whitham boussinesq mu ve}
            \begin{cases}
                \partial_t \zeta - \frac{1}{\mu}\mathcal{G}_b \psi= 0\\
                \partial_t \psi + \zeta + \dfrac{\varepsilon}{2}|\nabla_X \psi|^2 = 0,
            \end{cases}
        \end{align}
    \end{thm}

    \begin{remark} The system is a simplified version of the water waves equations where $\mathcal{G}_b$ is defined implicitly in terms of the solution of \eqref{phi pb}. The elliptic problem is  defined on the fixed domain $\mathcal{S}_b = \R^d \times [-1+\beta b, 0]$, but depends on the Dirichlet data $\psi$ which in turn depend on $\zeta$ through \eqref{Whitham boussinesq mu ve}. 
    \end{remark}

    \begin{thm}\label{thm Whitham Boussinesq}
        Let $\mathrm{F}_1$ and $\mathrm{F}_4$ be the two Fourier multipliers given in Definition \ref{Fourier mult}, and let $\mathcal{L}_1^{\mu}$ be given in Definition \ref{Prop op}. Then for any $\mu \in (0, 1]$, $\ve \in [0,1]$, and $\beta\in [0,1]$ the water waves equations \eqref{Water wave equations} are consistent, in the sense of Definition \ref{Consistency} with $n=6$, at order $O(\mu\ve + \mu^2 \beta^2)$ with the Boussinesq type system:
        \begin{align}\label{Whitham boussinesq}
            \begin{cases}
                \partial_t \zeta + \mathrm{F}_1\Delta_X \psi +  \beta(1     +
                \frac{\mu }{2} \mathrm{F}_4\Delta_X) \nabla_X\cdot ( \mathcal{L}_{1}^{\mu}[\beta b]\nabla_X \psi)  \\
                \hspace{4cm} + \varepsilon \mathrm{G}_1 \nabla_X\cdot ( \zeta \mathrm{G}_2\nabla_X\psi) - \frac{\mu \beta^2}{2} \nabla_X\cdot(\mathcal{B}[\beta b] \nabla_X \psi\big) = 0 \\
                \partial_t \psi + \zeta + \dfrac{\varepsilon}{2}(\mathrm{G}_1\nabla_X \psi)\cdot (\mathrm{G}_2\nabla_X \psi) = 0,
            \end{cases}
        \end{align}
        where 
        \begin{align*}
           \mathcal{B}[\beta b]\nabla_X \psi = 
            b \mathrm{F}_4\nabla_X(\nabla_X\cdot(b\nabla_X \psi))
            +
            h_b
             \nabla_X \big{(}b\mathrm{F}_4\nabla_X\cdot(b\nabla_X \psi)\big{)}
             +
             2h_b(\nabla_Xb)\mathrm{F}_1\nabla_X\cdot(b\nabla_X \psi),
        \end{align*}
        and $\mathrm{G}_1, \mathrm{G}_2$ are any Fourier multipliers such that for any $s \geq 0$ and $u \in H^{s+2}(\mathbb{R}^d)$, we have
        \begin{align}\label{G_j}
            |(\mathrm{G}_j - 1)u|_{H^s} \lesssim \mu|u|_{H^{s+2}}.
        \end{align}
    \end{thm}
    \noindent

    \begin{remark}
    \noindent
        \begin{itemize}
            \item Taking $\beta = 0$ in \eqref{Whitham boussinesq}, we get the class of full dispersion Boussinesq systems derived rigorously in \cite{Emerald21} with a precision $O(\mu \ve)$. These systems were rigorously justified on a time scale of order $O(\frac{1}{\ve})$ under additional decrease constraint on the Fourier multipliers $\mathrm{G}_1$ and $\mathrm{G_2}$ (see \cite{Emerald22} for more information).

             \item In the case $\mathrm{G}_1 = \mathrm{G}_2 = \mathrm{Id}$, \eqref{Whitham boussinesq} is believed to be ill-posed \cite{KleinLinaresPilodEtAl18} in the case $\beta=0$, unless one includes surface tension \cite{Paulsen22}. Alternatively, one can exploit the regularizing effect provided by the multipliers $\mathrm{G}_j$ \cite{Emerald22}.

             \item One could also add Fourier multipliers $\mathrm{G}_j$, defined by \eqref{G_j}, in the term of order $O(\mu \beta^2)$ without changing the precision of the model.

            \item Neglecting terms of order $O(\mu \ve + \mu\beta)$ and approximating $\mathcal{L}_1^{\mu}[\beta b]$ with estimate \eqref{L approx}, we arrive at the same models derived in \cite{DucheneMMWW21}.
        \end{itemize}
    \end{remark} 
    One can replace the pseudo-differential operator in \eqref{Whitham boussinesq} with estimate \eqref{L1 approx next order}. Indeed, we have the following result: 
    \begin{cor}
        Under the same assumptions as in Theorem \ref{thm Whitham Boussinesq}, we can take
        \begin{equation*}
                \mathcal{L}_1^{\mu}[\beta b]\bullet = -(b + \frac{\mu \beta^2}{6}b^3|\mathrm{D}|^2)\mathrm{F}_3\bullet,
            \end{equation*}
            in system \eqref{Whitham boussinesq} and keep the precision $O(\mu\varepsilon + \mu^2\beta^2)$.
    \end{cor}

    We also derive a Boussinesq-type system in the variables $(\zeta, \overline{V})$:

    \begin{thm}\label{thm Whitham Boussinesq V_bar}
        Let $\mathrm{F}_1$ and $\mathrm{F}_4$ be the two Fourier multipliers given in Definition \ref{Fourier mult}, and let $\mathcal{L}_1^{\mu}$ be given in Definition \ref{Prop op}. Then for any $\mu \in (0, 1]$, $\ve \in [0,1]$, and $\beta\in [0,1]$ the water waves equations \eqref{Water wave equations} are consistent, in the sense of Definition \ref{Consistency} with $n=7$, at order $O(\mu\ve + \mu^2 \beta^2)$ with the Boussinesq type system:
        \begin{align}\label{Whitham boussinesq V bar}
            \begin{cases}
                \partial_t \zeta + \nabla_X\cdot(h \overline{V}) = 0\\
                \partial_t \overline{V} + \mathcal{T}_0^{\mu}[\beta b, \ve \zeta] \nabla_X \zeta 
                +
                \frac{\ve}{2} \nabla_X|\overline{V}|^2
                = \mathbf{0},
            \end{cases}
        \end{align}
        where 
        \begin{align*}
            \mathcal{T}_0^{\mu}[\beta b, \ve \zeta] \bullet 
            & =
            \frac{1}{h}\Big{(}\mathrm{F}_1 \bullet  +  \beta \mathcal{L}_1^{\mu}[\beta b]\bullet  + \ve \zeta \mathrm{F}_1\bullet \Big{)}
            +
            \frac{\mu \beta}{2} \nabla_X\mathrm{F}_4\nabla_X \cdot \big(\mathcal{L}_1^{\mu}[\beta b] \bullet \big)
            \\ 
            & 
            \hspace{0.5cm}\notag
            -
            \frac{\mu\beta^2}{2}\nabla_X \big{(}b\mathrm{F}_4\nabla_X\cdot(b\bullet )\big{)}
            -
            \mu \beta^2 (\nabla_Xb)\mathrm{F}_1\nabla_X\cdot(b\bullet).
        \end{align*}
    \end{thm}

    \begin{remark}
        \noindent
        \begin{itemize}
            \item The first equation in \eqref{Whitham boussinesq V bar} is exact and is a formulation of the conservation of mass.
            
            \item Taking $\beta = 0$ in \eqref{Whitham boussinesq V bar}, we get the class of full dispersion Boussinesq systems derived rigorously in \cite{Paulsen22} with a precision $O(\mu \ve)$.

             \item One could also add Fourier multipliers $\mathrm{G}_j$, defined by \eqref{G_j}, in the term of order $O(\mu \beta^2)$ without changing the precision of the model.
        \end{itemize}
    \end{remark}

    \begin{cor}
        Under the same assumptions as in Theorem \ref{thm Whitham Boussinesq V_bar}, we can take
        \begin{align*}
            \mathcal{T}_0^{\mu}[\beta b, \ve \zeta] \bullet 
            & =
            \mathrm{F}_1 \bullet  
            + 
            \frac{\beta b}{h} (\mathrm{F}_1-\mathrm{F}_3)\bullet
            +
            \frac{\mu \beta^3}{6h}b^3|\mathrm{D}|^2\mathrm{F}_3\bullet
            -
            \frac{\mu \beta}{2} \nabla_X\mathrm{F}_4\nabla_X \cdot (b \bullet )
            \\ 
            & 
            \hspace{0.5cm}\notag
            -
            \frac{\mu\beta^2}{2}\nabla_X \big{(}b\mathrm{F}_4\nabla_X\cdot(b\bullet )\big{)}
            - 
            \mu \beta^2 (\nabla_Xb)\mathrm{F}_1\nabla_X\cdot(b\bullet),
            \end{align*}
            in system \eqref{Whitham boussinesq V bar} and keep the precision $O(\mu\varepsilon + \mu^2\beta^2)$.
    \end{cor}
    
    \noindent

    The next two results concern full dispersion Green-Naghdi systems.  

    \begin{thm}\label{Thm good WGN}
         Let $\mathrm{F}_2$ and $\mathrm{F}_4$ be the two Fourier multipliers given in Definition \ref{Fourier mult}, and let $\mathcal{L}_2^{\mu}$ be given in Definition \ref{Prop op}. Then for any $\mu \in (0, 1]$, $\ve \in [0,1]$, and $\beta \in [0,1]$ the water waves equations \eqref{Water wave equations} are consistent, in the sense of Definition \ref{Consistency} with $n=6$, at order $O(\mu^2\ve + \mu\ve \beta + \mu^2 \beta^2)$ with the Green-Naghdi type system:
        \begin{align}\label{Good-Whitham-Green-Naghdi intro}
            \begin{cases}
                \partial_t \zeta 
                +
                \nabla_X \cdot (h\mathcal{T}_1^{\mu}[\beta b, \ve \zeta]\nabla_X \psi)
                -  \frac{\mu \beta^2}{2} \nabla_X\cdot\big(\mathcal{B}[\beta b]  \nabla_X \psi\big) = 0
                \\
                \partial_t \psi + \zeta + \dfrac{\varepsilon}{2}|\nabla_X\psi|^2 - \dfrac{\mu\varepsilon}{2}h^2(\sqrt{\mathrm{F}_2}\Delta_X\psi)^2 = 0,
            \end{cases}
    \end{align}
    where 
    \begin{align*}
            \mathcal{B}[\beta b]\bullet &= 
            b \mathrm{F}_4\nabla_X(\nabla_X\cdot(b\bullet))
            +
            h_b
             \nabla_X \big{(}b\mathrm{F}_4\nabla_X\cdot(b\bullet)\big{)}
             +
             2h_b(\nabla_Xb)\mathrm{F}_1\nabla_X\cdot(b\bullet),
        \end{align*}
    and
    \begin{align*}
        \mathcal{T}_1^{\mu}[\beta b, \ve \zeta]\bullet
               &  =
                \mathrm{Id}
                +
                \dfrac{\mu}{3h} \nabla_X\sqrt{\mathrm{F}_2} \Big(\dfrac{h^3}{h_b^3}\sqrt{\mathrm{F}_2}\nabla_X \cdot \bullet  \Big) 
                +
                \frac{\mu \beta }{h}\nabla_X \Big{(}  \mathcal{L}_{2}^{\mu}[\beta b]\nabla_X \cdot \bullet \Big{)} 
                \\ 
                & 
                \hspace{0.5cm}
                +
                \frac{\mu \beta }{2h}\mathrm{F}_4\nabla_X \nabla_X \cdot \big(\mathcal{L}_1^{\mu}[\beta b] \bullet\big),
    \end{align*}
    and $\sqrt{\mathrm{F}_2}$ is the square root of  $\mathrm{F}_2$.
    \end{thm}
    \begin{remark}
    \noindent 
    
        \begin{itemize}
            \item System \eqref{Good-Whitham-Green-Naghdi intro} was first derived in \cite{Emerald21} in the case $\beta = 0$.
            \item In \cite{DucheneMMWW21}, the author derived a weakly dispersive Green-Naghdi type system with an order of precision given by $O(\mu^2 \ve +  \mu^2\beta)$.
            
            \item One could also add Fourier multipliers $\mathrm{G}_j$, defined by \eqref{G_j}, in the term of order $O(\mu \beta^2)$ without changing the precision of the model.

        \end{itemize}
    \end{remark}
    Again, we can simplify the system using Proposition \ref{Prop op} to obtain a system only depending on Fourier multipliers. 

    \begin{cor}
            Under the same assumptions as in Theorem \ref{Thm good WGN}, we can take
            \begin{align*}
                \mathcal{L}_1^{\mu}[\beta b]\bullet = -b \mathrm{F}_3\bullet,
            \end{align*}    
            and 
            \begin{align*}
                \mathcal{L}_2^{\mu}[\beta b]\bullet =-\frac{1}{2}b \mathrm{F}_4\bullet+\frac{\beta^2}{6}b^3\mathrm{F}_3\bullet,
            \end{align*}
            in system \eqref{Good-Whitham-Green-Naghdi intro}  keeping the precision $O(\mu^2\varepsilon + \mu \ve \beta + \mu^2\beta^2)$.
    \end{cor}
    
     Several generalizations can now be made, where the next system is chosen to mimic some of the properties of the classical Green-Naghdi systems:
    \begin{thm}\label{Thm WGN Vbar}
        Let $\mathrm{F}_2$ and $\mathrm{F}_4$ be the two the Fourier multipliers given in Definition \ref{Fourier mult}, let $\mathcal{L}_1^{\mu}$ and $\mathcal{L}_2^{\mu}$ be given in Definition \ref{Prop op}. Then for any $\mu \in (0, 1]$, $\ve \in [0,1]$, and $\beta \in [0,1]$ the water waves equations \eqref{Water wave equations} are consistent, in the sense of Definition \ref{Consistency} with $n=7$, at order $O(\mu^2 \ve + \mu \ve \beta + \mu^2 \beta^2)$ with the Green-Naghdi type system: 
        \begin{align}\label{Whitham-Green-Naghdi_Vbar}
        \begin{cases}
            \partial_t \zeta + \nabla_X \cdot(h\overline{V}) = 0, \\
            \partial_t(\mathcal{I}^{\mu}[h]\overline{V}) + \mathcal{I}^{\mu}[h]\mathcal{T}_2^{\mu}[\beta b, h]\nabla_X\zeta + \frac{\ve}{2}  \nabla_X \big( |\overline{V}|^2\big) + \mu\ve \nabla_X\mathcal{R}_1^{\mu}[\beta b, h, \overline{V}] = \mathbf{0},
        \end{cases}
        \end{align}
	where $\overline{V}$ defined by \eqref{V bar},
    \begin{equation*}
		\mathcal{I}^{\mu}[h]\bullet=  \mathrm{Id} - \frac{\mu}{3h}\sqrt{\mathrm{F}_2} \nabla_X \Big(h^3  \sqrt{\mathrm{F}_2} \nabla_X\cdot \bullet  \Big),
	\end{equation*}
	\begin{align*}
		\mathcal{T}_2^{\mu}[\beta b, \ve \zeta]\bullet
            & =
            \mathrm{Id} + \frac{\mu}{3h}\sqrt{\mathrm{F}_2} \nabla_X \Big(\frac{h^3}{h_b^3}  \sqrt{\mathrm{F}_2} \nabla_X\cdot \bullet  \Big) 
            +
            \frac{\mu\beta }{h} \nabla_X \Big(\mathcal{L}_2^{\mu}[\beta b] \nabla_X\cdot \bullet \Big)
            \\ 
            & 
            \hspace{0.3cm}
            +
            \frac{\mu \beta h_b}{2h}\nabla_X \mathrm{F}_4 \nabla_X \cdot \big(\mathcal{L}_1^{\mu}[\beta b] \bullet\big)
            -
            \frac{\mu\beta^2h_b}{2h}\nabla_X \big{(}b\mathrm{F}_4\nabla_X\cdot(b\bullet)\big{)}
            \\
            &
            \hspace{0.3cm}
            - 
            \frac{\mu \beta^2 h_b}{h} (\nabla_Xb)\mathrm{F}_1\nabla_X\cdot(b\bullet),
	\end{align*}
	and
	\begin{align*}
		\mathcal{R}_1^{\mu}[\beta b, h, \overline{V}] 
            =
            &- \frac{h^2}{2}  (\nabla_X \cdot \overline{V})^2
            -
            \frac{1}{3h}\big{(}\nabla_X(h^3 \nabla_X \cdot \overline{V})\big{)} \cdot \overline{V}
            - \frac{1}{2}h^3\Delta_X(|\overline{V}|^2) 
            +
            \frac{1}{6h} h^3 \Delta_X (|\overline{V}|^2).
	\end{align*}
        
    \end{thm}
    \begin{remark}
    \noindent
        \begin{itemize}
            \item As for  the classical Green-Naghdi system, we observe that the first equation is a formulation of mass conservation.
            \item The system depends on the elliptic operator $h\mathcal{I}^{\mu}[h]$ and is similar to the systems derived in \cite{WWP, Emerald21, DucheneMMWW21} in that sense.
            \item The presence of the term $\mathcal{I}^{\mu}[h]\mathcal{T}_2^{\mu}[\beta b, h]\nabla_X\zeta$ in the second equation makes it quite unique. Note that one may simplify it, but we chose to keep it under this form because in the study of the local well-posedness theory, one would apply the inverse of the elliptic operator $h\mathcal{I}^{\mu}[h]$ to the equation.

        \end{itemize}  
    \end{remark}

    \begin{cor}
            Under the same assumptions as in Theorem \ref{Thm WGN Vbar}, we can take
            \begin{align*}
                \mathcal{L}_1^{\mu}[\beta b]\bullet = -b \mathrm{F}_3\bullet,
            \end{align*}    
            and 
            \begin{align*}
                \mathcal{L}_2^{\mu}[\beta b]\bullet =-\frac{1}{2}b \mathrm{F}_4\bullet+\frac{\beta^2}{6}b^3\mathrm{F}_3\bullet,
            \end{align*}
            in system \eqref{Whitham-Green-Naghdi_Vbar} keeping the precision $O(\mu^2\varepsilon + \mu \ve \beta + \mu^2\beta^2)$.
    \end{cor}

        \subsection{Outline} The paper is organized as follows. In Section \ref{Asymptitic exp}, we set out to prove Proposition \ref{Prop G}. First, we start Subsection \ref{Trans pb} by transforming the elliptic problem \eqref{Laplace pb} so that its domain is time-independent. Then we use this new formulation to perform multi-scale expansions. In particular, in Subsection \ref{mult general pb} and \ref{mult phi}, we make several expansions of the velocity potential in terms of $\mu$, $\ve$ and $\beta$. From these expansions, we approximate the vertically averaged velocity potential $\overline{V}$ in Subsection \ref{mult V bar}, from which the proof of Proposition \ref{Prop G} is deduced in Subsection \ref{mult G}.  
        Section \ref{WH B} is dedicated to the proofs of Theorem \ref{thm Whitham Boussinesq} and Theorem \ref{thm Whitham Boussinesq V_bar}. We also formally derive a Hamiltonian Boussinesq type system. In  Section \ref{Derivation Whitham-Green-Naghdi systems with bathymetry} we prove Theorem \ref{Thm good WGN} and Theorem \ref{Thm WGN Vbar}. Lastly, the appendix is composed of three subsections. The Subsection \ref{A1} is dedicated to the proof of Proposition \ref{Prop op}. In the last two Subsections \ref{A2} and \ref{A3}, we state and prove technical tools.

	\section{Asymptotic expansions of the Dirichlet-Neumann operator}\label{Asymptitic exp}
	 
	In this section, we perform expansions of the Dirichlet-Neumann operator with an error of order $O(\mu \ve +  \mu^2 \beta^2)$ and $O(\mu^2\ve+ \mu \ve \beta + \mu^2\beta^2)$. 	The standard approach to deriving asymptotic models is by approximating the velocity potential $\Phi$, which in turn will give an approximation of \eqref{D-N operator}. Classically, one straightens the fluid domain to work on the flat strip, where we can easily make approximations. However, if we straighten the bottom, there will be an appearance of $\beta$ that will give approximations on the form $O(\mu\ve + \mu \beta)$ in the case Boussinesq type systems and $O(\mu^2\ve + \mu^2\beta)$ in the case Green-Naghdi type systems (see \cite{DucheneMMWW21} for the derivation of such models).

	\subsection{The transformed Laplace equation}\label{Trans pb}
	
	Motivated by the previous discussion, we make a change of variable that only straightens the top of the fluid domain.
	
	\begin{Def}\label{Map to semi-strip} 
		Let $s > \frac{d}{2}+1$, $b \in C^{\infty}_c(\R^d)$, and $\zeta \in H^{s}(\mathbb{R}^d)$  such that the non-cavitation assumptions \eqref{non-cavitation 1} and \eqref{non-cavitation 2} are satisfied. We define the time-dependent diffeomorphism mapping the domain 
		$$\mathcal{S}_b := \{ (X,z) \in \mathbb{R}^{d+1} : -1+\beta b \leq z \leq 0 \},$$
		onto the water domain $\Omega_t$ through
		\begin{equation*}\label{TrivialDiffeomorphismEq}
			\Sigma_b 
			: \: 
			\begin{cases}
				\mathcal{S}_b \hspace{0.6cm} \longrightarrow & \Omega_t \\
				(X,z) \: \:  \mapsto & (X, z + \sigma(X,z))
			\end{cases}
		\end{equation*}
		with
		\begin{equation}\label{sigma}
			\sigma(X,z) = \dfrac{\varepsilon\zeta(X)}{1-\beta b(X)}z + \varepsilon\zeta(X).
		\end{equation}
	\end{Def}

	\begin{remark} 
		The map given in Definition \ref{Map to semi-strip} is a diffeomorphism. Indeed, by computing the Jacobi matrix, we find that
		\begin{align*}
			J_{\Sigma_b} 
			=
			\begin{pmatrix}
				\mathrm{Id} & \mathbf{0}
				\\
				(\nabla_X \sigma)^T &1+ \partial_z \sigma
			\end{pmatrix},
		\end{align*}
		where
		\begin{equation*}
			1+ \partial_z\sigma = \frac{h}{h_b}.
		\end{equation*}
		Therefore, under the non-cavitation condition as stated in Definition \ref{Def: non-cavitation}, we have a non-zero determinant:
		\begin{align*}
			|J_{\Sigma_b}| \geq  \frac{h_{\min}}{1+\beta |b|_{L^{\infty}}}.
		\end{align*}
	\end{remark}

The next result shows that the properties of solutions of the boundary problem \eqref{Laplace pb} can be
obtained from the study of an equivalent elliptic boundary value problem defined on $\mathcal{S}_b$.

\begin{prop}
	Let $\phi_b = \Phi\circ \Sigma_b$ where the map $\Sigma_b$ is given in Definition \ref{Map to semi-strip}. Then under the provisions of  Definition \ref{Def Original Laplace} we have that $\Phi$ is a (variational, classical) solution of \eqref{Laplace pb} if and only if $\phi_b$ is a (variational, classical) solution of
	\begin{equation}\label{Elliptic pb Sb}
		\begin{cases}
			\nabla^{\mu}_{X,z} \cdot P(\Sigma_b) \nabla_{X,z}^{\mu }\phi_b = 0 \quad \text{in} \quad S_b
			\\	
			\phi_b|_{z = 0} = \psi,  \quad \partial_{n_{b}}^{P_b} \phi_b |_{z=-h_b} = 0,
		\end{cases}
	\end{equation}	
	where the matrix $P(\Sigma_b)$ is given by
	\begin{equation}\label{expression of P sigma}
		P(\Sigma_b) 
		= 
		|J_{\Sigma_b}| (I^{\mu})^{-1} J_{\Sigma_b}^{-1} (I^{\mu})^2 (J_{\Sigma_b}^{-1})^T(I^{\mu})^{-1},
	\end{equation}
	and the Neumann condition reads
	\begin{equation}\label{Neuman new}
		\partial_{n_{b}}^{P_b} \phi_b|_{z=-h_b} = \mathbf{n}_b \cdot I^{\mu} P(\Sigma_b) \nabla^{\mu}_{X,z} \phi_b|_{z=-h_b}.
	\end{equation}
	Moreover,  the matrix $P(\Sigma_b)$ is coercive, i.e. there exists $c>0$ such that for all $Y \in \R^{d+1}$ and any $(X,z) \in \mathcal{S}_b$ there holds,
	\begin{equation}\label{Coercivity}
		P(\Sigma_b) Y \cdot Y \geq c |Y|^2.
	\end{equation}
	\end{prop}
	\begin{remark}\label{Elliptic op}
		We may compute the inverse Jacobi-matrix $J_{\Sigma_b}^{-1}$ so that using the expression for $P(\Sigma_b)$ \eqref{expression of P sigma}, we find
		\begin{align*}
			P(\Sigma_b)
			=
			\begin{pmatrix} 
				(1+\partial_z \sigma )\mathrm{Id} & -\sqrt{\mu} \nabla_X \sigma 
				\\
				-\sqrt{\mu}(\nabla_X \sigma)^T & \dfrac{1+ \mu h_b|\nabla_X\sigma|^2}{1+\partial_z \sigma} 
			\end{pmatrix}.
		\end{align*}
		Now, since  $\sigma$ is given by \eqref{sigma} we find that
		\begin{equation*}
			1+ \partial_z\sigma =1 + \dfrac{\varepsilon\zeta}{h_b} =  \dfrac{h}{h_b},
		\end{equation*} 
		and 
		\begin{align*}
			P(\Sigma_b)
			=
			\begin{pmatrix} 
				\dfrac{h}{h_b} \mathrm{Id} & -\sqrt{\mu} \nabla_X \sigma
				\\ -\sqrt{\mu}(\nabla_X \sigma)^T & \dfrac{h_b + \mu h_b|\nabla_X\sigma|^2}{h} 
			\end{pmatrix},
		\end{align*}
		where 
		\begin{equation}\label{grad sigma}
			\nabla_X\sigma = \varepsilon\nabla_X\Big(\dfrac{\zeta}{h_b}\Big)z + \varepsilon\nabla_X\zeta.
		\end{equation}\\

	\end{remark}

	\begin{proof}
		The fact that $\nabla^{\mu}_{X,z} \cdot P(\Sigma_b) \nabla_{X,z}^{\mu }\phi_b  = 0$ in $\mathcal{S}_b$  and that $P(\Sigma_b)$ satisfies \eqref{Coercivity} is classical and we simply refer to \cite{WWP}, Proposition $2.25$ and Lemma $2.26$. 
		
		To verify the Neumann condition, we first use the chain rule to make the observation
		\begin{equation}\label{Chain rule}
			 \nabla_{X,z}^{\mu} \phi_b = I^{\mu} (J_{\Sigma_b})^T (I^{\mu})^{-1}(\nabla^{\mu}_{X,z} \Phi)\circ\Sigma_b.
		\end{equation}
		Then by \eqref{Neuman new}, we get that
		\begin{align*}
				\partial_{n_{b}}^{P_b} \phi_b 
				& = 
				|J_{\Sigma_b}|  \mathbf{n}_b \cdot  (J_{\Sigma_b})^{-1} I^{\mu} (\nabla_{X,z}^{\mu}\Phi) 
				\circ \Sigma_b 
				\\
				& = 
				\mathbf{n}_b
				\cdot I^{\mu}
				\begin{pmatrix}
					(1+ \partial_z \sigma)\mathrm{Id} & \mathbf{0} \\
					- (\nabla_X \sigma)^T & 1
				\end{pmatrix}
				(\nabla_{X,z}^{\mu}\Phi) 
				\circ \Sigma_b 
				\\
				& = 
				\mathbf{n}_b
				\cdot
				I^{\mu}(\nabla_{X,z}^{\mu}\Phi) 
				\circ \Sigma_b
				+
				\mathbf{n}_b
				\cdot I^{\mu}
				\begin{pmatrix}
					\partial_z \sigma \mathrm{Id} & \mathbf{0} \\
					- (\nabla_X \sigma)^T & 0
				\end{pmatrix}
				(\nabla_{X,z}^{\mu}\Phi) 
				\circ \Sigma_b.
		\end{align*}
		Now, use \eqref{Laplace pb} with the fact that for $z = -h_b$ then 
		\begin{align*}
			0 = \partial_{n_b} \Phi |_{-h_b} = 
			\mathbf{n}_b
			\cdot
			I^{\mu}(\nabla_{X,z}^{\mu}\Phi) 
			\circ \Sigma_b.
		\end{align*}
		Therefore, we are left with the expression
		\begin{align*}
			\partial_{n_{b}}^{P_b} \phi_b |_{z=-h_b}
			 & =  
			  \frac{1}{|\mathbf{n}_b|} 
			\begin{pmatrix}
				- \beta \nabla_X b 
				\\
				1
			\end{pmatrix}
			\cdot 
			\begin{pmatrix}
				\mu \partial_z\sigma  (\nabla_X \Phi)|_{z=-h_b}
				\\
				-\mu \nabla_X \sigma \cdot  (\nabla_X \Phi)|_{z=-h_b}
			\end{pmatrix}
			\\
			& = 
			-\frac{\mu}{|\mathbf{n}_b|} 
			( \beta \partial_z \sigma \nabla_X b+  \nabla_X \sigma  ) \cdot (\nabla_X \Phi)\big|_{z = -h_b}.
		\end{align*}
			But $\partial_z\sigma|_{z=-h_b} = \dfrac{\varepsilon\zeta}{h_b}$, $\nabla_X\sigma|_{z=-h_b} = -\beta \dfrac{\varepsilon\zeta}{h_b} \nabla_X b $, and thus the proof is complete. 
	\end{proof}

	In the next section, we will make expansions of $\phi_b$ and then use the expression of \eqref{V bar} to approximate the Dirichelet-Naumann operator. But first, we must relate the definition of $\overline{V}^{\mu}[\ve \zeta, \beta b]\psi$ with the new velocity potential on $\mathcal{S}_b$. \\
	
	\begin{prop}
		Let $\Sigma_b$ be given in Definition \ref{Map to semi-strip}.  Then under the provisions of Definition \ref{Def V bar}, the operator \eqref{V bar} is equivalent to the following formulation
		\begin{equation}\label{New V bar}
			\overline{V}^{\mu}[\ve \zeta, \beta b] \psi 
			=\frac{1}{h}
			\int_{-1+\beta b}^{0} \big{[}\frac{h}{h_b}\nabla_X \phi_b - (\varepsilon\nabla_X\Big(\dfrac{\zeta}{h_b}\Big)z + \varepsilon\nabla_X\zeta) \partial_z \phi_b \big{ ]}\: \mathrm{d}z,
		\end{equation}
		where $\phi_b = \Phi\circ\Sigma_b$.
	\end{prop}
	
	\begin{proof}
		We use the new variables defined by the mapping $\Sigma_b$ and the chain rule      to get
		\begin{align*}
			\overline{V}^{\mu}[\ve \zeta, \beta b] \psi 
			& =
			\dfrac{1}{h}\int_{-1+\beta b}^{0} (\nabla_X \Phi)\circ \Sigma_b   \: |J_{\Sigma_b}|\: \mathrm{d}z
			\\
			& = 
		\dfrac{1}{h}	
		\int_{-1+\beta b}^{0} \big{[}\frac{h}{h_b}\nabla_X \phi_b - \nabla_X \sigma \partial_z \phi_b \big{]}\: \mathrm{d}z.
		\end{align*}
		Then using \eqref{sigma}, we obtain the result.
	\end{proof}
	 
	\subsection{Multi-scale expansions}\label{mult general pb}
	
	In order to make expansions of $\phi_b$ we first make several observations on how to decompose system \eqref{Elliptic pb Sb}.

        \begin{obs}\label{A op}
        We can decompose the elliptic operator given in Remark \ref{Elliptic op} into:
	\begin{align*}
		\dfrac{h}{h_b}\nabla^{\mu}_{X,z} \cdot P(\Sigma_b) \nabla^{\mu}_{X,z} \phi_b = \Delta^{\mu}_{X,z} \phi_b + \mu\varepsilon  A[\nabla_X, \partial_z]\phi_b,
	\end{align*}
	where 
	\begin{align*}
		 A[\nabla_X, \partial_z]\phi_b
		&= 
		\dfrac{\zeta}{h_b} \Delta_X \phi_b
		+
		\dfrac{h}{h_b} \nabla_X \cdot \big( \dfrac{\zeta}{h_b} \nabla_X\phi_b\big) 
		-
		\dfrac{h}{h_b} \nabla_X \cdot \big(\dfrac{1}{\varepsilon}\nabla_X \sigma \partial_z \phi_b \big) 
		\\  
		&\notag
		-
		\dfrac{h}{h_b} \partial_z \big( \dfrac{1}{\varepsilon} \nabla_X\sigma \cdot \nabla_X \phi_b \big) 
		+
		\partial_z\big( \dfrac{1}{\varepsilon}|\nabla_X \sigma|^2 \partial_z\phi_b \big).
	\end{align*}
        We may simplify this expression by using formula \eqref{grad sigma} for $\nabla_X \sigma$ to get that
        \begin{align}\label{A}
		 A[\nabla_X, \partial_z]\phi_b
		&= 
		\frac{\zeta}{h_b}(1+\frac{h}{h_b})\Delta_X\phi_b
            -
            \frac{h}{h_b}\nabla_X \zeta \cdot \nabla_X \partial_z \phi_b
		\\  
		&\notag
		-
		\dfrac{h}{h_b} \nabla_X \cdot \big(\dfrac{1}{\varepsilon}\nabla_X \sigma \partial_z \phi_b \big) 
		+
		\partial_z\big( \dfrac{1}{\varepsilon}|\nabla_X \sigma|^2 \partial_z\phi_b \big).
	\end{align}
        In this formula, we emphasize the terms that do not contain $\partial_z \phi_b$. This is because these are the leading terms in the approximations that are performed below. 
        
        \end{obs}
        \begin{obs}\label{B tilde op}
	Similarly, we can also decompose the Neumann condition into
            \begin{align*}
                \frac{h}{h_b}|\mathbf{n}_b|\partial_{n_{b}}^{P_b} \phi_b|_{z=-h_b} 
                & =
               [ \partial_z \phi_b 
                -
                \mu \beta \frac{h}{h_b} \nabla_Xb \cdot \nabla_X \phi_b 
                -
                \mu \beta^2 \ve \zeta \frac{|\nabla_X b|^2}{h_b} \partial_z \phi_b]|_{z=-h_b} 
                \\ 
                & 
                = 
                 [\partial_z \phi_b 
                -
                \mu \beta \nabla_Xb \cdot \nabla_X \phi_b]|_{z=-h_b}  
                + 
                \mu \ve \beta B[\nabla_X, \partial_z] \phi_b|_{z=-h_b} .
            \end{align*}
            where
            \begin{equation*}
                B[\nabla_X, \partial_z] \phi_b  
                =
                -
                \frac{\zeta}{h_b} \nabla_Xb \cdot \nabla_X \phi_b 
                -
                 \beta  \zeta \frac{|\nabla_X b|^2}{h_b} \partial_z \phi_b.
            \end{equation*}
        \end{obs}

	To summarize the observations, we now have that $\phi_b$ solves
        \begin{equation}\label{decomposed elliptic pb}
            \begin{cases}
        		\Delta^{\mu}_{X,z}\phi_b = -\mu \ve A[\nabla_X, \partial_z]\phi_b \quad \text{in} \quad \mathcal{S}_b
        		\\	
        		\phi_b|_{z = 0} = \psi,  \quad  [\partial_z \phi_{b} - \mu \beta \nabla_X b \cdot \nabla_X \phi_b ]|_{z=-h_b} = \mu \ve \beta  B[\nabla_X, \partial_z]\phi_b|_{z = -h_b}.
            \end{cases}
        \end{equation}

        \begin{remark}\label{Remark Craig sulem}
            In the paper \cite{CraigSulemGuyenneNicholls05}, their strategy is to  solve \eqref{decomposed elliptic pb} first in the case $\ve = 0$, where the solution is defined in terms of the inverse of a pseudo-differential operator. If we add the parameters $\mu$ and $\beta$ then, in dimension one, this operator is given by
            \begin{align}\label{PDO Craig Sulem}
                \mathcal{L}^{\mu}[\beta b]  = - \cosh{((-1+\beta b(X))\sqrt{\mu}D)}^{-1}\sinh{(\beta b(X) \sqrt{\mu }D)}\mathrm{sech}(\sqrt{\mu }D).
            \end{align}
            Formally, in dimension one, they obtain the first order approximation:
            \begin{equation*}
                \mathcal{G}_0 = \sqrt{\mu}D\tanh(\sqrt{\mu}D) + \sqrt{\mu}D\mathcal{L}^{\mu}[\beta b].
            \end{equation*}
            At higher order they obtain  the expansion of $\mathcal{G}^{\mu}$ given on the form
            \begin{align*}
                \frac{1}{\mu}\mathcal{G}^{\mu} = \frac{1}{\mu}\sum\limits_{j=0}^n\ve^j \mathcal{G}_j + O(\ve^{n+1}),
            \end{align*}
            where $\mathcal{G}_j$ defined recursively for $j\geq 0$ and is the classical expansion for small amplitude waves when $\beta=0$ \cite{CraigSulemGuyenneNicholls05} (see also \cite{WWP} where the approximation is proved with Sobolev bounds when $\beta=0$). In this paper, our approach allow us to decouple the parameters $\mu$, $\ve$ and $\beta$, writing expansions of the Dirichlet-Neumann operator which do not include the inversion of a pseudo-differential operator. 
        \end{remark}

    \subsection{Multi-scale expansions of the velocity potential $\phi_b$}\label{mult phi}
    We will now use \eqref{decomposed elliptic pb} to make multi-scale expansions of $\phi_b$. But first, we state an important result to justify the procedure.
    \begin{prop}\label{Prop elliptic est}
        Let $d=1,2$, $t_0>\frac{d}{2}$, and $k\in \N $. Let $b \in C^{\infty}_c(\R^d)$ and $\zeta \in H^{{\max\{t_0 +2, k+1\}}}(\R^d)$ such that \eqref{non-cavitation 1} and \eqref{non-cavitation 2} are satisfied. Let also $f \in H^{k,k}(\mathcal{S}_b)$ and $g\in H^k(\R^d)$  be two given functions. Then the boundary value problem
        \begin{equation}\label{Elliptic est pb}
		\begin{cases}
			\nabla^{\mu}_{X,z} \cdot P(\Sigma_b) \nabla_{X,z}^{\mu }u = f \quad \text{in} \quad \mathcal{S}_b
			\\	
			u|_{z = 0} = 0,  \quad \partial_{n_{b}}^{P_b} u |_{z=-1 + \beta b} = g,
		\end{cases}
	\end{equation}	
        admits a unique solution $u \in H^{k+1,0}(\mathcal{S}_b)$. Moreover, the solution satisfies the estimate
        \begin{equation}\label{elliptic est}
            \|\nabla^{\mu}_{X,z} u\|_{H^{k,0}(\mathcal{S}_b)} \leq M(k+1) ( |g|_{H^k} + \sum\limits_{j=0}^k\|f\|_{H^{k-j,j}(\mathcal{S}_b)}).
        \end{equation}
    \end{prop}
    The proof of Proposition \ref{Prop elliptic est} is similar to the one of Proposition $4.5$ in \cite{DucheneMMWW21} and is postponed for Appendix $A$, Subsection \ref{A2} to ease the presentation. We may now use this result to construct an implicit function such that $\phi = \phi_b + O(\mu \ve)$.

    \begin{prop}\label{Proposition phi0} Let $d=1,2$, $t_0>\frac{d}{2}$, and $k\in \N $. Let $\psi \in \dot{H}^{k+3}(\R^d)$. Let also $b \in C^{\infty}_c(\R^d)$ and $\zeta\in H^{{\max\{t_0 +2, k+2\}}}(\R^d)$ such that \eqref{non-cavitation 1} and \eqref{non-cavitation 2} are satisfied. Then there exists a unique solution $\phi_0 \in H^{k,0}(\mathcal{S}_b)$ solving
    \begin{align}\label{elliptic problem phi}
            \begin{cases}
                \Delta_{X,z}^{\mu} \phi = 0 \ \ \mathrm{in} \ \ \mathcal{S}_b, \\
                \phi|_{z=0} = \psi, \ \ \big[\partial_z\phi - \mu\beta \nabla_X b \cdot \nabla_X \phi\big]\big|_{z=-1+\beta b}  =0, 
            \end{cases}
        \end{align}
        where the solution satisfies the estimates
        \begin{equation}\label{est on phi in pf}
            \| \nabla_{X,z}  \phi \|_{H^{k,0}(\mathcal{S}_b)} \leq M(k+1)|\nabla_X \psi|_{H^k},
        \end{equation}
        and
        \begin{equation}\label{est phi}
            \|\nabla_{X,z}^{\mu}(\phi_b - \phi) \|_{H^{k,0}(\mathcal{S}_b)} \leq  
            \mu \ve M(k+2) |\nabla_X \psi|_{H^{k + 2}}.
        \end{equation}
    \end{prop}
    \begin{proof}
        The existence and uniqueness is a direct consequence Riesz representation Theorem, and the Poincaré inequality \eqref{Poincare 1}. Moreover, we know that $\tilde \phi = \phi(X,zh_b)$ defined is defined on the fixed strip $\mathcal{S} = \R^d \times [-1,0]$ and satisfies
        \begin{equation*}
            \| \nabla_{X,z} \tilde \phi \|_{H^{k,0}(\mathcal{S})} \leq M(k+1)|\nabla_X \psi|_{H^k}.
        \end{equation*}
        by Proposition $2.37$ with $\ve = 0$ in \cite{WWP}. Then we may use this result, together with the relations $\nabla_{X,z} \phi(X,z) = \nabla_{X,z} \big(\tilde \phi (X, \frac{z}{h_b})\big) \in H^{k,0}(\mathcal{S}_b)$ and $\partial_z^2 \phi=-\mu \Delta_X \phi$ to obtain the bound
        \begin{equation*}
            \| \nabla_{X,z}  \phi \|_{H^{k,0}(\mathcal{S}_b)} \leq M(k+1)|\nabla_X \psi|_{H^k}.
        \end{equation*}
        Consequently, estimate \eqref{est phi}  follows from Proposition \ref{Prop elliptic est} and the observation that
        \begin{equation*}
            \|\nabla_{X,z}^{\mu}(\phi_b - \phi) \|_{H^{k,0}(\mathcal{S}_b)} \leq \mu \ve M(k+1)
            \|\frac{h_b}{h}  A[\nabla_X, \partial_z]\phi\|_{H^{k,0}(\mathcal{S}_b)}
            + \| B[\nabla_X, \partial_z]\phi|_{z = -h_b}\|), 
        \end{equation*}
        where $A[\nabla_X, \partial_z]$ is a second order differential operator and $B[\nabla_X, \partial_z]$ is a first order operator (second order after using the trace estimate \eqref{Poincare 2}). Then we simply conclude by using \eqref{est on phi in pf} combined with product estimates for $H^k(\R^d)$ given by \eqref{Classical prod est} and \eqref{prod est division}.

    \end{proof}

    \color{black}

    Next, we will make expansion with respect to $\mu \beta$ by splitting problem \eqref{decomposed elliptic pb} into two parts. In particular, we will construct a function $\phi_0 = \phi_b + O(\mu(\ve + \beta))$ by solving the first part of the \lq\lq straightened\rq\rq \: Laplace problem with an explicit error of order $O(\mu \beta)$, that will be canceled later, and an additional error of $O(\mu \ve)$. 
    \begin{prop}\label{Proposition phi0} Let $d=1,2$, $t_0>\frac{d}{2}$, and $k\in \N $. Let $\psi \in \dot{H}^{k+3}(\R^d)$. Let also $b \in C^{\infty}_c(\R^d)$ and $\zeta\in H^{{\max\{t_0 +2, k+2\}}}(\R^d)$ such that \eqref{non-cavitation 1} and \eqref{non-cavitation 2} are satisfied. If $\phi_0$ satisfies the following Laplace problem:
        \begin{align}\label{elliptic problem phi 0}
            \begin{cases}
                \Delta_{X,z}^{\mu} \phi_0 = 0 \ \ \mathrm{in} \ \ \mathcal{S}_b, \\
                \phi_0|_{z=0} = \psi, \ \ \big[\partial_z\phi_0 - \mu\beta \nabla_X b \cdot \nabla_X \phi_0 \big]\big|_{z=-1+\beta b}  = \mu \beta \nabla_X \cdot \mathcal{L}_{1}^{\mu}[\beta b]\nabla_X\psi, 
            \end{cases}
        \end{align}
        where
        \begin{align*}
            \mathcal{L}_{1}^{\mu}[\beta b]\nabla_X\psi =  -  \frac{1}{\beta}\sinh{(\beta b(X) \sqrt{\mu}|\mathrm{D}|)}\mathrm{sech}(\sqrt{\mu}|\mathrm{D}|)\dfrac{1}{\sqrt{\mu}|\mathrm{D}|}\nabla_X\psi,
        \end{align*}
        then for $z\in [-1 + \beta b,0]$ its expression is given by
        \begin{align}\label{phi0}
            \phi_0 = \dfrac{\cosh{((z+1)\sqrt{\mu}|\mathrm{D}|)}}{\cosh{(\sqrt{\mu}|\mathrm{D}|)}}\psi = \mathrm{F}_0 \psi.
        \end{align}
        Moreover, the solution satisfies the estimate
        \begin{equation}\label{est phi0}
            \|\nabla_{X,z}^{\mu}(\phi_b - \phi_0) \|_{H^{k,0}(\mathcal{S}_b)} \leq  \mu (\ve+\beta) M(k+2) |\nabla_X \psi|_{H^{k + 2}}.
        \end{equation}
    \end{prop}
    \begin{proof}
         Since $\phi_0$ is given by the solution of the Laplace problem when the bottom is flat, we only need to verify the boundary condition at the bottom. In fact, we have that
        \begin{align*}
            \mathrm{LHS}: 
            & =
            \big[\partial_z \phi_0 - \mu\beta \nabla_X b \cdot \nabla_X\phi_0\big]\big|_{z=-1 + \beta b} 
            \\
            & = \mathcal{F}^{-1}\Big{(}\sqrt{\mu}|\xi|\sinh{((z+1)\sqrt{\mu}|\xi|)}\mathrm{sech}(\sqrt{\mu}|\xi|) \widehat{\psi}(\xi) \Big{)}(X)\big|_{z=-1+\beta b(X)}
            \\
            &
            \hspace{0.5cm} - \mathcal{F}^{-1}\Big{(} \mu \beta \nabla_X b(X) \cdot i \xi \cosh{((z+1)\sqrt{\mu}|\xi|)}\mathrm{sech}(\sqrt{\mu}|\xi|) \widehat{\psi}(\xi) \Big{)}(X) \big|_{z=-1+\beta b(X)}
            \\
            & =
            -\sqrt{\mu}\nabla_X\cdot\big( \sinh{(\beta b(X) \sqrt{\mu}|\mathrm{D}|)}\mathrm{sech}(\sqrt{\mu}|\mathrm{D}|)\dfrac{1}{|\mathrm{D}|}\nabla_X\psi\big).
        \end{align*}
 
         The next step is to prove that $\phi_0$ approximates $\phi_b$ with a precision of $O(\mu( \ve + \beta) )$. To that end, we first note that $u = \phi_b - \phi_0$ solves the elliptic problem \eqref{Elliptic est pb} with
         \begin{equation*}
            f = - \mu \ve \frac{h_b}{h} A[\nabla_X, \partial_z]\phi_0,
         \end{equation*}
         and
         \begin{equation*}
             g = 
             \mu \beta\frac{h_b }{h|\mathbf{n}_b|}\big{(}   
             \nabla_X \cdot \big(\mathcal{L}_{1}^{\mu}[\beta b]\nabla_X\psi\big)
             +
             \ve B[\nabla_X, \partial_z]\phi_0 \big{)}|_{z=-1 + \beta b},
         \end{equation*}
         where the expressions of $f$ and $g$ are deduced from the decompositions of Observations \ref{A op} and \ref{B tilde op} and the construction of $\phi_0$.
         Moreover, since $-h_b(X)>-2$ (see \eqref{non-cavitation 2}), we can extend the definition of $\phi_0$ to the domain $\mathcal{S} := \R^d \times [-2,0]$. For any $(X,z) \in \mathcal{S}$, we write
         \begin{align*}
             \phi_0 = \dfrac{\cosh{((z+1)\sqrt{\mu}|\mathrm{D}|)}}{\cosh{(\sqrt{\mu}|\mathrm{D}|)}}\psi.
         \end{align*}
         This extension is a Fourier multiplier depending on $z$, and we can use the estimates in Proposition \ref{Prop F0} together with the fact that $A[\nabla_X, \partial_z]\bullet$, given by \eqref{A}, only depends on functions of $X$ and is polynomial in $z$.
         Thus, combining the elliptic estimate \eqref{elliptic est} with \eqref{L1 est}, the non-cavitation conditions \eqref{non-cavitation 1}, \eqref{non-cavitation 2}, the  product estimates for $H^k(\R^d)$ given by \eqref{Classical prod est}  and \eqref{prod est division},  we obtain that
         \begin{align*}
            \|\nabla^{\mu}_{X,z} u\|_{H^{k,0}(\mathcal{S}_b)} & 
            \leq
            \mu \ve M(k+1)
            \|\frac{h_b}{h}  A[\nabla_X, \partial_z]\phi_0\|_{H^{k,0}(\mathcal{S}_b)} 
            \\
            &
            \hspace{0.5cm}
            +
            \mu \ve M(k+1)
            \big{|}\frac{ \zeta }{h}\big{|}_{H^{k+2}}
            (
            |\nabla_X \cdot \big(\mathcal{L}_{1}^{\mu}[\beta b]\nabla_X\psi\big)|_{H^k}
            +
            \big|\tilde{B}[\nabla_X, \partial_z]\phi_0|_{z=-h_b}\big|_{H^k})
            \\
            &
            \leq \mu (\ve + \beta) M(k+2)
            \|\frac{h_b}{h}  A[\nabla_X, \partial_z]\phi_0\|_{H^{k,0}(\mathcal{S})} 
            +
            \mu (\ve + \beta )M(k+2)
            |\nabla\psi|_{H^{k+1}}\\
            & 
            \lesssim \mu (\ve + \beta) M(k+2)
             |\nabla \psi|_{H^{k+2}}.
         \end{align*}
    \end{proof}
    \begin{remark}
        The source term $\mu \beta \nabla_X \cdot \mathcal{L}_{1}^{\mu}[\beta b]\nabla_X\psi$ in the Neumann condition of \eqref{elliptic problem phi 0} is chosen so that the solution $\phi_0$ of the system does not depend on the inverse of a pseudo-differential operator. Indeed, any other source term in the Neumann condition would induce the dependence of the solution on operators of this kind.
    \end{remark}
    We now construct the next order approximation by canceling the error of order $O(\mu\beta)$. But first, we make an observation on the problem that needs to be solved.
    \begin{obs}
        To make the next order approximation $\phi_1$ such that $\phi_b = \phi_0 + \mu\beta \phi_1 + O(\mu( \ve + \mu \beta^2))$, we solve the problem
        \begin{align*}
        \begin{cases}
            \Delta_{X,z}^{\mu}\phi_1 = \mu \beta F \ \ \mathrm{in} \ \ \mathcal{S}_b,\\
            \phi_1|_{z=0} = 0, \ \ \partial_z \phi_1|_{z=-1 + \beta b} =   -  \nabla_X \cdot \big(\mathcal{L}_{1}^{\mu}[\beta b]\nabla_X\psi\big),
        \end{cases} 
    \end{align*}
    where F is to be chosen and satisfies
    \begin{align}\label{est on Source}
        \| F \|_{H^{k,k}(\mathcal{S}_b)} \leq  M(k+2)|\nabla_X \psi|_{H^{k+2}}.
    \end{align}
    so that formally
    \begin{align*}
        \begin{cases}
            \frac{h}{h_b}\nabla_{X,z}\cdot P(\Sigma_b) \nabla_{X,z} (\phi_0 + \mu\beta \phi_1) = O(\mu \ve + \mu^2 \beta^2)  \ \ \mathrm{in} \ \ \mathcal{S}_b,\\
            (\phi_0 + \mu \beta \phi_1)|_{z=0} = \psi, \ \ \frac{h}{h_b}\partial_{n_b}^{P_b} (\phi_0 + \mu\beta\phi_1)|_{z=-1 + \beta b} = O(\mu\ve + \mu^2\beta^2).
        \end{cases} 
    \end{align*}
    Moreover, the presence of the source term $\mu\beta F$ is motivated by the fact that the boundary conditions require a function of the form
    \begin{align*}
        \phi_1 = -h_b\frac{\sinh(\frac{z}{h_b}\sqrt{\mu}|\mathrm{D}|)}{\cosh(\sqrt{\mu}|\mathrm{D}|)}\frac{1}{\sqrt{\mu}|\mathrm{D}|}\nabla_X \cdot \big(\mathcal{L}_{1}^{\mu}[\beta b]\nabla_X\psi\big{)},
    \end{align*}
    for $-h_b\leq z \leq 0$. Indeed, if we let $G = \nabla_X \cdot \big(\mathcal{L}_{1}^{\mu}[\beta b]\nabla_X\psi\big)$, then
    \begin{align*}
        \partial_z \phi_1|_{z=-h_b} 
        & =
        -
        \mathcal{F}^{-1}\Big{(} \frac{\cosh(\frac{z}{h_b(X)}\sqrt{\mu}|\xi|)}{\cosh(\sqrt{\mu}|\xi|)} \hat{G}(\xi)\Big{)}(X)|_{z=-h_b(X)}
        \\
        & = 
         -G(X).
    \end{align*}
    Now, let us compute the Laplace operator. To do so, we introduce the notation
    \begin{equation*}
        T_1(z)[X,\mathrm{D}] \bullet = \mathcal{F}^{-1}\Big{(} \frac{\sinh(\frac{z}{h_b(X)}\sqrt{\mu}|\xi|)}{\cosh(\sqrt{\mu}|\xi|)} \hat{\bullet}\Big{)}(X),
    \end{equation*}
    and
    \begin{equation*}
        T_2(z)[X,\mathrm{D}] \bullet = \mathcal{F}^{-1}\Big{(} \frac{\cosh(\frac{z}{h_b(X)}\sqrt{\mu}|\xi|)}{\cosh(\sqrt{\mu}|\xi|)} \hat{\bullet}\Big{)}(X).
    \end{equation*}
    Using the identity $\Delta_X = - |\mathrm{D}|^2$, we observe that
    \begin{align*}
        \partial_z^2\phi_1 & =  \frac{\mu}{h_b} T_1(z)[X,\mathrm{D}]\frac{\Delta_X}{\sqrt{\mu} |\mathrm{D}|} G.
    \end{align*}
    Similarly, after some computations we find
    \begin{align*}
        \mu\Delta_X \phi_1 
        & =
        -\mu h_b T_1(z)[X, \mathrm{D}]\frac{\Delta_X}{\sqrt{\mu}|\mathrm{D}|}G + \mu [h_b T_1(z)[X,D]\frac{1}{\sqrt{\mu}|D|}, \Delta] G\\
        & = -\mu T_1(z)[X, \mathrm{D}]\frac{\Delta_X}{\sqrt{\mu}|\mathrm{D}|}G + \mu [h_b T_1(z)[X,D]\frac{1}{\sqrt{\mu}|D|}, \Delta] G + \mu\beta b T_1(z)[X, \mathrm{D}]\frac{\Delta_X}{\sqrt{\mu}|\mathrm{D}|}G.
    \end{align*}
    We define $\tilde F$ by
    \begin{align*}
        \tilde F = \mu \big[h_b T_1(z)[X,D]\frac{1}{\sqrt{\mu}|D|}, \Delta\big] G + \mu\beta b T_1(z)[X, \mathrm{D}]\frac{\Delta_X}{\sqrt{\mu}|\mathrm{D}|}G
    \end{align*}
    where $\mu \big[h_b T_1(z)[X,D]\frac{1}{\sqrt{\mu}|D|}, \Delta\big] G = O(\mu \beta)$ by direct calculation.
    From this expression, we identify $F$ by 
    \begin{align*}
        \Delta_{X,z}^{\mu} \phi_1 &  = \mu \beta \tilde{F} + \mu (\frac{1}{h_b}-1)T_1(z)[X,\mathrm{D}]\frac{\Delta_X}{\sqrt{\mu} |\mathrm{D}|} G
        \\
        & = 
        \mu \beta \tilde{F} 
        +
        \frac{\mu \beta b}{h_b}T_1(z)[X,\mathrm{D}]\frac{\Delta_X}{\sqrt{\mu} |\mathrm{D}|} G 
        \\
        & = \mu \beta F.
    \end{align*}
    The estimate  \eqref{est on Source} on $F$  is a consequence of the boundedness of $T_1$ and $T_2$ for $z \in [-h_b,0]$, given by Proposition \ref{Prop T1 and T2}, while we estimate $\mathcal{L}_1^{\mu}$ in $H^{k+2}(\R^d)$ by  Proposition \ref{Prop op} with inequality \eqref{L1 est}.
    
    \end{obs}
    We summarize these observations in the next Proposition.

    \begin{prop}\label{Prop phi 1}
    Let $d=1,2$, $t_0>\frac{d}{2}$, and $k\in \N $. Let $\psi \in \dot{H}^{k+4}(\R^d)$. Let also $b \in C^{\infty}_c(\R^d)$ and $\zeta \in H^{{\max\{t_0 +2, k+2\}}}(\R^d)$ such that \eqref{non-cavitation 1} and \eqref{non-cavitation 2} are satisfied.  Then the function $\phi_1$ given by
    \begin{align}\label{phi_1}
        \phi_1 = -h_b\frac{\sinh(\frac{z}{h_b}\sqrt{\mu}|\mathrm{D}|)}{\cosh(\sqrt{\mu}|\mathrm{D}|)}\frac{1}{\sqrt{\mu}|\mathrm{D}|}\nabla_X \cdot \big(\mathcal{L}_{1}^{\mu}[\beta b]\nabla_X\psi\big{)},
    \end{align}
    satisfies 
    \begin{align}\label{System phi1}
        \begin{cases}
            \Delta_{X,z}^{\mu}\phi_1 = \mu \beta F \ \ \mathrm{in} \ \ \mathcal{S}_b,\\
            \phi_1|_{z=0} = 0, \ \ \partial_z \phi_1|_{z=-1 + \beta b} =   -  \nabla_X \cdot \big(\mathcal{L}_{1}^{\mu}[\beta b]\nabla_X\psi),
        \end{cases} 
    \end{align}
    where
    $F \in H^{k,k}(\mathcal{S}_b)$ is such that
    \begin{align}\label{est on Source 1}
        \| F \|_{H^{k,k}(\mathcal{S}_b)} \leq  M(k+2)|\nabla_X \psi|_{H^{k+2}},
    \end{align}
    and
    \begin{align*}
        \mathcal{L}_{1}^{\mu}[\beta b]\nabla_X\psi =  -  \frac{1}{\beta}\sinh{(\beta b(X) \sqrt{\mu}|\mathrm{D}|)}\mathrm{sech}(\sqrt{\mu}|\mathrm{D}|)\dfrac{1}{\sqrt{\mu}|\mathrm{D}|}\nabla_X\psi.
    \end{align*}
    Moreover, for $\phi_b$ satisfying \eqref{Elliptic pb Sb} and $\phi_0$ given by \eqref{phi0} there holds,
    \begin{align}\label{Est muve mu2beta2}
        \|\nabla^{\mu}_{X,z}(\phi_b - (\phi_0 + \mu\beta \phi_1))\|_{H^{k,0}(\mathcal{S}_b)} \lesssim   (\mu\varepsilon +\mu^2\beta^2) M(k+2)|\nabla \psi|_{H^{k+3}}.
    \end{align}
    \end{prop}
    \begin{proof}
        By constriction of $\phi_1$ given by \eqref{phi_1}, we know there exists an $F$ such that \eqref{est on Source 1} is satisfied. Now, let us prove \eqref{Est muve mu2beta2}. First, observe that the function 
        \begin{equation*}
            u = \phi_b - (\phi_0 + \mu \beta\phi_1)
        \end{equation*}
        solves
        \begin{align*}
            \frac{h}{h_b}\nabla_{X,z}^{\mu}P(\Sigma_b)\nabla_{X,z}^{\mu} u & = - \mu \ve A[\nabla_X, \partial_z]\phi_0 - \mu^2 \ve \beta A[\nabla_X,\partial_z]\phi_1 - \mu^2 \beta^2 F
            \\
            & =: f.
        \end{align*}
        Moreover, at $z  =  -h_b$, we have the Neumann condition
        \begin{align*}
            \frac{h}{h_b}|\mathbf{n}_b|\partial_{n_{b}}^{P_b} u
                & =
                 \partial_z \phi_0
                -
                \mu \beta \nabla_Xb \cdot \nabla_X \phi_0  
                + 
                \mu \ve \beta B[\nabla_X, \partial_z] \phi_0  + \mu \beta \partial_z \phi_1 +
                \mu^2 \ve \beta^2 B[\nabla_X, \partial_z] \phi_1   
                \\
                & = 
                \mu \ve \beta B[\nabla_X, \partial_z] \phi_0  + 
                \mu^2 \ve \beta^2 B[\nabla_X, \partial_z] \phi_1 
                \\
                & = : g.
        \end{align*}
        Estimating each terms, noting that $A[\nabla_X, \partial_z]$ is a differential operator of order two and $B[\nabla_X, \partial_z]$ is of order one, while the error due to $F$ is given by construction, we obtain that
        \begin{align*}
            \|\nabla^{\mu}_{X,z} u\|_{H^{k,0}(\mathcal{S}_b)} \leq \mu(\ve + \ve \beta + \mu \beta^2)M(k+2)|\nabla \psi|_{H^{k+2}}.
        \end{align*}
        
    \end{proof}

    \begin{obs}\label{obs approx of A and tildeB}
        We now construct an approximation of $\phi_b$ to the order $O(\mu(\mu \varepsilon + \ve \beta + \mu\beta^2))$. To do so, we add a term of order $\mu\ve$ in the approximation of $\phi_b$ in order to cancel the terms of order $\mu\varepsilon$. In particular, we consider $\phi_2$ solution of the problem
        \begin{align*}
        \begin{cases}
            \partial_z^2\phi_2 = - \dfrac{\zeta}{h_b}\big( 1 + \dfrac{h}{h_b} \big) \Delta_X \psi \ \ \mathrm{in} \ \ \mathcal{S}_b,\\
            \phi_2|_{z=0} = 0, \ \ \partial_z \phi_2|_{z=-1 + \beta b} =   
            0.
        \end{cases} 
    \end{align*}
    Indeed, if we use the decomposition given by Observations \eqref{B tilde op} and \eqref{A op}, and the definitions of $\phi_0$ and $\phi_1$, we get:
    \begin{align*}
        \dfrac{h}{h_b}\nabla^{\mu}_{X,z} \cdot P(\Sigma_b) \nabla^{\mu}_{X,z}(\phi_b - \phi_0 - \mu\beta\phi_1 - \mu\varepsilon\phi_2) = -\mu\varepsilon\partial_z^2\phi_2 - \mu \ve A[\nabla_X,\partial_z]\phi_0 + O(\mu^2\varepsilon),
    \end{align*}
    and
    \begin{align}\label{neumann condition for phi 1}\notag
        \frac{h}{h_b}|\mathbf{n}_b|\partial_{n_{b}}^{P_b} (\phi_b - \phi_0 - \mu\beta \phi_1 - \mu\ve\phi_2)|_{z = -h_b} = -\mu\varepsilon\partial_z\phi_2|_{z=-h_b} + O(\mu(\mu\ve + \ve \beta + \mu\beta^2)).
    \end{align}
    Moreover,  using the estimates in Proposition \ref{Prop F0} with $t_0 > \frac{d}{2}$, one can deduce from the definition of $A[\nabla_X, \partial_z] \bullet $, given by \eqref{A}, that
    \begin{align*}
        \mathrm{LHS} :
        & = 
        \| A[\nabla_X, \partial_z]\phi_0 - \dfrac{\zeta}{h_b}\big( 1 + \dfrac{h}{h_b} \big) \Delta_X \psi\|_{H^{k,0}(\mathcal{S}_b)}
        \\
        & 
        \lesssim 
        \big{|}\frac{\zeta}{h_b}(1+\frac{h}{h_b})\big{|}_{H^{\max(t_0,k)}}\|\Delta_X(\phi_0-\psi)\|_{H^{k,0}(\mathcal{S})}
        \\
        & 
        \hspace{0.5cm}
        +
        \big{|} \frac{h}{h_b}\big{|}_{H^{\max(t_0,k)}} |\nabla_X \zeta|_{H^{\max(t_0,k)}} \|\nabla_X \partial_z \phi_0\|_{H^{k,0}(\mathcal{S})}
        \\
        & 
        \hspace{0.5cm}
        +
        \|\dfrac{h}{h_b} \nabla_X \cdot \big(\dfrac{1}{\varepsilon}\nabla_X \sigma \partial_z \phi_0 \big) \|_{H^{k,0}(\mathcal{S})}
        +
        \|\partial_z\big( \dfrac{1}{\varepsilon}|\nabla_X \sigma|^2 \partial_z\phi_0 \big)\|_{H^{k,0}(\mathcal{S})}
        \\
        & \leq \mu M(k+2) |\nabla_X \psi|_{H^{k+3}},
    \end{align*}
    for any $k\in \N$.
    
    \end{obs}

    With this observation in mind, we can write the following result.

    \begin{prop}\label{Prop phi 2}
    Let $d=1,2$, $t_0>\frac{d}{2}$ and $k\in \N$. Let $\psi \in \dot{H}^{k+4}(\R^d)$. Let also $b \in C^{\infty}_c(\R^d)$ and $\zeta \in H^{{\max\{t_0 +2, k+2\}}}(\R^d)$ such that \eqref{non-cavitation 1} and \eqref{non-cavitation 2} are satisfied. If $\phi_2$ satisfies the following Laplace problem
    \begin{align*}
        \begin{cases}
            \partial_z^2\phi_2 = -\dfrac{\zeta}{h_b}\big( 1 + \dfrac{h}{h_b} \big) \Delta_X \psi \ \ \mathrm{in} \ \ \mathcal{S}_b,\\
            \phi_2|_{z=0} = 0, \ \ \partial_z \phi_2|_{z=-1 + \beta b} 
            = 0.
        \end{cases}
    \end{align*}
    Then its expression is given by:
    \begin{align*}
        \phi_2 = -(\dfrac{z^2}{2} + h_b z) \dfrac{\zeta}{h_b}\big( 1 + \dfrac{h}{h_b} \big) \Delta_{X}\psi.
    \end{align*}
    Moreover, for $\phi_b$ satisfying \eqref{Elliptic pb Sb}, $\phi_0$ given by \eqref{phi0} and $\phi_1$ given by \eqref{phi_1}, there holds
    \begin{equation}\label{Est mu2ve}
        \|\nabla^{\mu}_{X,z}(\phi_b - (\phi_0 + \mu\beta\phi_1 + \mu\ve\phi_2))\|_{H^{k,0}(\mathcal{S}_b)} \lesssim   \mu(\mu\varepsilon + \ve\beta + \mu\beta^2) M(k+2)|\nabla \psi|_{H^{k+3}}.
    \end{equation} 
    \end{prop}
    \begin{proof}
        The function $\phi_2$ satisfies a simple ODE and is solved by integrating the equation two times in $z$:
    \begin{align*}
        \phi_2 
        =
        \int_{z}^0 \int_{-1+\beta b}^{z'} 
        \dfrac{\zeta}{h_b}\big( 1 + \dfrac{h}{h_b} \big) \Delta_X \psi \: \mathrm{d}z'' \mathrm{d}z'
        =
        - (\dfrac{z^2}{2} + h_b z) \dfrac{\zeta}{h_b}\big( 1 + \dfrac{h}{h_b} \big) \Delta_X\psi.
    \end{align*}
    Then, by construction, we have that $u = \phi_b  - (\phi_0  + \mu \beta \phi_1 + \mu\ve\phi_2)$ satisfies
    \begin{equation}\label{syst f g}
		\begin{cases}
			 \frac{h}{h_b}\nabla^{\mu}_{X,z} \cdot P(\Sigma_b) \nabla_{X,z}^{\mu }u = f \quad \text{in} \quad \mathcal{S}_b
			\\	
			u|_{z = 0} = 0,  \quad \frac{h}{h_b}|n_b|\partial_{n_{b}}^{P_b} u |_{z=-h_b} = g,
		\end{cases}
    \end{equation}
    with 
    \begin{align*}
        f   = 
        -
        &\mu \ve [A[\nabla_X, \partial_z]\phi_0 - \frac{\zeta}{h_b}(1+\frac{h}{h_b})\Delta_X \psi]
        +
        \mu^2\beta^2 F
        -
        \mu^2\ve \beta A[\nabla_X,\partial_z]\phi_1\\
        -
        &\mu^2\ve (\Delta_X \phi_2 + \ve  A[\nabla_X, \partial_z]\phi_2),
    \end{align*} 
    and
    \begin{align*}
        g = -&\mu\ve\beta B[\nabla_X,\partial_z]\phi_0|_{z=-h_b} + \mu^2\beta^2 \nabla_X b \cdot \nabla_X\phi_1|_{z=-h_b} - \mu^2\ve\beta^2 B[\nabla_X,\partial_z]\phi_1|_{z=-h_b}\\
        + &\mu^2\ve\beta \nabla_X b \cdot \nabla_X\phi_2|_{z=-h_b} - \mu^2\ve^2\beta B[\nabla_X,\partial_z]\phi_2|_{z=-h_b}.
    \end{align*}
    Then we use the elliptic estimate \eqref{elliptic est} to get that
    \begin{align*}
        \|\nabla^{\mu}_{X,z} u\|_{H^{k,0}(\mathcal{S}_b)} \leq \mu(\mu\ve + \ve \beta + \mu \beta^2)M(k+1)( |g|_{H^k} + \sum\limits_{j=0}^k\|f\|_{H^{k-j,j}(\mathcal{S}_b)}),
    \end{align*} 
    with the the usual product estimates for $H^k(\R^d)$ combined with Observation \ref{obs approx of A and tildeB}, Proposition \ref{Prop phi 2} and the fact that $\phi_2$ is polynomial in $z$, we get
    \begin{align*}
        \|\nabla^{\mu}_{X,z} u\|_{H^{k,0}(\mathcal{S}_b)} \leq \mu(\mu\ve + \ve\beta + \mu\beta^2)M(k+2)|\nabla_X\psi|_{H^{s+3}}.
    \end{align*}
    
    \end{proof}

    We will now make two observations that will further simplify the presentation.
        
    \begin{obs}\label{obs second order}
    We may use Plancherel's identity and the Taylor series expansions:
    \begin{align*}
            \cosh(x) 
            & =
            1 + \frac{x^2}{2} \int_0^1\cosh(tx)(1-t) \:dt
            \\ 
            \frac{1}{\cosh(x)} & = 1 + \frac{x^2}{2} \int_0^1 \Big(\frac{\tanh(tx)^2}{\cosh(tx)}- \frac{ 1}{\cosh(tx)^3}\Big)(1-t)\:dt,
        \end{align*}
    for $x\in [0,1]$, to deduce that
    \begin{align}\label{id one}
        \|(\phi_0 - \psi) - \mu(\frac{z^2}{2}+z)|\mathrm{D}|^2\psi \|_{H^{k,0}(\mathcal{S}_b)} & \lesssim \mu^2 ||\mathrm{D}|^4 \psi|_{H^k} \leq \mu^{2} |\nabla_X \psi|_{H^{k+3}},
    \end{align}
    with $z\in (-h_b,0)$ and assumption \eqref{non-cavitation 2} on $\beta b(X)$. 
     \end{obs}

    \begin{obs}\label{obs phi app}
    From the second-order expansions given by the previous Observation \ref{obs second order} we have
        \begin{align}\label{Formula phi0}
            \phi_0 - \psi = \mu (\frac{z^2}{2} + z)  |\mathrm{D}|^2\psi + \mu^2z^2R,
        \end{align}
        where $R$ is some generic function satisfying the estimate 
        \begin{align}\label{Rest phi app}
            |R|_{H^k}\leq M(k)|\nabla_X \psi|_{H^{k+3}}.
        \end{align}
        It allows us to approximate the quantity $\phi_0 + \mu\ve\phi_2$: 
        \begin{align*}
            \phi_0 + \mu \ve \phi_2 
            & = 
            \phi_0 + \mu (\frac{z^2}{2} + h_b z) \frac{\ve \zeta }{h_b}(1 + \frac{h}{h_b}) |\mathrm{D}|^2\psi \\
            & = \phi_0 + (\phi_0 - \psi)\big( \frac{h}{h_b} - 1 \big)\big(\frac{h}{h_b} + 1\big) + \mu(\mu\ve + \ve\beta )R
            \\
            & = 
            \phi_0 + (\phi_0 - \psi)\big( \frac{h^2}{h_b^2} - 1 \big) + \mu(\mu\ve + \ve\beta )R
            \\
            & = 
            \psi + \frac{h^2}{h_b^2}(\phi_0 - \psi) + \mu(\mu\ve + \ve\beta )R.
        \end{align*}

    \end{obs}

    We can make the formal computations in Observation \ref{obs phi app} rigorous.

    \begin{prop}\label{Phi app aprox}
    Let $d=1,2$, $t_0>\frac{d}{2}$ and $k \in \N$ such that $k\geq t_0 + 1$. Let $\psi \in \dot{H}^{k+4}(\R^d)$. Let also $b \in C^{\infty}_c(\R^d)$ and  $\zeta \in H^{k+3}(\R^d)$ such that \eqref{non-cavitation 1} and \eqref{non-cavitation 2} are satisfied. Lastly, let $\phi_{\mathrm{app}}$ be defined by
    \begin{equation}\label{Phi app}
        \phi_{\mathrm{app}} 
             =
            \psi 
            +
            \Big{(} \frac{h}{h_b}\Big{)}^2 (\phi_0 - \psi)  + \mu \beta \phi_1,
    \end{equation}
    with $\phi_1$ given by \eqref{phi_1}. Then for $\phi_b$ satisfying \eqref{Elliptic pb Sb} there holds,
    \begin{align}\label{Est app}
        \|\nabla^{\mu}_{X,z}(\phi_b - \phi_{\mathrm{app}})\|_{H^{k,0}(\mathcal{S}_b)} \lesssim   \mu(\mu\varepsilon + \ve\beta + \mu\beta^2) M(s+2)|\nabla \psi|_{H^{k+3}}.
    \end{align}
        
    \end{prop}

    \begin{proof} 
        We first use Proposition \ref{Prop phi 2} to get the estimate
        \begin{align*}
            \|\nabla^{\mu}_{X,z}(\phi_b - \phi_{\mathrm{app}})\|_{H^{k,0}(\mathcal{S}_b)} 
            & \lesssim
            \mu(\mu\varepsilon + \ve\beta + \mu\beta^2)M(k+2)|\nabla \psi|_{H^{k+3}} 
            \\
            & 
            \hspace{0.5cm}
            + \|\nabla^{\mu}_{X,z}(\phi_0 + \mu\varepsilon\phi_1 - \phi_{\mathrm{app}})\|_{H^{k,0}(\mathcal{S}_b)}.
        \end{align*}
        Making the same approximations as in Observation \ref{obs phi app} will complete the proof. In particular, accounting for the loss of derivatives given by \eqref{Rest phi app} yields,
        \begin{align*}
            \|\nabla^{\mu}_{X,z}(\phi_0 + \mu\beta\phi_1 + \mu \ve \phi_2 - \phi_{\mathrm{app}})\|_{H^{k,0}(\mathcal{S})} \lesssim \mu(\mu\ve + \ve\beta ) M(k+1)|\nabla \psi|_{H^{k+3}}.
        \end{align*}
        Gathering these estimates concludes the proof.\color{black}
    \end{proof}

    \subsection{Multi-scale expansions of $\overline{V}$}\label{mult V bar} In this subsection we will use the expression of $\phi, \phi_0$, $\phi_1$, and $\phi_{\mathrm{app}}$ to construct approximations of $\overline{V}$. The first result is given in the following  proposition. 
    \begin{prop}\label{prop V[0,b]}
        Let $d=1,2$, $t_0>\frac{d}{2}$ and $s\geq 0$. Let $b \in C^{\infty}_c(\R^d)$, $\psi \in \dot{H}^{k+3}(\R^d)$ and $\zeta \in H^{\max\{t_0+2,s+3\}}(\R^d)$ be such that \eqref{non-cavitation 1} is satisfied. Then for $\phi$ defined by the solution of \eqref{elliptic problem phi} and $\overline{V}[0, \beta b]$ defined by
        \begin{equation}\label{V[0,b]}
            \overline{V}[0, \beta b] = \frac{1}{h_b} \int_{-h_b}^0 \nabla_X\phi \: \mathrm{d}z
        \end{equation}
        there holds,
        \begin{equation}
            |\overline{V} - \overline{V}[0, \beta b]|_{H^s} \leq \mu \ve |\nabla_X \psi|_{H^{s+3}}
        \end{equation}

    \end{prop}
    \begin{proof} We will first prove the estimate on $\overline{V} - \overline{V}[0, \beta b]$ for $k\in \N$, and then use interpolation for $s\geq 0$. By definition \eqref{New V bar} and \eqref{V[0,b]} we have that
    \begin{align*}
            |\overline{V} - \overline{V}[0, \beta b]|_{H^s}
            & 
            =
            \Big{|} 
		\int_{-1+\beta b}^{0} \big{[}\frac{1}{h_b}\nabla_X( \phi_b - \phi)- \frac{1}{h}(\varepsilon\nabla_X\Big(\dfrac{\zeta}{h_b}\Big)z + \varepsilon\nabla_X\zeta) \partial_z \phi_b \big{ ]}\: \mathrm{d}z \Big{|}_{H^k}
          .
    \end{align*}
    Now, note that $h_b$ and $h$ are only functions of $X$ and satisfies \eqref{non-cavitation 1} and \eqref{non-cavitation 2}, we can therefore use \eqref{Classical prod est}, \eqref{prod est division}, and \eqref{estimate for V bar} to get that
        \begin{align*}
            |\overline{V} - \overline{V}[0, \beta b]|_{H^k}
            & 
            \lesssim \Big{|} \frac{1}{h_b}\int_{-1+\beta b}^{0}\nabla_X( \phi_b - \phi ) \: \mathrm{d}z\Big{|}_{H^k}
            +
            \ve \Big{|}\frac{1}{h}\nabla_X\Big(\dfrac{\zeta}{h_b}\Big)\int_{-1+\beta b}^{0} z\partial_z \phi_b\: \mathrm{d}z \Big{|}_{H^k}
            \\
            & \hspace{0.5cm}
            +
            \ve \Big{|}\frac{1}{h}\nabla_X\zeta\int_{-1+\beta b}^{0} \partial_z \phi_b\: \mathrm{d}z \Big{|}_{H^k} 
            \\
            &
            \leq M(k)
            \|\nabla_{X,z}^{\mu} (\phi_b - \phi)\|_{H^{k+1,0}(\mathcal{S}_b)} 
            + 
            \ve M(k+1)
            \|\partial_z \phi_b\|_{H^{k,0}(\mathcal{S}_b)}
            \\
            & 
            \hspace{0.5cm}
            +
            M(k)\sum\limits_{j=1}^k\|\nabla_{X,z}^{\mu}\partial_z^{j-1}(\phi_b - \phi )\|_{H^{k-j+1,0}(\mathcal{S}_b)} 
            \\
            & 
            \hspace{0.5cm}
            +
            \ve M(k+1)
            \sum\limits_{j=1}^k\|\partial_z^{j+1} \phi_b\|_{H^{k-j,0}(\mathcal{S}_b)}
            \\
            & = 
            J_1 + J_2 + J_3 + J_4 .
        \end{align*}
        We will now estimate each term. To estimate $J_1$, we apply \eqref{est phi} to get that
        \begin{align*}
            J_1 \leq  \mu \ve M(k+3)|\nabla_X \psi |_{H^{k+4}}.
        \end{align*}
        To estimate $J_2$, we use Proposition \ref{Prop F0} to see that $|\partial_z \phi_0|_{H^{k}} \lesssim  \mu|\nabla_X \psi|_{H^{k+1}}$ and combine it with  \eqref{est phi0} to get the estimate,
        \begin{align*}
            J_2 
            & \leq
            \ve M(k+1)(\|\nabla^{\mu}_{X,z} (\phi_b-\phi_0)\|_{H^{k+1,0}}+\| \partial_z \phi_0\|_{H^{k,0}})
            \\
            & \leq 
            \mu\ve M(s+3) |\nabla_X \psi|_{H^{k+3}}
            .
        \end{align*}
        Finally, we will deal with $J_3$ and $J_4$. To that end, we need to trade the derivatives in $\partial_z$ with derivatives in the horizontal variable by relating the functions with an elliptic problem. We introduce the notation
        \begin{equation}\label{useful notation}
            f \sim g \iff f(X,z) = r(X)g(X,z),
        \end{equation}
        with $r \in H^k(\R^d)$ such that $|r|_{H^{k}}\leq  M(k+1)$. Then, by construction, we have from \eqref{A} that
        \begin{align*}
            (1+ \mu|\nabla_X\sigma|^2)\partial_z^2 \phi_b 
            & =
            -\mu
            \Delta_X \phi_b - \mu \ve (A[\nabla_X, \partial_z]\phi_b- \frac{1}{\ve}|\nabla_X\sigma|^2\partial_z^2 \phi_b)
            \\
            & =:
            -\mu
            \Delta_X \phi_b - \mu \ve \tilde{A}[\nabla_X, \partial_z]\phi_b,
        \end{align*}
        where $\nabla_X\sigma$ is given by \eqref{grad sigma} and is of the form
        \begin{equation*}
            \nabla_X \sigma \sim  \ve (1+z),
        \end{equation*}
        while $\tilde{A}[\nabla_X, \partial_z]$ is of the form
        \begin{equation*}
            \tilde{A}[\nabla_X, \partial_z] \phi_b \sim \Delta_X \phi_b + (1+z)\nabla_Xf\cdot \nabla_X \partial_z \phi_b + z\partial_z \phi_b,
        \end{equation*}
        for some function $f \in H^{k+3}(\R^d)$. Similarly, for $\phi$ defined by \eqref{elliptic problem phi}, we have the relation
        \begin{align*}
            (1+ \mu|\nabla_X\sigma|^2)\partial_z^2 (\phi_b -\phi)
            & = 
            \mu \Delta_X (\phi_b - \phi) 
            -
            \mu \ve \tilde{A}[\nabla_X, \partial_z](\phi_b -\phi)
            -
            \mu \ve \tilde{A}[\nabla_X, \partial_z]\phi
            \\
            & 
            \hspace{0.5cm}
            -
            \mu|\nabla_X\sigma|^2\partial_z^2\phi.
        \end{align*}
        Consequently, we can trade two derivatives in $z$ by $\Delta_X$, $\nabla_X \partial_z$, and $\partial_z$. From that point, we can deduce that for $k\geq 3$, we have
        \begin{align*}
            \partial_z^k (\phi_b -\phi) 
            \sim \mu \sum \limits_{\gamma\in \N^d \: |\gamma|\leq k -1} \partial_X^{\gamma}\partial_z\big{(}(\phi_b - \phi) - \ve \phi \big{)} + \sum\limits_{j = 1}^k\mu\ve^2\partial_z^j\phi,
        \end{align*}
        For the last term we can use that $\partial_z^2\phi = -\mu \Delta_X \phi$. From these relations, and the control of the residual terms $r(X)$ in \eqref{useful notation} with the product estimate \eqref{Classical prod est}, we may conclude from  \eqref{est phi}, \eqref{est on phi in pf}, and \eqref{Poincare 1} that
        \begin{align*}
            J_3
            & \leq M(k+1)(\|\nabla_{X,z}^{\mu}(\phi_b - \phi)\|_{H^{k+1,0}(\mathcal{S}_b)} + \mu
            \ve|\nabla_{X} \psi|_{H^{k+1}})
            \\
            & \leq \mu( \ve + \ve \beta +\mu \beta^2) M(k+3)|\nabla_X \psi|_{H^{k+3}}.
        \end{align*}
        To conclude, we estimate $J_4$. Since there is an $\varepsilon$ appearing we only need to introduce $\phi_0$ and we obtain
        \begin{align*}
            J_4 
            & =
            \ve M(k+1)
            \sum\limits_{j=1}^k\Big{(}\|\partial_z^{j+1} (\phi_b-\phi_0)\|_{H^{k-j,0}(\mathcal{S}_b)} + \|\partial_z^{j+1} \phi_0\|_{H^{k-j,0}(\mathcal{S}_b)}\Big{)}
            \\
            & \leq \ve (\mu \ve + \mu \beta + \mu) M(k+3)|\nabla_X \psi|_{H^{k+3}}.
        \end{align*}

    \end{proof}

    The next result concerns the expansion of $\overline{V}$ with respect to $\mu \beta$:
    \color{black}

    \begin{prop}\label{Prop V}
        Let $d=1,2$, $t_0>\frac{d}{2}$ and $s\geq 0$. Let $b \in C^{\infty}_c(\R^d)$ and $\zeta \in H^{\max\{t_0+2,s+3\}}(\R^d)$ be such that \eqref{non-cavitation 1} and \eqref{non-cavitation 2} are satisfied. Let $\mathcal{L}_{1}^{\mu}[\beta b]$ and $\mathcal{L}_{2}^{\mu}[\beta b]$  be two pseudo-differential operators defined by
    \begin{align*}
        \mathcal{L}_{1}^{\mu}[\beta b]  
        & =  
        -
        \frac{1}{\beta}\sinh{(\beta b(X) \sqrt{\mu}|\mathrm{D}|)}\mathrm{sech}(\sqrt{\mu}|\mathrm{D}|)\dfrac{1}{\sqrt{\mu}|\mathrm{D}|}
        \\
        \mathcal{L}_{2}^{\mu}[\beta b] & = -(\mathcal{L}_{1}^{\mu}[\beta b] +  b) \frac{1}{\mu |\mathrm{D}|^2}.
    \end{align*}
    Let also $\mathrm{F}_1$, $\mathrm{F}_2$,  $\mathrm{F}_3$, and $\mathrm{F}_4$ be four Fourier multipliers defined by
    \begin{align*}
            \mathrm{F}_1 = \dfrac{\tanh{(\sqrt{\mu}|\mathrm{D}|)}}{\sqrt{\mu}|\mathrm{D}|},\quad \mathrm{F}_2 = \frac{3}{\mu |\mathrm{D}|^2}(1- \mathrm{F}_1), \quad  \mathrm{F}_3 =\mathrm{sech}(\sqrt{\mu}|D|), \quad \mathrm{F}_4 =  \frac{2}{\mu |\mathrm{D}|^2}(1- \mathrm{F}_3).
    \end{align*}
    Let $\psi \in \dot{H}^{s+5}(\R^d)$ and consider the approximation:
    \begin{align}\label{V0}
        \overline{V}_0 
            & =
            \frac{1}{h_b}\mathrm{F}_1 \nabla_X \psi 
            +
            \frac{\beta}{h_b}\mathcal{L}_1^{\mu}[\beta b] \nabla_X \psi
            +
            \frac{\mu \beta}{2} \nabla_X\mathrm{F}_4\nabla_X \cdot \big(\mathcal{L}_1^{\mu}[\beta b] \nabla_X \psi\big)
            \\ 
            & 
            \hspace{0.5cm}\notag
            -
            \frac{\mu\beta^2}{2}\nabla_X \big{(}b\mathrm{F}_4\nabla_X\cdot(b\nabla_X \psi)\big{)}
            -
            \mu \beta^2 (\nabla_Xb)\mathrm{F}_1\nabla_X\cdot(b\nabla_X \psi).
    \end{align}
    Then for $\overline{V}$ defined by \eqref{V bar}, there holds
    \begin{equation}\label{V0 est}
        |\overline{V} -  \overline{V}_0|_{H^s} \leq (\mu \ve + \mu^2 \beta^2) M(s+3)|\nabla_X \psi|_{H^{s+4}}.
    \end{equation}       
    Furthermore, let $\overline{V}_{\mathrm{app}}$ be defined by the approximation:
    \begin{align}\label{V app}
            \overline{V}_{\mathrm{app}}  
            & = 
            \nabla_X \psi 
           +
           \frac{\mu}{h}\nabla_X 
           \Big{(}
            \frac{h^3}{h_b^3} \mathrm{F}_2 \psi
           \Big{)}
           +
           \frac{\mu\beta }{h}
           \nabla_X \Big{(} \frac{h^3}{h_b^3} \mathcal{L}_2^{\mu}[\beta b] \psi\Big{)}
           +
           \frac{\mu \beta}{2} \nabla_X\mathrm{F}_4\nabla_X \cdot \big(\mathcal{L}_1^{\mu}[\beta b] \nabla_X \psi\big)
            \\ 
            & 
            \hspace{0.5cm}\notag
            -
            \frac{\mu\beta^2}{2}\nabla_X \big{(}b\mathrm{F}_4\nabla_X\cdot(b\nabla_X \psi)\big{)}
            -
            \mu \beta^2 (\nabla_Xb)\mathrm{F}_1\nabla_X\cdot(b\nabla_X \psi)
        \end{align}
        Then there holds
        \begin{equation}\label{V app est}
            |\overline{V}- \overline{V}_{\mathrm{app}} |_{H^s} \leq (\mu^2\ve +\mu\ve\beta + \mu^2\beta^2) M(s+3) |\nabla_X \psi|_{H^{s+4}}.
        \end{equation}

    \end{prop}

    \begin{proof}
        We give the proof in four steps. \\

        \noindent
        \underline{Step 1.} \textit{Construction of $\overline{V}_0$.} To construct $\overline{V}_0$, we use the solution of $\phi_0$ given by \eqref{phi0}, the solution $\phi_1$ given by \eqref{phi_1},  and formula \eqref{New V bar}, formally discarding terms of order $\mu \ve$, to get that 
        \begin{align*}
            h_b \overline{V}_0 
            &
            =
            \int_{-1 + \beta b(X)}^0 \nabla_X\phi_0 \: \mathrm{d}z +\mu \beta \int_{-1 + \beta b(X)}^0 \nabla_X\phi_1 \: \mathrm{d}z 
            \\
            &  =
            I_1 + I_2. 
        \end{align*}
        Then by direct computations, we get
        \begin{align*}
            I_1 & =
            \mathcal{F}^{-1}\Big{(} \int_{-1 + \beta b(X)}^0 \cosh{((z+1)\sqrt{\mu}|\xi|)}\mathrm{sech}{(\sqrt{\mu}|\xi|)} \: i \xi \hat{\psi}(\xi)\:  \mathrm{d}z\Big{)}(X)
            \\ 
            & =
            \dfrac{\tanh{(\sqrt{\mu}|\mathrm{D}|)}}{\sqrt{\mu}|\mathrm{D}|}\nabla_X\psi - \sinh{(\beta b(X)\sqrt{\mu}|\mathrm{D}|)}\mathrm{sech}(\sqrt{\mu}|\mathrm{D}|)\dfrac{1}{\sqrt{\mu}|\mathrm{D}|}\nabla_X\psi.
        \end{align*}
        While for $I_2$, we simplify the notation by defining $G = \nabla_X \cdot \mathcal{L}_{1}^{\mu}[\beta b]\nabla_X\psi$ and then make the observation
        \begin{align*}
            \mu \beta \int_{-h_b}^0 \nabla_X \phi_1 (X,z) \mathrm{d}z = \mu \beta h_b \int_{-1}^0 (\nabla_X \phi_1)(X, h_b z) \mathrm{d}z.
        \end{align*}
        Then by the chain rule, we have the relation
        \begin{align*}
            \nabla_X (\phi_1(X,h_b z)) = (\nabla_X \phi_1)(X,h_b z) - \beta \nabla_X b (\partial_z \phi_1)(X,h_b z),
        \end{align*}
        and
        \begin{align*}
            \partial_z(\phi_1(X,h_b z)) = h_b (\partial_z \phi_1)(X,h_b z),
        \end{align*}
        from which we obtain
        \begin{align*}
            I_2 
            & = 
            \mu\beta \int_{-h_b}^0 \nabla_X \phi_1 (X,z) \mathrm{d}z 
            \\
            & =
            \mu \beta h_b \int_{-1}^0 \nabla_X(\phi_1(X,h_b z)) \mathrm{d}z 
            +
            \mu\beta^2 (\nabla_X b) \int_{-1}^0 \partial_z (\phi_1(X,h_b z)) \mathrm{d}z
            \\ 
            & 
            =
            \mu \beta h_b \nabla_X \big(h_b (1-\mathrm{sech}(\sqrt{\mu}|\mathrm{D}|))\frac{1}{\mu|\mathrm{D}|^2}G\big) 
            -
            \mu\beta^2 (\nabla_X b) h_b\frac{\tanh(\sqrt{\mu}|\mathrm{D}|)}{\sqrt{\mu}|\mathrm{D}|} G
            .
        \end{align*}
        Adding these computations yields,
        \begin{align*}
            \overline{V}_0 &= 
            \frac{1}{h_b}\mathrm{F}_1 \nabla_X \psi 
            +
            \frac{\beta}{h_b}\mathcal{L}_1^{\mu}[\beta b] \nabla_X \psi 
            +
            \frac{\mu \beta }{2}\nabla_X(h_b \mathrm{F}_4G) 
            -
            \mu\beta^2 (\nabla_X b) \mathrm{F}_1 G,
        \end{align*}
        To conclude this step, we use \eqref{L approx} to approximate $G = \nabla_X\cdot(b\nabla_X \psi) + \mu R_1$, where $|R_1|_{H^k}\leq M(k+1)|\nabla_X \psi |_{H^{k+4}}$. Then we have constructed the approximation:
        \begin{align*}
            \overline{V}_0 
            & =
            \frac{1}{h_b}\mathrm{F}_1 \nabla_X \psi 
            +
            \frac{\beta}{h_b}\mathcal{L}_1^{\mu}[\beta b] \nabla_X \psi
            +
            \frac{\mu \beta}{2} \nabla_X\mathrm{F}_4\nabla_X \cdot \mathcal{L}_1^{\mu}[\beta b] \nabla_X \psi
            \\ 
            & 
            \hspace{0.5cm}
            -
            \frac{\mu\beta^2}{2}\nabla_X \big{(}b\mathrm{F}_4\nabla_X\cdot(b\nabla_X \psi)\big{)}
            - 
            \mu \beta^2 (\nabla_Xb)\mathrm{F_1}\nabla_X\cdot(b\nabla_X \psi) + \mu^2\beta^2 R_2.
        \end{align*}
        For some $R_2$ satisfying $|R_2|_{H^k} \leq M(k+1)|\nabla_X \psi|_{H^{k+4}}$.
        \\

        \noindent
        \underline{Step 2.} We will now prove the estimate on $\overline{V} - \overline{V}_0$, where we argue as in the proof of Proposition \ref{prop V[0,b]}. In particular, we do the estimates for $k\in \N$, and then use interpolation for $s\geq 0$. Also, we define the approximation 
        $$\phi_{\mathrm{app}}^1 = \phi_0+\mu\beta \phi_1,$$
        and let $R$ be the function constructed in the previous step satisfying estimate $|R|_{H^k} \leq M(k+1) |\nabla_X\psi|_{H^{k+4}}$. Then we have that
        \begin{align*}
            |\overline{V} - \overline{V}_0 |_{H^k}
            & 
            =
            \Big{|} 
		\int_{-1+\beta b}^{0} \big{[}\frac{1}{h_b}\nabla_X( \phi_b - \phi_{\mathrm{app}}^1)- \frac{1}{h}(\varepsilon\nabla_X\Big(\dfrac{\zeta}{h_b}\Big)z + \varepsilon\nabla_X\zeta) \partial_z \phi_b \big{ ]}\: \mathrm{d}z \Big{|}_{H^k}
            + \mu^2 \beta^2 |R|_{H^k}.
        \end{align*}
        We now use \eqref{Classical prod est}, \eqref{prod est division} and \eqref{estimate for V bar} to obtain
        \begin{align*}
            |\overline{V} - \overline{V}_0 |_{H^k}
            & 
            \lesssim \Big{|} \frac{1}{h_b}\int_{-1+\beta b}^{0}\nabla_X( \phi_b - \phi_{\mathrm{app}}^1 ) \: \mathrm{d}z\Big{|}_{H^k}
            +
            \ve \Big{|}\frac{1}{h}\nabla_X\Big(\dfrac{\zeta}{h_b}\Big)\int_{-1+\beta b}^{0} z\partial_z \phi_b\: \mathrm{d}z \Big{|}_{H^k}
            \\
            & \hspace{0.5cm}
            +
            \ve \Big{|}\frac{1}{h}\nabla_X\zeta\int_{-1+\beta b}^{0} \partial_z \phi_b\: \mathrm{d}z \Big{|}_{H^k} +  \mu^2 \beta^2 |R|_{H^k}
            \\
            &
            \leq M(k)
            \|\nabla_{X,z}^{\mu} (\phi_b - \phi_{\mathrm{app}}^1 )\|_{H^{k+1,0}(\mathcal{S}_b)} 
            + 
            \ve M(k+1)
            \|\partial_z \phi_b\|_{H^{k,0}(\mathcal{S}_b)}
            \\
            & 
            \hspace{0.5cm}
            +
            M(k)\sum\limits_{j=1}^k\|\nabla_{X,z}^{\mu}\partial_z^{j-1}(\phi_b - \phi_{\mathrm{app}}^1 )\|_{H^{k-j+1,0}(\mathcal{S}_b)} 
            + 
            \ve M(k+1)
            \sum\limits_{j=1}^k\|\partial_z^{j+1} \phi_b\|_{H^{k-j,0}(\mathcal{S}_b)}
            \\
            & 
            \hspace{0.5cm}
            +  \mu^2 \beta^2 M(k+1) |\nabla_X \psi|_{H^{k+4}}
            \\
            & = 
            II_1 + II_2 + II_3 + II_4  + II_5.
        \end{align*}
        We will now estimate each term. To estimate $II_1$, we apply \eqref{Est muve mu2beta2} to get that
        \begin{align*}
            II_1 \leq  \mu (\ve + \ve \beta + \mu \beta^2) M(k+3)|\nabla_X \psi |_{H^{k+4}}.
        \end{align*}
        The estimate on $II_2$, is the same as for $J_2$ in the proof of Proposition \ref{prop V[0,b]}: 
        \begin{align*}
            II_2 
            & \leq
            \ve M(k+1)(\|\nabla^{\mu}_{X,z} (\phi_b-\phi_0)\|_{H^{k+1,0}}+\| \partial_z \phi_0\|_{H^{k,0}})
            \\
            & \leq 
            \mu\ve M(s+3) |\nabla_X \psi|_{H^{k+3}}
            .
        \end{align*}
        Lastly, the estimates on $II_3$ and $II_4$ are similar to the estimates on $J_3$ and $J_4$ in the proof of Proposition \ref{prop V[0,b]}. In particular,  we trade the derivatives in $\partial_z$ with derivatives in the horizontal variable by relating the functions with an elliptic problem. We recall the notation \eqref{useful notation}:
        \begin{equation*}
            f \sim g \iff f(X,z) = r(X)g(X,z),
        \end{equation*}
        with $r \in H^k(\R^d)$ such that $|r|_{H^{k}}\leq  M(k+1)$. Then
        for $\phi_{\mathrm{app}}^1 =\phi_0+ \mu \beta\phi_1$ defined by \eqref{elliptic problem phi 0} and \eqref{System phi1}, we have the relation
        \begin{align*}
            (1+ \mu|\nabla_X\sigma|^2)\partial_z^2 (\phi_b -\phi_{\mathrm{app}}^1)
            & = 
            \mu \Delta_X (\phi_b - \phi_{\mathrm{app}}^1) 
            -
            \mu \ve \tilde{A}[\nabla_X, \partial_z](\phi_b -\phi_{\mathrm{app}}^1)
            -
            \mu \ve \tilde{A}[\nabla_X, \partial_z]\phi_{\mathrm{app}}^1
            \\
            & 
            \hspace{0.5cm}
            -
            \mu|\nabla_X\sigma|^2\partial_z^2\phi_{\mathrm{app}}^1 + \mu^2 \beta^2 F.
        \end{align*}
        where $F$ is some function satisfying \eqref{est on Source 1} and goes into the rest. Consequently, we can trade two derivatives in $z$ by $\Delta_X$, $\nabla_X \partial_z$, and $\partial_z$. From that point, we can deduce that for $k\geq 3$, we have
        \begin{align*}
            \partial_z^k (\phi_b -\phi_{\mathrm{app}}^1) 
            \sim \mu \sum \limits_{\gamma\in \N^d \: |\gamma|\leq k -1} \partial_X^{\gamma}\partial_z\big{(}(\phi_b - \phi_{\mathrm{app}}^1) - \ve \phi_{\mathrm{app}}^1\big{)} + \sum\limits_{j = 1}^k\mu\ve^2\partial_z^j\phi_{\mathrm{app}}^1,
        \end{align*}
        From this estimate, where we control the residual terms $r(X)$ in \eqref{useful notation} with the product estimate \eqref{Classical prod est}, then combine it with \eqref{Est muve mu2beta2} and \eqref{Poincare 1} to get
        \begin{align*}
            II_3
            & \leq M(k+1)(\|\nabla_{X,z}^{\mu}(\phi_b - \phi_{\mathrm{app}}^1)\|_{H^{k+1,0}(\mathcal{S}_b)} + \mu
            \ve|\nabla_{X} \psi|_{H^{k+1}})
            \\
            & \leq \mu( \ve + \ve \beta +\mu \beta^2) M(k+3)|\nabla_X \psi|_{H^{k+3}}.
        \end{align*}
        To conclude, we need an estimate on $II_4$. But since $II_4 = J_4$, we have that
        \begin{align*}
            II_4 \leq \ve (\mu \ve + \mu \beta + \mu) M(k+3)|\nabla_X \psi|_{H^{k+3}}.
        \end{align*}\\

        \noindent
        \underline{Step 3.} \textit{Construction of $\overline{V}_{\mathrm{app}}$.} 
        The next step is to construct $\overline{V}_{\mathrm{app}}$ by replacing $\phi_b$ with $\phi_{\mathrm{app}}$ in \eqref{New V bar}:
        \begin{align}\label{V app integral}
            \overline{V}_{\mathrm{app}}  = 
		\int_{-1+\beta b}^{0} \big{[}\frac{1}{h_b}\nabla_X \phi_{\mathrm{app}} 
            -
            \frac{1}{h}(\varepsilon\nabla_X\Big(\dfrac{\zeta}{h_b}\Big)z + \varepsilon\nabla_X\zeta )\partial_z \phi_{\mathrm{app}} \big{]}\: \mathrm{d}z.
        \end{align}
        Then using \eqref{Phi app}, we obtain that
        \begin{align*}
            \overline{V}_{\mathrm{app}}
            &=
            \int_{-1+\beta b}^0 \dfrac{1}{h_b}\nabla_X\psi \: \mathrm{d}z 
            +
            \int_{-1+\beta b}^0 \dfrac{1}{h_b}\nabla_X\Big( \dfrac{h^2}{h_b^2}(\phi_0-\psi) \Big) \: \mathrm{d}z
            \\
            &\hspace{0.5cm} 
            -
            \int_{-1+\beta b}^0\dfrac{1}{h}(z\varepsilon\nabla_X(\dfrac{\zeta}{h_b}) + \varepsilon\nabla_X\zeta)
            \partial_z \Big( \dfrac{h^2}{h_b^2} (\phi_0-\psi) \Big)\: \mathrm{d}z 
            \\
            & 
            \hspace{0.5cm}
            + \mu\beta \frac{1}{h_b} \int_{-1+\beta b}^0 \nabla_X \phi_1 \mathrm{d}z 
            -
            \mu\ve\beta 
            \int_{-1+\beta b}^0\dfrac{1}{h}(z\nabla_X(\dfrac{\zeta}{h_b}) + \nabla_X\zeta)
            \partial_z \phi_1\: \mathrm{d}z 
            \\
            & = 
            III_1 + III_2 + III_3 + III_4 + III_5.
        \end{align*}
        Clearly, $III_1 = \nabla_X \psi$ and to compute $III_2+III_3$ we use formula \eqref{phi0} for $\phi_0$:
        \begin{align*}
            III_2 + III_3
            & 
            =
            \dfrac{1}{h} \nabla_X \Big( \dfrac{h^3}{h_b^3}\dfrac{\tanh{(\sqrt{\mu}|\mathrm{D}|)}}{\sqrt{\mu}|\mathrm{D}|} 
           \psi \Big) 
            \\
            & 
            \hspace{0.5cm}
            -
            \dfrac{1}{h} \nabla_X \Big( \dfrac{h^3}{h_b^3} \Big(\sinh{(\beta b(X)\sqrt{\mu}|\mathrm{D}|)}\mathrm{sech}{(\sqrt{\mu}|\mathrm{D}|)}\dfrac{1}{\sqrt{\mu}|\mathrm{D}|}\psi -(-1+\beta b) \psi \Big)\Big)
            \\
            & 
            \hspace{0.5cm}
            -
            \varepsilon\beta \zeta h \dfrac{\nabla_X  b}{h_b^3}\Big(\cosh{(\beta b(X)\sqrt{\mu}|\mathrm{D}|)}\mathrm{sech}{(\sqrt{\mu}|\mathrm{D}|)}-1\Big)\psi
            \\
            & = 
            \dfrac{\mu}{3h} \nabla_X \Big( \dfrac{h^3}{h_b^3}\dfrac{3}{\mu |\mathrm{D}|^2}\Big( 1-\dfrac{\tanh{(\sqrt{\mu}|\mathrm{D}|)}}{\sqrt{\mu}|\mathrm{D}|} 
            \Big)\Delta_X \psi \Big) 
            \\
            & 
            \hspace{0.5cm}
            -
            \dfrac{1}{h} \nabla_X \Big( \dfrac{h^3}{h_b^3} \Big(\sinh{(\beta b(X)\sqrt{\mu}|\mathrm{D}|)}\mathrm{sech}{(\sqrt{\mu}|\mathrm{D}|)}\dfrac{1}{\sqrt{\mu}|\mathrm{D}|}\psi -\beta b \psi \Big)\Big)
            + 
            \mu\ve\beta R_5,
        \end{align*}
        where $R_5$ is given by
        \begin{align*}
            R_5
            & = 
            -\zeta h \dfrac{\nabla_X  b}{h_b^3}\Big(\cosh{(\beta b(X)\sqrt{\mu}|\mathrm{D}|)}\mathrm{sech}{(\sqrt{\mu}|\mathrm{D}|)}-1\Big)\frac{1}{\mu|\mathrm{D}|^2}\Delta_X\psi.
        \end{align*} 
       Moreover, using the algebra property of the Sobolev spaces \eqref{Classical prod est}, \eqref{prod est division}, and estimate \eqref{L3}, we have that 
       \begin{equation}\label{Rest R}
           |R_5|_{H^k}\leq M(k+1)|\nabla_X \psi|_{H^{k+1}}.
       \end{equation}
       Next, we see that $III_4$ is already treated in Step 1. and satisfies:
       \begin{align*}
            III_4  
            & = 
            \frac{\mu \beta}{2} \nabla_X\mathrm{F}_4\nabla_X \cdot \mathcal{L}_1^{\mu}[\beta b] \nabla_X \psi
            -
            \frac{\mu\beta^2}{2}\nabla_X \big{(}b\mathrm{F}_4\nabla_X\cdot(b\nabla_X \psi)\big{)}
            -
            \mu \beta^2 (\nabla_Xb)\mathrm{F}_1\nabla_X\cdot(b\nabla_X \psi)
            \\
            &
            \hspace{0.5cm}
            + 
            \mu^2 \beta^2 R_6,
       \end{align*}
       for some function $R_6$ satisfying $|R_6|_{H^k} \leq M(k+1) |\nabla_X\psi|_{H^{k+4}}$. Lastly, for the term $III_5$, we use integration by parts to find the expressions
       \begin{align*}
           III_5 
           & =
           \mu\ve\beta 
            \int_{-1+\beta b}^0\dfrac{1}{h}(z\nabla_X(\dfrac{\zeta}{h_b}) + \nabla_X\zeta)
            \partial_z \phi_1\: \mathrm{d}z
            \\ 
            & = 
            -\mu\ve\beta \frac{h_b^2}{h}  \nabla_X(\frac{\zeta}{h_b} )
            \Big{(}
          \frac{1}{2}\mathrm{F}_4\nabla_X \cdot(\mathcal{L}_1^{\mu}[\beta b] \nabla_X \psi)
            +\frac{\tanh(\sqrt{\mu}|\mathrm{D}|)}{\sqrt{\mu}|\mathrm{D}|}\nabla_X \cdot(\mathcal{L}_1^{\mu}[\beta b] \nabla_X \psi)\Big{)}
            \\ 
            & 
            \hspace{0.5cm}
            -
            \mu\ve\beta \frac{h_b}{h}\nabla_X \zeta \frac{\tanh(\sqrt{\mu}|\mathrm{D}|)}{\sqrt{\mu}|\mathrm{D}|} \nabla_X \cdot(\mathcal{L}_1^{\mu}[\beta b] \nabla_X \psi)
       \end{align*}
       The multipliers are bounded on $H^{k}(\R^d)$ and combined with Proposition \ref{Prop op} we get that
       \begin{align*}
           |III_5|_{H^k} \leq \mu \ve \beta M(k+1)|\nabla_X \psi|_{H^{k+1}}.
       \end{align*}
       
       Adding these identities in the definition of $\overline{V}_{\mathrm{app}}$ we get that
       \begin{align*}
           \overline{V}_{\mathrm{app}}
           & = 
           \int_{-1+\beta b}^{0} \big{[}\frac{1}{h_b}\nabla_X \phi_{\mathrm{app}} - \dfrac{1}{h}(z\varepsilon\nabla_X(\dfrac{\zeta}{h_b}) + \varepsilon\nabla_X\zeta) \partial_z \phi_{\mathrm{app}} \big{]}\: \mathrm{d}z + \mu \ve \beta R_7
           \\ 
           & = 
           \nabla_X \psi 
           +
           \frac{\mu}{h}\nabla_X 
           \Big{(}
            \frac{h^3}{h_b^3} \mathrm{F}_2 \psi
           \Big{)}
           +
           \frac{\mu\beta }{h}
           \nabla_X \Big{(} \frac{h^3}{h_b^3} \mathcal{L}_2^{\mu}[\beta b] \psi\Big{)}
           +
           \frac{\mu  \beta}{2}\nabla_X\mathrm{F}_4\nabla_X \cdot \big(\mathcal{L}_1^{\mu}[\beta b] \nabla_X \psi\big)
           \\ 
           & 
           \hspace{0.5cm}
            -
            \frac{\mu\beta^2}{2}\nabla_X \big{(}b\mathrm{F}_4\nabla_X\cdot(b\nabla_X \psi)\big{)}
            -
            \frac{\mu \beta^2}{2} (\nabla_Xb)\mathrm{F}_1\nabla_X\cdot(b\nabla_X \psi),
       \end{align*}
       where $R_7$ is some generic function satisfying $|R_7|_{H^k} \leq M(k+1) |\nabla_X \psi|_{H^{k+4}}$. \\

         \noindent
        \underline{Step 4.} \textit{Proof of \eqref{V app est}.} We use the definition \eqref{V app integral} of $\overline{V}_{\mathrm{app}}$ and  \eqref{estimate for V bar} to identify the terms
        \begin{align*}
            |\overline{V} - \overline{V}_{\mathrm{app}} |_{H^k}
            & 
            =
            \Big{|} 
		\int_{-1+\beta b}^{0} \big{[}\frac{1}{h_b}\nabla_X( \phi_b - \phi_{\mathrm{app}})- \frac{1}{h}(\varepsilon\nabla_X\Big(\dfrac{\zeta}{h_b}\Big)z + \varepsilon\nabla_X\zeta) \partial_z  (\phi_b-\phi_{\mathrm{app}}) \big{ ]}\: \mathrm{d}z \Big{|}_{H^k}
            \\
            & 
            \leq
            M(k)\|\nabla_{X,z}^{\mu}( \phi_b - \phi_{\mathrm{app}})\|_{H^{k+1,0}(\mathcal{S}_b)} 
            +
            M(k+1)\|\partial_z(\phi_b-\phi_{\mathrm{app}})\|_{H^{k,0}(\mathcal{S}_b)}
            \\
            & 
            \hspace{0.5cm}
            +
            M(k)\sum\limits_{j=1}^k\|\nabla_{X,z}^{\mu}\partial_z^{j-1}(\phi_b - \phi_{\mathrm{app}})\|_{H^{k-j+1,0}(\mathcal{S}_b)}  
            \\
            & 
            \hspace{0.5cm}
            +
            \ve M(k+1)
            \sum\limits_{j=1}^k\|\partial_z^{j+1} (\phi_b - \phi_{\mathrm{app}})\|_{H^{k-j,0}(\mathcal{S}_b)}
            \\
            & = 
            IV_1 + IV_2 + IV_3 + IV_4
            .
        \end{align*}
        For the two first terms we use estimate \eqref{Est app} to get that
        \begin{align*}
             IV_1 + IV_2 \leq (\mu^2 \ve + \mu \ve \beta + \mu^2 \beta^2 )M(k+3)|\nabla_X\psi|_{H^{k+4}}.
        \end{align*}
        For the estimate of $IV_3$ and $IV_4$, we will use the same ideas that we used for $I_3$ and $I_4$. We first note that we only need to work with 
        $$\phi^2_{\mathrm{app}} : = \phi_0 + \mu \beta \phi_1 + \mu\ve \phi_2,$$
        constructed in Propositions \ref{Proposition phi0}, \ref{Prop phi 1} and \ref{Prop phi 2}.  Indeed, from Observation \ref{obs phi app} we used the approximation \eqref{Formula phi0} and depends polynomially on $z$. So formula \eqref{Phi app} is related by $\phi^2_{\mathrm{app}}$ through the relation
        \begin{equation}\label{phi app 1 and phi app}
            \partial_z^k(\phi^2_{\mathrm{app}} - \phi_{\mathrm{app}})  \sim \mu(\mu \ve +  \ve \beta)\partial_z^k(z^2R),
        \end{equation}
        for $k\geq 1$ and where $R = R(X)$ satisfies \eqref{Rest phi app}. Then by definition of $\phi_0$, $\phi_1$, and $\phi_2$  we have that
        \begin{align*}
            \partial_z^2(\phi_b - \phi^2_{\mathrm{app}})
            & = - \mu \Delta_X(\phi_b - \phi_0 - \mu \beta \phi_1) - \mu \ve A[\nabla_X, \partial_z](\phi_b - \phi_0) - \mu\ve (A[\nabla_X ,\partial_z]\phi_0 + \partial_z^2 \phi_2)
            \\
            & = 
            - \mu \Delta_X(\phi_b - \phi_0 - \mu \beta \phi_1) - \mu \ve \tilde{A}[\nabla_X, \partial_z](\phi_b - \phi_0) - \mu\ve (A[\nabla_X ,\partial_z]\phi_0 + \partial_z^2 \phi_2)
            \\
            & \hspace{0.4cm}
            -
            \mu|\nabla_X \sigma|^2 \partial_z^2(\phi_b - \phi_0- \mu \beta \phi_1-\mu \ve \phi_2)) - \mu^2 \beta |\nabla_X \sigma|^2\partial_z^2\phi_1+ \mu^2 \ve |\nabla_X \sigma|^2\partial_z^2\phi_2,
        \end{align*}
        so that
        \begin{align*}
           (1+ \mu|\nabla_X \sigma|^2 )\partial_z^2(\phi_b - \phi^2_{\mathrm{app}})
           & =
           -
           \mu \Delta_X(\phi_b - \phi_0 - \mu \beta \phi_1)
           -
           \mu \ve \tilde{A}[\nabla_X, \partial_z](\phi_b - \phi_0) 
           \\
            & \hspace{0.5cm}
            -
            \mu\ve (A[\nabla_X ,\partial_z]\phi_0 + \partial_z^2 \phi_2)
            -
            \mu^2|\nabla_X \sigma|^2\partial_z^2( \beta \phi_1+ \ve \partial_z^2\phi_2).
        \end{align*}
        Here derivatives of $\phi_1$ is bounded using Proposition \ref{Prop T1 and T2} and by definition of $\sigma$, given by \eqref{grad sigma}, we have that
        \begin{equation*}
            \mu^2 \beta |\nabla_X \sigma|^2\partial_z^2\phi_1 \sim \mu^2 \ve^2 \beta \partial_z^2\phi_1.
        \end{equation*}
        Moreover, since $\phi_2$ is only polynomial in $z$ can use the notation above \eqref{useful notation} to see the last term as
        \begin{equation*}
            \mu^2 \ve |\nabla_X \sigma|^2\partial_z^2\phi_2 \sim \mu^2 \ve^3( 1+ z+z^2).
        \end{equation*}
        Also, we see from observation \ref{obs approx of A and tildeB} that
        \begin{align*}
            \mu\ve (A[\nabla_X ,\partial_z]\phi_0 + \partial_z^2 \phi_2) \sim \mu \ve \Delta_X(\phi_0 - \psi) + \mu \ve (1+z)\nabla_Xf\cdot \nabla_X\partial_z \phi_0 
            +
            \mu \ve z\partial_z \phi_0,
        \end{align*}
        for some $f \in H^{k+3}(\R^d)$. Then arguing as in Step $2$, we get the induction relation for $k\geq 3$:
        \begin{align*}
            \partial_z^k(\phi_b - \phi^2_{\mathrm{app}})
            & \sim
            \mu \sum \limits_{\gamma\in \N^d \: |\gamma|\leq k-1 } 
            \partial_X^{\gamma}\partial_z
            \Big{(}(\phi_b - \phi_{\mathrm{app}}^1) 
            +
            \ve (\phi_b - \phi_0)\Big{)} 
            \\ 
            & \hspace{0.5cm}
            +
            \mu \ve \sum\limits_{j=1}^{k-2} \partial_z^{j}\big{(}\Delta_X\phi_0
            +
            \nabla_Xf\cdot \nabla_X\partial_z \phi_0
            +
            \partial_z \phi_0\big{)}
            +
            \mu^2 \ve^2 \beta \sum\limits_{j=1}^{k} \partial_z^j\phi_1.
        \end{align*}
        Then as a result, we use these estimates with the product estimate \eqref{Classical prod est}, \eqref{Prop F0}, and \eqref{Prop T1 and T2}  to obtain the bound
        \begin{align*}
            \sum\limits_{j=1}^k\|\partial_z^{j+1} (\phi_b - \phi^2_{\mathrm{app}})\|_{H^{k-j}(\mathcal{S}_b)} & 
            \lesssim \mu M(k+1)\Big{(}\|\partial_z(\phi_b - \phi_{\mathrm{app}}^1)\|_{H^{k,0}(\mathcal{S}_b)} + \ve\|\partial_z(\phi_b - \phi_{0})\|_{H^{k,0}(\mathcal{S}_b)} 
            \\
            & 
            + 
            \mu^2 \ve |\nabla_X\psi|_{H^{k+1}} + \mu^2\ve^2 \beta |\nabla_X \cdot \mathcal{L}_1^{\mu}[\beta b] \nabla \psi|_{H^{k+1}}
            \Big{)}
            ,
        \end{align*}
        from which the estimate on $IV_4$ follows by \eqref{L1 est},  the relation \eqref{phi app 1 and phi app} with estimate \eqref{Rest phi app}, and then  \eqref{Est mu2ve} and \eqref{est phi0}:
        \begin{align*}
            IV_4 
            & \leq M(k)\sum\limits_{j=1}^k\|\nabla_{X,z}^{\mu}\partial_z^{j-1}(\phi_b - \phi^2_{\mathrm{app}})\|_{H^{k-j+1,0}(\mathcal{S}_b)} 
            +
            \mu(\mu \ve + \ve \beta ) |\nabla_X\psi|_{H^{k+4}}
            \\ 
             & \leq 
            \mu(\mu \ve + \ve \beta ) M(k+2)|\nabla_X\psi|_{H^{k+4}}.
        \end{align*}
        The same estimate holds for $IV_5$, and therefore completes the proof.

    \end{proof}

    \subsection{Multi-scale expansions of $\mathcal{G}^{\mu}$}\label{mult G}
    In this section, we give the expansions of the Dirichlet-Neumann operator. We  will use that $\mathcal{G}^{\mu}$ is directly related to $\overline{V}$ through \eqref{D-N operator} and \eqref{New V bar}. In particular, we have the following result two results.
    \begin{prop}\label{Prop G[0,b]}Under the provisions of Proposition \ref{prop V[0,b]} we define
        \begin{equation*}\label{def G[0,b]}
            \frac{1}{\mu}\mathcal{G}_b \psi = -\nabla_{X} \cdot \Big{(} \frac{h}{h_b} \int_{-h_b}^0 \nabla_X \phi \: \mathrm{d}z\Big{)},
        \end{equation*}
        and for $\psi \in \dot{H}^{s+5}(\R^d)$ we have the estimate
        \begin{equation}\label{est G[0,b]}
            \dfrac{1}{\mu}|\mathcal{G}^{\mu}\psi 
         - \mathcal{G}_b \psi  |_{H^{s}} \leq \mu\ve M(s+3)|\nabla_X \psi|_{H^{s+4}}.
        \end{equation}

    \end{prop}

    \begin{proof}
        By definition of the Dirichlet-Neumann operator \eqref{Relation: D-N op} and  Proposition \ref{prop V[0,b]} we have the result
        \begin{align*}
            \dfrac{1}{\mu}|\mathcal{G}^{\mu} \psi 
         - \mathcal{G}_b \psi|_{H^{s}} 
            & = |\nabla_X \cdot (h(\overline V - \overline{V}[0,\beta b]\psi))|_{H^s}
            \\ 
            & 
            \leq \mu \ve  M(s+3)|\nabla_X \psi|_{H^{s+4}}.
        \end{align*}
    \end{proof}

    \color{black}

    \begin{prop}\label{Cor G} 
    Under the provisions of Proposition \ref{Prop V}, we can define the approximations
    \begin{align*}
        \frac{1}{\mu}\mathcal{G}_0\psi & = 
        -
        \mathrm{F}_1\Delta_X \psi  
        -\beta(1+
        \frac{\mu}{2}\mathrm{F}_4 \Delta_X) 
        \nabla_X\cdot \big(\mathcal{L}_{1}^{\mu}[\beta b]\nabla_X \psi \big) 
        -
        \varepsilon \nabla_X\cdot \big( \zeta\mathrm{F}_1 \nabla_X\psi \big)
        \\
        &
        \hspace{0.5cm}
        +
        \frac{\mu  \beta^2}{2}\nabla_X\cdot\big(\mathcal{B}[\beta b] \nabla_X \psi\big),
    \end{align*}
    and
    \begin{align}\label{second approximation of the DN operator}
        \frac{1}{\mu}\mathcal{G}_1 \psi & = -\nabla_X\cdot (h \nabla_X \psi) 
         -
         \frac{\mu}{3} \Delta_X\Big{(}\frac{h^3}{h_b^3}\mathrm{F}_2 \Delta_X \psi \Big{)} 
         -
         \mu \beta \Delta_X \big{(}   \mathcal{L}_{2}^{\mu}[\beta b]\Delta_X \psi\big{)}
         \\ 
         & 
         \hspace{0.5cm}\notag
         -
        \frac{\mu \beta}{2} \mathrm{F}_4 \Delta_X \nabla_X \cdot \big(\mathcal{L}_1^{\mu}[\beta b] \nabla_X \psi\big) 
        +
        \frac{\mu  \beta^2}{2} \nabla_X\cdot\big(\mathcal{B}[\beta b] \nabla_X \psi\big),
    \end{align}
    where
    \begin{align}\label{B bathymetry}
        \mathcal{B}[\beta b]\nabla_X \psi &= 
            b \mathrm{F}_4\nabla_X(\nabla_X\cdot(b\nabla_X \psi))
            \\ 
            & \hspace{0.5cm} \notag
            +
            h_b
             \nabla_X \big{(}b\mathrm{F}_4\nabla_X\cdot(b\nabla_X \psi))
             +
             2h_b(\nabla_Xb)\mathrm{F}_1\nabla_X\cdot(b\nabla_X \psi).
    \end{align}
    Moreover, for $\psi \in \dot{H}^{s+6}(\R^d)$ we have the following estimates on the Dirichlet-Neumann operator 
    \begin{align}\label{G0}
         & \dfrac{1}{\mu}|\mathcal{G}^{\mu}\psi 
         - \mathcal{G}_0\psi  |_{H^{s}} \leq (\mu\ve + \mu^2 \beta^2) M(s+3)|\nabla_X \psi|_{H^{s+5}}
         \\\label{G1}
         & \dfrac{1}{\mu}|\mathcal{G}^{\mu}\psi 
         - \mathcal{G}_1\psi |_{H^{s}} \leq (\mu^2\ve +\mu \ve \beta + \mu^2 \beta^2) M(s+3)|\nabla_X \psi|_{H^{s+5}}.
    \end{align}

    \end{prop}
    \begin{proof}
        To prove inequality \eqref{G0}, we introduce a generic function $R$ such that
        \begin{equation}\label{Rest G0}
            |R|_{H^s} \leq M(s+3) |\nabla_X \psi|_{H^{s+5}}.
        \end{equation}
        Then  note that the first two terms in $\mathcal{G}_0$ are obtained from the first two terms in $\overline{V}_0$. Indeed, let $G = \nabla_X \cdot \mathcal{L}_1^{\mu}[\beta b] \nabla_X \psi$ and use formula \eqref{V0} to observe that
        \begin{align*}
            \frac{1}{\mu}\mathcal{G}_0\psi
            & =
            -\nabla_X \cdot(h \overline{V}_0)
            \\
            & 
            =-
            \nabla_X \cdot \Big( \dfrac{h}{h_b} \mathrm{F}_1\nabla_X\psi\Big) 
             -
             \beta \nabla_X\cdot \Big(\dfrac{h}{h_b}\mathcal{L}_{1}^{\mu}[\beta b]\nabla_X \psi\Big)
             \\ 
             & 
             \hspace{0.5cm}
             -
             \frac{\mu\beta}{2}\nabla_X\cdot(h\mathrm{F}_4\nabla_X G) 
             +
             \frac{\mu\beta^2}{2}\nabla_X \cdot\Big{(}h\big{(}
             \nabla_X \big{(}b\mathrm{F}_4\nabla_X\cdot(b\nabla_X \psi)\big{)}
             +
             2(\nabla_Xb)\mathrm{F}_1\nabla_X\cdot(b\nabla_X \psi)\big{)}\Big{)}
             \\ 
             & =\mathrm{RHS}_1 + \mathrm{RHS}_2 + \mathrm{RHS}_3,
        \end{align*}
        where
        \begin{align*}
            \mathrm{RHS}_1
            : & = 
            -
            \nabla_X\cdot \Big( \dfrac{h}{h_b}\mathrm{F}_1\nabla_X\psi\Big) 
            -
            \beta \nabla_X\cdot \Big(\dfrac{h}{h_b}\mathcal{L}_{1}^{\mu}[\beta b]\nabla_X \psi\Big) 
            \\
            & =
            -
            \mathrm{F}_1\Delta_X\psi 
            -
            \varepsilon \nabla_X\cdot\big( \dfrac{\zeta}{h_b}\mathrm{F}_1\nabla_X\psi \big) 
            -
            \beta\nabla_X\cdot \big( \mathcal{L}_{1}^{\mu}[\beta b]\nabla_X \psi \big) 
            -
            \varepsilon\beta \nabla_X\cdot \big(\dfrac{\zeta}{h_b} \mathcal{L}_{1}^{\mu}[\beta b]\nabla_X \psi \big),
        \end{align*}
        and we use \eqref{L approx} to get the approximation
        \begin{align*}
            \mathcal{L}_{1}^{\mu}[\beta b]\nabla_X \psi = -\beta b \nabla_X\psi + \mu R.
        \end{align*}
        Then we obtain that
        \begin{align*}
            \mathrm{RHS}_1
            &= 
            -
            \mathrm{F}_1\Delta_X\psi  
            -
            \beta \nabla_X\cdot \big( \mathcal{L}_{1}^{\mu}[\beta b]\nabla_X \psi \big) 
            -
            \varepsilon \nabla_X\cdot\big( \dfrac{\zeta}{h_b}\mathrm{F}_1\nabla_X\psi \big) 
            +
            \varepsilon \nabla_X\cdot \big(\dfrac{\zeta}{h_b} \beta b \nabla_X \psi \big) + \mu\varepsilon R
            \\
             &=
             -
             \mathrm{F}_1\Delta_X \psi  
             -
             \beta\nabla_X\cdot \big(\mathcal{L}_{1}^{\mu}[\beta b]\nabla_X \psi \big) 
             -
             \varepsilon \nabla_X\cdot \big( \zeta\mathrm{F}_1 \nabla_X\psi \big) +\mu \ve R.
        \end{align*} 
        For the remaining three terms, we first note that
        \begin{align*}
            \mathrm{RHS}_2
            &: = 
            -\frac{\mu\beta}{2}\nabla_X\cdot(h\mathrm{F}_4\nabla_X G) 
            \\
            & = 
            -
            \frac{\mu \beta}{2}  \mathrm{F}_4 \Delta_X G
            +
            \frac{\mu  \beta^2 }{2}\nabla_X\cdot(b \mathrm{F}_4\nabla_X G)
            +
            \frac{\mu \ve\beta }{2} R_1,
        \end{align*}
        where $R_1$ is given by
        \begin{align*}
            R_1 = 
            -\nabla_X\cdot(\zeta \mathrm{F}_4 \nabla_X G),
        \end{align*}
        Using the estimates in Proposition \ref{Simple est} and \eqref{L approx} allows us to put $R_1$ in the rest $R$ satisfying  \eqref{Rest G0}. Moreover, since $G = \nabla_X \cdot \big(\mathcal{L}_1^{\mu}[\beta b] \nabla_X \psi\big)$, we obtain
        \begin{align*}
            \mathrm{RHS}_2 
             & = 
             -
            \frac{\mu \beta }{2}\mathrm{F}_4 \Delta_X  \nabla_X \cdot \big(\mathcal{L}_1^{\mu}[\beta b] \nabla_X \psi \big) 
            +
            \frac{\mu  \beta^2}{2} \nabla_X\cdot\big(b \mathrm{F}_4\nabla_X(\nabla_X\cdot(b\nabla_X \psi)) \big)
            \\ 
            &
            \hspace{0.5cm}
            +
            (\mu \ve \beta + \mu^2 \beta^2) R.
        \end{align*}
        To conclude, we identify the remaining terms with the ones in $\nabla_X\cdot\big(\mathcal{B}[\beta b]\nabla_X \psi\big)$ by \eqref{B bathymetry}, and we conclude by \eqref{V0 est} that
        \begin{align*}
            \dfrac{1}{\mu}|\mathcal{G}^{\mu} \psi 
         - \mathcal{G}_0 \psi|_{H^{s}} 
            & = |\nabla_X \cdot (h(\overline V - \overline{V}_0))|_{H^s}
            \\ 
            & 
            \leq M(s+3)|\nabla_X \psi|_{H^{s+5}}.
        \end{align*}

        The proof of  inequality \eqref{G1} is similar, where we first use formula \eqref{V app} to get that
        \begin{align*}
            \frac{1}{\mu} \mathcal{G}_1\psi 
            & =
            -\nabla_X\cdot(h\overline{V}_{\mathrm{app}})  
            \\ 
            & = 
            -\nabla_X\cdot(h\nabla_X \psi )
           -
           \mu\Delta_X
           \Big{(}
            \frac{h^3}{h_b^3} \mathrm{F}_2 \psi
           \Big{)}
           -
           \mu\beta 
           \Delta_X\Big{(} \frac{h^3}{h_b^3} \mathcal{L}_2^{\mu}[\beta b] \psi\Big{)}
           \\ \notag
           & 
           \hspace{0.5cm}
           -
           \frac{\mu \beta}{2} \nabla_X\cdot\big(h\nabla_X\mathrm{F}_4\nabla_X \cdot \big(\mathcal{L}_1^{\mu}[\beta b] \nabla_X \psi\big)\big)
           \\ \notag
           & 
           \hspace{0.5cm}
             +
             \frac{\mu \beta^2}{2}\nabla_X \cdot\Big(h  (\nabla_X \big{(}b\mathrm{F}_4\nabla_X\cdot(b\nabla_X \psi)\big{)}
             +
             2(\nabla_Xb)\mathrm{F}_1\nabla_X\cdot(b\nabla_X \psi))\Big{)}.
        \end{align*}
        Then using the same arguments as for $\mathcal{G}_0$, for the last three terms, we know there is a function $R$ such that
        \begin{align*}
            \frac{1}{\mu} \mathcal{G}_1\psi 
            & =
            -\nabla_X\cdot(h\nabla_X \psi )
           -
           \mu\Delta_X
           \Big{(}
            \frac{h^3}{h_b^3} \mathrm{F}_2 \psi
           \Big{)}
           -
           \mu\beta 
           \Delta_X\Big{(} \frac{h^3}{h_b^3} \mathcal{L}_2^{\mu}[\beta b] \psi\Big{)}
           \\ \notag
           & 
           \hspace{0.5cm}
           -
           \frac{\mu \beta}{2} \mathrm{F}_4\Delta_X\nabla_X \cdot \big(\mathcal{L}_1^{\mu}[\beta b] \nabla_X \psi\big)
            +
            \frac{\mu  \beta^2}{2} \nabla_X\cdot\big(\mathcal{B}[\beta b] \nabla_X \psi\big) + (\mu \ve \beta + \mu^2 \beta^2)R,
        \end{align*}
        where $R$ satisfies \eqref{Rest G0}. Thus, we only use \eqref{id op 3} to say 
        \begin{equation}\label{Approx L2 beta}
        	\mu \beta \mathcal{L}_2^{\mu}[\beta b] =   \mu \beta R. 
        \end{equation}
        and combine it with the observation $\frac{h^3}{h_b^3}-1 = \ve R$, allowing us to neglect the term
        \begin{align*}
        	\mu \beta  \Delta_X  \Big{(} ( \frac{h^3}{h_b^3} -1)\mathcal{L}_{2}^{\mu}[\beta b]\Delta_X \psi\Big{)} = \mu \ve \beta R.
        \end{align*}
       By estimate \eqref{V app est} we conclude that \eqref{G1} holds true.

    \end{proof}

    \section{Derivation of Boussinesq type systems with bathymetry}\label{WH B} 

    In this section, we derive a family of weakly dispersive Boussinesq systems in the shallow water regime with precision $O(\mu\ve)$ and $O(\mu\ve+ \mu^2 \beta^2)$.
    
    \subsection{Derivation of a Boussinesq type system with precision $O(\mu\ve)$} 
    
    We will now derive a system with precision $O(\mu\ve)$. This system is defined implicitly through the solution of an elliptic problem on a fixed domain with solution $\phi$ which depends on time through the Dirichlet data $\psi$. 
    \begin{thm}
        Let $\mathcal{G}_b$ be defined by \eqref{def G[0,b]}. Then for any $\mu \in (0, 1]$, $\ve \in [0,1]$, and $\beta\in [0,1]$ the water waves equations \eqref{Water wave equations} are consistent, in the sense of Definition \ref{Consistency} with $n= 5$, at order $O(\mu\ve)$ with the Boussinesq type system:
        \begin{align}\label{Whitham boussinesq mu ve}
            \begin{cases}
                \partial_t \zeta - \frac{1}{\mu}\mathcal{G}_b \psi= 0\\
                \partial_t \psi + \zeta + \dfrac{\varepsilon}{2}|\nabla_X \psi|^2 = 0,
            \end{cases}
        \end{align}
    \end{thm}

    \begin{proof}
        For the first equation we use the approximation given  in Proposition \ref{Prop G[0,b]}. While for the second equation we simply use \eqref{G est 1} to replace the Dirichlet-Neumann operator by terms of order $O(\mu \ve)$.
    \end{proof}

    \subsection{Derivation of Boussinesq type systems with precision $O(\mu\ve + \mu^2 \beta^2)$}
    The system derived in this section will have the benefit of being explicit. This will reduce the computational cost from a numerical perspective, where the price we pay is given by an additional term of order $\mu^2 \beta^2$. However, the system has improved dispersive properties when compared to classical models. Moreover, since the precision is of higher order in $\beta$, these systems can handle larger amplitude topography variations. The first result of this section reads: \color{black}
    \begin{thm}
        Let $\mathrm{F}_1$ and $\mathrm{F}_4$ be the two Fourier multipliers given in Definition \ref{Fourier mult}, and let $\mathcal{L}_1^{\mu}$ be given in Definition \ref{Prop op}. Then for any $\mu \in (0, 1]$, $\ve \in [0,1]$, and $\beta\in [0,1]$ the water waves equations \eqref{Water wave equations} are consistent, in the sense of Definition \ref{Consistency} with $n=6$, at order $O(\mu\ve + \mu^2 \beta^2)$ with the Boussinesq type system:
        \begin{align}\label{Whitham boussinesq}
            \begin{cases}
                \partial_t \zeta + \mathrm{F}_1\Delta_X \psi +  \beta(1     +
                \frac{\mu }{2} \mathrm{F}_4\Delta_X) \nabla_X\cdot ( \mathcal{L}_{1}^{\mu}[\beta b]\nabla_X \psi) \\
                \hspace{4cm}+ \varepsilon \mathrm{G}_1 \nabla_X\cdot ( \zeta \mathrm{G}_2\nabla_X\psi) - \frac{\mu \beta^2}{2} \nabla_X\cdot\big(\mathcal{B}[\beta b] \nabla_X \psi\big) = 0\\
                \partial_t \psi + \zeta + \dfrac{\varepsilon}{2}(\mathrm{G}_1\nabla_X \psi)\cdot (\mathrm{G}_2\nabla_X \psi) = 0,
            \end{cases}
        \end{align}
        where 
        \begin{align*}
            \mathcal{B}[\beta b]\bullet &= 
            b \mathrm{F}_4\nabla_X(\nabla_X\cdot(b\bullet))
            +
            h_b
             \nabla_X \big{(}b\mathrm{F}_4\nabla_X\cdot(b\bullet))
             +
             2h_b(\nabla_X b)\mathrm{F}_1\nabla_X\cdot(b\bullet),
        \end{align*}
        and $\mathrm{G}_1, \mathrm{G}_2$ are any Fourier multipliers such that for any $s \geq 0$ and $u \in H^{s+2}(\mathbb{R}^d)$, we have
        \begin{align*}
            |(\mathrm{G}_j - 1)u|_{H^s} \lesssim \mu|u|_{H^{s+2}}.
        \end{align*}
    \end{thm}

    \begin{proof}
        To start,  we replace the Dirichlet-Neumann operator by \eqref{G est 1} and its expansion given by \eqref{G0} and discarding all the terms of order $O(\mu(\varepsilon +\mu \beta^2))$ in the water waves equations \eqref{Water wave equations} yields,
        \begin{align*}
            \begin{cases}
                \partial_t \zeta 
                +
        \mathrm{F}_1\Delta_X \psi  
        +\beta(1 +
        \frac{\mu}{2}\mathrm{F}_4 \Delta_X) 
        \nabla_X\cdot \big(\mathcal{L}_{1}^{\mu}[\beta b]\nabla_X \psi \big) 
        +
        \varepsilon \nabla_X\cdot \big( \zeta\mathrm{F}_1 \nabla_X\psi \big)   
                \\ 
                 \hspace{7cm}-  \frac{\mu  \beta^2}{2}\nabla_X\cdot\big(\mathcal{B}[\beta b] \nabla_X \psi\big) = (\mu \ve + \mu^2 \beta^2)R,
                \\
                \partial_t \psi + \zeta + \dfrac{\varepsilon}{2}|\nabla_X\psi|^2 = \mu \ve R.
            \end{cases}
        \end{align*}
        where we introduced a generic function $R$ such that
        \begin{equation}\label{Rest WB}
            |R|_{H^s} \leq M(s+3) |\nabla_X \psi|_{H^{s+5}}.
        \end{equation}
        To complete the proof, we use the assumption  on $\mathrm{G}_j$ whenever there is the appearance of an $\ve$. Then apply estimate \eqref{G0} up to the rest $R$ satisfying \eqref{Rest WB}. 
    \end{proof}

    The next result concerns a Boussinesq type system for which the first equation is exact and where the unknowns are given in terms of $(\zeta, \overline{V})$. 
    \begin{thm}
        Let $\mathrm{F}_1$ and $\mathrm{F}_4$ be the two Fourier multipliers given in Definition \ref{Fourier mult}, and let $\mathcal{L}_1^{\mu}$ be given in Definition \ref{Prop op}. Then for any $\mu \in (0, 1]$, $\ve \in [0,1]$, and $\beta\in [0,1]$ the water waves equations \eqref{Water wave equations} are consistent, in the sense of Definition \ref{Consistency} with $n=7$, at order $O(\mu\ve + \mu^2 \beta^2)$ with the Boussinesq type system:
        \begin{align}\label{Whitham boussinesq}
            \begin{cases}
                \partial_t \zeta + \nabla_X\cdot(h \overline{V}) = 0\\
                \partial_t \overline{V} + \mathcal{T}_0^{\mu}[\beta b, \ve \zeta] \nabla_X \zeta 
                +
                \frac{\ve }{2}\nabla_X|\overline{V}|^2
                = \mathbf{0},
            \end{cases}
        \end{align}
        where 
        \begin{align*}
            \mathcal{T}_0^{\mu}[\beta b, \ve \zeta] \bullet 
            & =
            \frac{1}{h}\Big{(}\mathrm{F}_1 \bullet  +  \beta \mathcal{L}_1^{\mu}[\beta b]\bullet  + \ve \zeta \mathrm{F}_1\bullet \Big{)}
            +
            \frac{\mu \beta}{2} \nabla_X\mathrm{F}_4\nabla_X \cdot \big(\mathcal{L}_1^{\mu}[\beta b] \bullet \big)
            \\ 
            & 
            \hspace{0.5cm}\notag
            -
            \frac{\mu\beta^2}{2}\nabla_X \big{(}b\mathrm{F}_4\nabla_X\cdot(b\bullet )\big{)}
            -
            \mu \beta^2 (\nabla_Xb)\mathrm{F}_1\nabla_X\cdot(b\bullet).
        \end{align*}
    \end{thm}
    \begin{proof}
        The first equation is exact by identity \eqref{Relation: D-N op}, and so we only work with the second equation of \eqref{Water wave equations}. However, using Theorem \ref{thm Whitham Boussinesq} we can work directly of on the second equation of \eqref{Whitham boussinesq} in the case $\mathrm{G}_1 = \mathrm{G}_2 = \mathrm{Id}$. Also, since we will take the gradient of $\psi$ we need to increase the regularity of our rest function. In particular, let $R$ be a generic function such that
        \begin{align*}
            |R|_{H^s} \leq M(s+3)|\nabla_X \psi|_{H^{s+6}}.
        \end{align*}
        Then by \eqref{V0} there holds
        \begin{align*}
            h\overline{V}  
            & =
            \frac{h}{h_b}\mathrm{F}_1 \nabla_X \psi 
            +
            \beta
            \frac{h}{h_b}\mathcal{L}_1^{\mu}[\beta b] \nabla_X \psi
            +
            \frac{\mu \beta}{2} h_b\nabla_X\mathrm{F}_4\nabla_X \cdot \mathcal{L}_1^{\mu}[\beta b] \nabla_X \psi
            \\ 
            & 
            \hspace{0.5cm}\notag
            -
            \frac{\mu\beta^2}{2}h_b\nabla_X \big{(}b\mathrm{F}_4\nabla_X\cdot(b\nabla_X \psi)\big{)}
            -
            \mu \beta^2 h_b(\nabla_Xb)\mathrm{F}_1\nabla_X\cdot(b\nabla_X \psi) 
            +
            (\mu \ve + \mu^2 \beta^2)R.
        \end{align*}
        Moreover, by \eqref{L approx} and \eqref{Simple est} we make the observation 
        \begin{align*}
            \frac{h}{h_b}\Big{(}\mathrm{F}_1 \nabla_X \psi 
            +
            \beta
            \mathcal{L}_1^{\mu}[\beta b] \nabla_X \psi\Big{)} =  \mathrm{F}_1 \nabla_X \psi +  \beta \mathcal{L}_1^{\mu}[\beta b] \nabla_X \psi + \ve \zeta \mathrm{F}_1\nabla_X \psi + \mu \ve  R,
        \end{align*}
        so that
        \begin{align*}
            h\overline{V} & = h\mathcal{T}_0^{\mu}[\beta b, \ve \zeta] \nabla_X \psi + (\mu \ve + \mu^2 \beta^2) R.
        \end{align*}
        From this expression, we can use the first equation to see that $\partial_t h = - \ve \nabla_X \cdot (h\overline{V})$, and the estimates \eqref{Simple est} together with the relation $\nabla_X \psi = \overline{V} + \mu R$ to get that:
        \begin{align*}
            h \partial_t \overline{V} 
            & =
            (\partial_t h)\big(\mathrm{F}_1 \nabla_X \psi - \overline{V}) +  h\mathcal{T}_0^{\mu}[\beta b, \ve \zeta] \nabla_X \partial_t\psi
            \\ 
            & = 
            h \mathcal{T}_0^{\mu}[\beta b, \ve \zeta] \nabla_X \partial_t\psi.
        \end{align*}
        We may now use this relation in the second equation of \eqref{Whitham boussinesq} where we apply the gradient and $\mathcal{T}_0[\beta b, \ve \zeta]$ to obtain that
        \begin{align*}
            h\partial_t \overline{V} + h\mathcal{T}_0^{\mu}[\beta b, \ve \zeta] \nabla_X \zeta + \frac{\ve}{2}h\mathcal{T}_0^{\mu}[\beta b, \ve \zeta] \nabla_X |\overline{V}|^2 = (\mu \ve + \mu^2 \beta^2)R.
        \end{align*}
        Then we conclude from the fact that $\mathcal{T}_0^{\mu}[\beta b, \ve \zeta] =  \mathrm{Id} + \mu R$.
        
    \end{proof}

    \subsubsection{Hamiltonian structure} We end this section by briefly commenting on the Hamiltonian structure of Boussinesq type systems with bathymetry.
    To do so, we recall the Hamiltonian of the water waves equations \eqref{Water wave equations} \cite{Zakharov68}:
    \begin{align}\label{Hamiltonian}
        H(\zeta, \psi) = \frac{1}{2} \int_{\R^d} \zeta^2 \: \mathrm{d}X + \frac{1}{2\mu} \int_{\R^d}\psi \mathcal{G}^{\mu}\psi \: \mathrm{d}X,
    \end{align}
    with $H(\zeta,\psi)$ satisfying the system
    \begin{equation}\label{System Hamiltonian}
        \begin{cases}
            \partial_t \zeta & = \delta_{\psi}H
            \\
            \partial_t \psi & = -\delta_{\zeta}H,
        \end{cases}
    \end{equation}
    where $\delta_{\psi}$ and $\delta_{\zeta}$ are functional derivatives. Then replacing the Dirichlet-Neumann operator in \eqref{Hamiltonian} with its approximation  its approximation \eqref{G0} we obtain
    \begin{align}\label{V bar hamiltonian}
        H(\zeta, \psi) 
        & =
        \frac{1}{2} \int_{\R^d} \zeta^2 \: \mathrm{d}X 
        +
        \frac{1}{2} \int_{\R^d} \mathrm{F}_1\nabla_X\psi\cdot \nabla_X \psi \: \mathrm{d}X 
        \\
        & \notag
        \hspace{0.2cm}
        +
        \frac{\ve}{2}\int_{\R^d} \zeta \mathrm{G}\nabla_X \psi \cdot \mathrm{G}\nabla_X \psi \: \mathrm{d}X 
        +
        \frac{\beta}{2} \int_{\R^d} \mathcal{L}^{\mu}_1[\beta b]\nabla_X\psi\cdot \nabla_X \psi \: \mathrm{d}X
        \\
        & \notag
        \hspace{0.2cm}
        + \frac{\mu\beta}{4} \int_{\R^d} \mathrm{F}_4 \Delta_X\mathcal{L}^{\mu}_1[\beta b]\nabla_X\psi\cdot \nabla_X \psi \: \mathrm{d}X -
        \frac{\mu\beta^2}{4} \int_{\R^d} \nabla_X\psi \cdot \mathcal{B}[\beta b]\nabla_X\psi \: \mathrm{d}X
        \\
        &
        \hspace{0.5cm}
        +\notag
        O(\mu \ve+\mu^2\beta^2),
    \end{align}
    for some Fourier multiplier $\mathrm{G}$ of the form $\mathrm{G} = 1 + O(\mu)$.
    
    Now, to compute the functional derivatives in system \eqref{System Hamiltonian} we note that the Fourier multipliers that appear are self-adjoint.  While for the pseudo-differential operator of order zero, $\mathcal{L}_1^{\mu}$, one can use the fact that there exists an adjoint. However, a simpler approach is to approximate it by \eqref{L1 approx next order} and gives
    \begin{equation*}
        \mathcal{L}_1^{\mu}[\beta b] = - b \mathrm{F}_3 - \frac{\mu\beta^2}{6}b^3|\mathrm{D}|^2\mathrm{F}_3+ O(\mu^2 \beta^4).
    \end{equation*}
    Using this relation implies
    \begin{align*}
        \mathrm{RHS}_1 : & =
        \frac{\beta}{2} \int_{\R^d} \mathcal{L}^{\mu}_1[\beta b]\nabla_X\psi\cdot \nabla_X \psi \: \mathrm{d}X
        \\ 
        & = 
        -\frac{\beta}{2} \int_{\R^d} b\mathrm{F}_3 \nabla_X \psi \cdot \nabla_X \psi \: \mathrm{d}X
        + \frac{\mu\beta^3}{12} \int_{\R^d} b^3\Delta_X\mathrm{F}_3 \nabla_X \psi \cdot \nabla_X \psi \: \mathrm{d}X
        +
        O(\mu^2 \beta^5),
    \end{align*}
    and 
    \begin{align*}
        \mathrm{RHS}_2 : & = \frac{\mu\beta}{4} \int_{\R^d} \mathrm{F}_4 \Delta_X\mathcal{L}^{\mu}_1[\beta b]\nabla_X\psi\cdot \nabla_X \psi \: \mathrm{d}X\\
        & = - \frac{\mu\beta}{4} \int_{\R^d} \mathrm{F}_4 \Delta_X(b\nabla_X\psi)\cdot \nabla_X \psi \: \mathrm{d}X + O(\mu^2\beta^3).
    \end{align*}
    In particular, the first equation in \eqref{System Hamiltonian} is given by
    \begin{align*}
        \delta_{\psi}H & =
        - 
        \mathrm{F}_1 \Delta_X \psi 
        +
        \nabla_X\cdot(\mathcal{A}^{\mu}[\beta b] \nabla_X \psi)
        - 
        \ve \mathrm{G}\nabla_X\cdot (\zeta \mathrm{G}\nabla_X \psi) + \frac{\mu\beta^2}{4}\nabla_X\cdot\Big(\big(\mathcal{B}[\beta b] + \mathcal{B}[\beta b]^*\big)\nabla_X\psi\Big),
    \end{align*}
    where
    \begin{align*}
        \mathcal{A}^{\mu}[\beta b]\bullet
         =
        \frac{\beta}{2} \big(\mathrm{F}_3(b\bullet) + b\mathrm{F}_3\bullet\big)
        + \frac{\mu\beta}{2} \big( \mathrm{F}_4 \Delta_X(b\bullet) + b\mathrm{F}_4\Delta_X\bullet\big)- \frac{\mu\beta^3}{12} \big( b^3\Delta_X\mathrm{F}_3\bullet + \Delta_X\mathrm{F}_3(b^3\bullet)\Big)
        ,
    \end{align*}
    and where $\mathcal{B}[\beta b]^*$ stands for the adjoint of $\mathcal{B}[\beta b]$ and reads
    \begin{align*}
        \mathcal{B}[\beta b]^*\nabla_X\psi = b\mathrm{F}_4\nabla_X \big(\nabla_X\cdot(b\nabla_X\psi)\big) + b \mathrm{F}_4\nabla_X(b \nabla_X\cdot(h_b\nabla_X\psi)) + 2b\mathrm{F}_1\nabla_X(h_b\nabla_X b\cdot\nabla_X\psi).
    \end{align*}
    Similarly for the second equation:
    \begin{align*}
        \delta_{\zeta}H & = -\zeta  - \frac{\ve}{2}|\mathrm{G}\nabla_X \psi|^2.
    \end{align*}
    Then using \eqref{System Hamiltonian}, we will arrive at the following system
    \begin{align}\label{Hamiltonian WB}
        \begin{cases}
            \partial_t \zeta 
            +
            \mathrm{F}_1 \Delta_X \psi 
            -
            \nabla_X \cdot \big(\mathcal{A}^{\mu}[\beta b] \nabla_X \psi\big) 
            + 
            \ve \mathrm{G}\nabla_X\cdot (\zeta \mathrm{G}\nabla_X \psi)
            \\\hspace{5cm}- \frac{\mu\beta^2}{4}\nabla_X\cdot\Big(\big(\mathcal{B}[\beta b]
            + \mathcal{B}[\beta b]^*\big)\nabla_X\psi\Big) = 0
            \\
            \partial_t \psi +\zeta  + \frac{\ve}{2}|\mathrm{G}\nabla_X \psi|^2 = 0,
        \end{cases}
    \end{align}
    where its Hamiltonian reads:
     \begin{align*}
        H(\zeta, \psi) 
        & =
        \frac{1}{2} \int_{\R^d} \zeta^2 \: \mathrm{d}X 
        +
        \frac{1}{2} \int_{\R^d} \mathrm{F}_1\nabla_X\psi\cdot \nabla_X \psi \: \mathrm{d}X 
        \\
        & \notag
        \hspace{0.5cm}
        +
        \frac{\ve}{2}\int_{\R^d} \zeta \mathrm{G}\nabla_X \psi \cdot \mathrm{G}\nabla_X \psi \: \mathrm{d}X 
        -\frac{\beta}{2} \int_{\R^d} (1 +  \frac{\mu }{2}\mathrm{F}_4 \Delta_X)b\mathrm{F}_3 \nabla_X \psi\cdot \nabla_X \psi \: \mathrm{d}X
        \\
        &\notag
        \hspace{0.5cm}
        -
        \frac{\mu\beta^2}{4} \int_{\R^d} \mathcal{B}[\beta b]\nabla_X\psi \cdot \nabla_X\psi \: \mathrm{d}X,
    \end{align*}
    and is preserved by smooth solutions of \eqref{Hamiltonian WB}.
    \begin{remark}
        If we neglect terms of order $O(\mu \ve+\mu\beta)$, using $\mathrm{F}_3 = 1+ O(\mu)$, we obtain the system derived in \cite{DucheneMMWW21}. 
    \end{remark}

    \color{black}
    \section{Derivation of Green-Naghdi type systems with bathymetry}\label{Derivation Whitham-Green-Naghdi systems with bathymetry}
    
    In this section, we derive weakly dispersive Green-Naghdi systems in the shallow water regime with precision $O(\mu^2\ve + \mu\ve \beta + \mu^2 \beta^2)$. The following Green-Naghdi type system may be derived from the water waves equations: 
    \begin{thm}
         Let $\mathrm{F}_2$ and $\mathrm{F}_4$ be the two Fourier multipliers given in Definition \ref{Fourier mult}, and let $\mathcal{L}_2^{\mu}$ be given in Definition \ref{Prop op}. Then for any $\mu \in (0, 1]$, $\ve \in [0,1]$, and $\beta \in [0,1]$ the water waves equations \eqref{Water wave equations} are consistent, in the sense of Definition \ref{Consistency} with $n=6$, at order $O(\mu^2\ve + \mu\ve \beta + \mu^2 \beta^2)$ with the Green-Naghdi type system:
        \begin{align}\label{Good-Whitham-Green-Naghdi}
            \begin{cases}
                \partial_t \zeta 
                +
                \nabla_X \cdot (h\mathcal{T}_1^{\mu}[\beta b, \ve \zeta]\nabla_X \psi)
                - \frac{\mu \beta^2}{2} \nabla_X\cdot\big(\mathcal{B}[\beta b]  \nabla_X \psi\big) = 0
                \\
                \partial_t \psi + \zeta + \dfrac{\varepsilon}{2}|\nabla_X\psi|^2 - \dfrac{\mu\varepsilon}{2}h^2(\sqrt{\mathrm{F}_2}\Delta_X\psi)^2 = 0,
            \end{cases}
    \end{align}
    where 
    \begin{align*}
            \mathcal{B}[\beta b]\bullet &= 
            b \mathrm{F}_4\nabla_X(\nabla_X\cdot(b\bullet))
            +
            h_b
             \nabla_X \big{(}b\mathrm{F}_4\nabla_X\cdot(b\bullet)\big{)}
             +
             2h_b(\nabla_Xb)\mathrm{F}_1\nabla_X\cdot(b\bullet),
        \end{align*}
    and
    \begin{align*}
        \mathcal{T}_1^{\mu}[\beta b, \ve \zeta]\bullet
               &  =
                \mathrm{Id}
                +
                \dfrac{\mu}{3h} \nabla_X\sqrt{\mathrm{F}_2} \Big(\dfrac{h^3}{h_b^3}\sqrt{\mathrm{F}_2}\nabla_X \cdot \bullet  \Big) 
                +
                \frac{\mu \beta }{h}\nabla_X \Big{(}  \mathcal{L}_{2}^{\mu}[\beta b]\nabla_X \cdot \bullet \Big{)} 
                \\ 
                & 
                \hspace{0.5cm}
                +
                \frac{\mu \beta }{2h}\mathrm{F}_4\nabla_X \nabla_X \cdot \big(\mathcal{L}_1^{\mu}[\beta b] \bullet\big),
    \end{align*}
    and $\sqrt{\mathrm{F}_2}$ is the square root of  $\mathrm{F}_2$.
    \end{thm}

    \begin{proof}
        We see that the first equation can be deduced by trading the Dirichlet-Neumann operator with its approximation \eqref{G1}. Indeed, we obtain that
        \begin{align*}
        	\partial_t \zeta 
        	+
        	\nabla_X\cdot (h \nabla_X\psi) 
        	&
                +
        	\dfrac{\mu}{3}\Delta_X \Big(\dfrac{h^3}{h_b^3}\mathrm{F}_2 \Delta_X 
        	\psi \Big)  
                +
        	\mu \beta \Delta_X \big{(}  \mathcal{L}_{2}^{\mu}[\beta b]\Delta_X \psi\big{)} 
             \\ 
                 &
                +\frac{\mu \beta}{2} \mathrm{F}_4\Delta_X \nabla_X \cdot \mathcal{L}_1^{\mu}[\beta b] \nabla_X \psi
                - \frac{ \mu \beta^2}{2} \mathcal{B}[\beta b]  \nabla_X \psi
        	= (\mu^2 \ve + \mu \ve \beta + \mu^2 \beta^2)R,
        \end{align*}
        where we introduce a generic function $R$ such that
        \begin{equation}\label{Rest WGN}
            |R|_{H^s} \leq M(s+3) |\nabla_X \psi|_{H^{s+5}}.
        \end{equation}
        Therefore we need to approximate the term
        \begin{align*}
                \dfrac{\mu}{3}\Delta_X\Big(\dfrac{h^3}{h_b^3}\mathrm{F}_2 \Delta_X 
                \psi \Big) 
        \end{align*}
        at order $O(\mu^2 \ve + \mu \ve \beta)$. Indeed, using \eqref{Simple est} to say $\mathrm{F}_2 = 1 + \mu R$ and  $\sqrt{\mathrm{F}_2} = 1 + \mu R$, we obtain
        \begin{align*}
             \dfrac{\mu}{3}\Delta_X \Big((\dfrac{h^3}{h_b^3}-1)\mathrm{F}_2 \Delta_X \psi \Big) =\dfrac{\mu}{3}\Delta_X\sqrt{\mathrm{F}_2} \Big((\dfrac{h^3}{h_b^3}-1)\sqrt{\mathrm{F}_2} \Delta_X 
                \psi \Big) + \mu^2 \ve R.
        \end{align*}
        Gathering these observations yields
        \begin{align}\label{from F2 to sqrt F2}
            \dfrac{\mu}{3}\Delta_X\Big(\dfrac{h^3}{h_b^3}\mathrm{F}_2 \Delta_X 
                \psi \Big) 
             =  
                \dfrac{\mu}{3}\Delta_X\sqrt{\mathrm{F}_2} \Big(\dfrac{h^3}{h_b^3}\sqrt{\mathrm{F}_2} \Delta_X 
                \psi \Big) 
                +
                \mu^2 \ve R.
        \end{align}

        For the second equation, we use \eqref{G est 1} to make the observation
        \begin{align*}
            & \frac{(\frac{1}{\mu} \mathcal{G}^{\mu}[\ve \zeta, \beta b]\psi + \ve \nabla_X \zeta \cdot \nabla_X \psi)^2}{1+ \ve^2 \mu |\nabla_X \zeta|^2} 
            -
            \Big{(}\frac{1}{\mu} \mathcal{G}^{\mu}[\ve \zeta, \beta b]\psi + \ve \nabla_X \zeta \cdot \nabla_X \psi\Big{)}^2
            \\
            &\hspace{0.5cm} 
            =
            \frac{\mu \ve^2 |\nabla_X \zeta|^2 (\frac{1}{\mu} \mathcal{G}^{\mu}[\ve \zeta, \beta b]\psi + \ve \nabla_X \zeta \cdot \nabla_X \psi)^2 }{1+ \ve^2 \mu |\nabla_X \zeta|^2}
            \\
            & 
            \hspace{0.5cm} = \mu \ve^2 R.
        \end{align*}
        Meaning that we only need to make an approximation of
        \begin{equation*}
            \Big{(}\frac{1}{\mu} \mathcal{G}^{\mu}[\ve \zeta, \beta b]\psi + \ve \nabla_X \zeta \cdot \nabla_X \psi\Big{)}^2,
        \end{equation*}
        at order $O(\mu)$. In particular, we use  \eqref{G est 1} to simplify the second equation in the water waves equations \eqref{Water wave equations} to get that
        \begin{align}\label{Equation to be simplified}
            \partial_t \psi + \zeta + \frac{\ve}{2}|\nabla_X \psi|^2 - \frac{\mu \ve}{2}\Big{(} \frac{1}{\mu} \mathcal{G}^{\mu}[\ve \zeta, \beta b]\psi + \ve \nabla_X \zeta \cdot \nabla_X \psi\Big{)}^2 = \mu^2 \ve R.
        \end{align}
        Then using \eqref{G est 1} we have that
        \begin{align*}
            \frac{1}{\mu} \mathcal{G}^{\mu}[\ve \zeta, \beta b]\psi 
            & = 
            - \nabla_{X} \cdot(h\nabla_{X} \psi) + \mu R
            \\
            & = 
            - h\Delta_{X} \psi - \ve \nabla_X \zeta \cdot \nabla_X \psi + (\mu + \beta) R,
        \end{align*}
        and we may use this expression to simplify \eqref{Equation to be simplified} where we again use that $\sqrt{\mathrm{F}_2} = 1 + \mu R$.  Thus, we conclude the proof of this theorem with estimate \eqref{G est 1} up to a rest $R$ satisfying \eqref{Rest WGN}.

    \end{proof}

    One may also derive a system with unknowns $(\zeta,\overline{V})$ instead of $(\zeta,\psi)$, for which the first equation is exact.  The new system reads:
    \begin{thm}
        Let $\mathrm{F}_2$ and $\mathrm{F}_4$ be the two the Fourier multipliers given in Definition \ref{Fourier mult}, let $\mathcal{L}_1^{\mu}$ and $\mathcal{L}_2^{\mu}$ be given in Definition \ref{Prop op}. Then for any $\mu \in (0, 1]$, $\ve \in [0,1]$, and $\beta \in [0,1]$ the water waves equations \eqref{Water wave equations} are consistent, in the sense of Definition \ref{Consistency} with $n=7$, at order $O(\mu^2 \ve + \mu \ve \beta + \mu^2 \beta^2)$ with the Green-Naghdi type system:
        \begin{align}\label{Whitham-Green-Naghdi_Vbar}
        \begin{cases}
            \partial_t \zeta + \nabla_X \cdot(h\overline{V}) = 0, \\
            \partial_t(\mathcal{I}^{\mu}[h]\overline{V}) + \mathcal{I}[h]\mathcal{T}_2^{\mu}[\beta b, h]\nabla_X\zeta + \frac{\ve}{2}  \nabla_X \big( |\overline{V}|^2\big) + \mu\ve \nabla_X \mathcal{R}_1^{\mu}[\beta b, h, \overline{V}] = \mathbf{0},
        \end{cases}
        \end{align}
	where $\overline{V}$ defined by \eqref{V bar},
    \begin{equation}\label{operator I}
		\mathcal{I}^{\mu}[h]\bullet=  \mathrm{Id} - \frac{\mu}{3h}\sqrt{\mathrm{F}_2} \nabla_X \Big(h^3  \sqrt{\mathrm{F}_2} \nabla_X\cdot \bullet  \Big),
	\end{equation}
	\begin{align*}\label{T1}
		\mathcal{T}_2^{\mu}[\beta b, \ve \zeta]\bullet
            & =
            \mathrm{Id} + \frac{\mu}{3h}\sqrt{\mathrm{F}_2} \nabla_X \Big(\frac{h^3}{h_b^3}  \sqrt{\mathrm{F}_2} \nabla_X\cdot \bullet  \Big) 
            +
            \frac{\mu\beta }{h} \nabla_X \Big(\mathcal{L}_2^{\mu}[\beta b] \nabla_X\cdot \bullet \Big)
            \\ 
            & 
            +
            \frac{\mu \beta h_b}{2h}\nabla_X \mathrm{F}_4 \nabla_X \cdot \big(\mathcal{L}_1^{\mu}[\beta b] \bullet\big)
            -
            \frac{\mu\beta^2h_b}{2h}\nabla_X \big{(}b\mathrm{F}_4\nabla_X\cdot(b\bullet)\big{)}
            \\
            &
            - 
            \frac{\mu \beta^2 h_b}{h} (\nabla_Xb)\mathrm{F}_1\nabla_X\cdot(b\bullet),
	\end{align*}
	and
	\begin{equation}\label{R1}
		\mathcal{R}_1^{\mu}[\beta b, h, \overline{V}] 
            = - \frac{h^2}{2}  (\nabla_X \cdot \overline{V})^2
            -
            \frac{1}{3h}\big{(}\nabla_X(h^3 \nabla_X \cdot \overline{V})\big{)} \cdot \overline{V}
            - \frac{1}{2}h^3\Delta_X(|\overline{V}|^2) 
            +
            \frac{1}{6h} h^3 \Delta_X (|\overline{V}|^2).
	\end{equation}
        
    \end{thm}
    \begin{proof}
        The first equation is exact so we only need to work on the second equation. Also, since $\overline{V}$ is related to the gradient of $\psi$ we need to increase the regularity of the rest function $R$. In particular, we introduce a generic function $R$ satisfying
        \begin{align}
            |R|_{H^s} \leq M(s+3)|\nabla_X \psi|_{H^{s+6}}.
        \end{align}
        Then from the estimate \eqref{V app}, \eqref{Approx L2 beta} and the argument in the previous proof, we know that
        \begin{align}\label{relation Vbar psi}
            h\overline{V} 
            & =
            h \nabla_X\psi + \frac{\mu}{3}\nabla_X\sqrt{\mathrm{F}_2} \Big(\frac{h^3}{h_b^3} \sqrt{\mathrm{F}_2} \Delta_X\psi \Big)  + \mu \beta \nabla_X \Big(  \mathcal{L}_2^{\mu}[\beta b] \Delta_X \psi \Big) 
            \\ 
            & 
            \hspace{0.5cm }\notag
            +
           \frac{\mu \beta h_b}{2} \nabla_X\mathrm{F}_4\nabla_X \cdot \big(\mathcal{L}_1^{\mu}[\beta b] \nabla_X \psi\big)
            -
            \frac{\mu\beta^2h_b}{2}\nabla_X \big{(}b\mathrm{F}_4\nabla_X\cdot(b\nabla_X \psi)\big{)}
            \\ 
            & 
            \hspace{0.5cm}\notag
            -
            \mu \beta^2h_b (\nabla_Xb)\mathrm{F}_1\nabla_X\cdot(b\nabla_X \psi)
            +
            (\mu^2\ve + \mu \ve \beta + \mu^2\beta^2)R.
        \end{align}
    	Deriving this equality in time and using the definition of $\mathcal{T}_2[\beta b, \ve \zeta] \bullet$ we obtain the relation
    	\begin{align}\label{derivation in time of V app}
        h \partial_t\overline{V}  
        &=
        \partial_t h(\nabla_X\psi - \overline{V}) 
        +
         h\mathcal{T}_2^{\mu}[\beta b, \ve \zeta]\nabla_X\partial_t \psi 
        +
         \frac{\mu}{3}\nabla_X\sqrt{\mathrm{F}_2}\Big(\partial_t\Big(\frac{h^3}{h_b^3}\Big) \sqrt{\mathrm{F}_2} \Delta_X\psi \Big) 
         \\ 
         & 
         \hspace{0.5cm}\notag
         +
         (\mu^2\ve + \mu \ve \beta + \mu^2\beta^2)R.
    \end{align}
    Moreover, noting that $\dfrac{1}{h_b} = 1 + \beta R$ we can deduce that
    \begin{equation*}
    	\frac{\mu}{3}\nabla_X\sqrt{\mathrm{F}_2}\Big(\partial_t\Big(\frac{h^3}{h_b^3}\Big) \sqrt{\mathrm{F}_2} \Delta_X\psi \Big) = -\mu \ve \nabla_X \sqrt{\mathrm{F}_2}\big( h^2\big{(}\nabla_X\cdot(h\overline{V})\big{)} \sqrt{\mathrm{F}_2} \Delta_X \psi\big) + \mu\ve\beta R,
    \end{equation*}
    and  using the first equation of system \eqref{Whitham-Green-Naghdi_Vbar} we have that \eqref{derivation in time of V app} is approximated by
    \begin{align}\label{derivation in time of V app 2}
    	h\mathcal{T}_2^{\mu}[\beta b, \ve \zeta]\nabla_X\partial_t \psi	
    	&=
    	h \partial_t\overline{V}  
    	+
    	\ve \big{(} \nabla_X \cdot (h \overline{V})\big{)}(\nabla_X\psi - \overline{V}) 
    	\\
    	& 
    	\hspace{0.5cm}\notag
    	+
    	\mu \ve \nabla_X \sqrt{\mathrm{F}_2}\big( h^2\big{(}\nabla_X\cdot(h\overline{V})\big{)}\sqrt{\mathrm{F}_2}\Delta_X \psi\big) +  (\mu^2\ve + \mu \ve \beta + \mu^2\beta^2)R.
    \end{align}
   To conclude we simply need to approximate $\nabla_X \psi$ by $\overline{V}$ where we use  \eqref{V app} to get the classical approximation:
    \begin{align}\label{classical approximations nabla psi in terms of Vbar}
        \begin{cases}
            \nabla_X\psi = \overline{V} + \mu R\\
            \nabla_X\psi = \overline{V} - \frac{\mu}{3h}\nabla_X(h^3\nabla_X\cdot \overline{V}) + \mu^2R.
        \end{cases} 
    \end{align}
   Furthermore, using \eqref{classical approximations nabla psi in terms of Vbar} and \eqref{derivation in time of V app 2}  we obtain that
    \begin{align}\label{Approximation nabla psi in terms of Vbar}
        h \mathcal{T}_2^{\mu}[\beta b, \ve \zeta] \partial_t \nabla_X\psi
        &= h \partial_t\overline{V} - \frac{\mu\ve}{3h}\big{(}\nabla_X\cdot(h\overline{V})\big{)}\nabla_X(h^3\nabla_X\cdot\overline{V})\\\notag
        &\hspace{0.5cm} + \mu \ve \nabla_X \sqrt{\mathrm{F}_2}\big( h^2\big{(}\nabla_X\cdot(h\overline{V})\big{)}\sqrt{\mathrm{F}_2} \nabla_X\cdot \overline{V}\big) 
        \\ 
         & 
         \hspace{0.5cm}\notag
         +
         (\mu^2\ve + \mu \ve \beta + \mu^2\beta^2)R.
    \end{align}
    We will now simplify the second equation of the water waves system \eqref{Water wave equations} at order $O(\mu^2\ve + \mu \ve \beta + \mu^2\beta^2)$. Using Theorem \ref{Thm good WGN} allows us to work with the second equation of \eqref{Good-Whitham-Green-Naghdi}. First use  \eqref{classical approximations nabla psi in terms of Vbar} to deduce that
    \begin{equation*}
    	|\nabla_X \psi|^2 = |\overline{V}|^2 - \frac{2\mu}{3h} \big{(}\nabla_X(h^3 \nabla_X \cdot \overline{V})\big{)} \cdot \overline{V} + \mu^2 R.  
    \end{equation*}
    With this relation, we may apply the gradient to the second equation of \eqref{Good-Whitham-Green-Naghdi}, and  then apply the operator $\mathcal{T}_2^{\mu}[\beta b, \ve \zeta ] \bullet$, using the approximation \eqref{classical approximations nabla psi in terms of Vbar}, and discarding all the terms of order $O(\mu^2\ve + \mu \ve \beta + \mu^2\beta^2)$ to get
    \begin{align*}
    	\mathcal{T}_2^{\mu}[\beta b, \ve \zeta]\nabla_X\partial_t \psi 
    	+&
    	\mathcal{T}_2^{\mu}[\beta b, \ve \zeta]\nabla_X \zeta 
    	+ \dfrac{\varepsilon}{2}\mathcal{T}_2^{\mu}[\beta b, \ve \zeta]\Big{(}\nabla_X(|\overline{V}|^2 
    	-
    	\frac{2\mu}{3h} \big{(}\nabla_X(h^3 \nabla_X \cdot \overline{V})\big{)} \cdot \overline{V} )\Big{)}
    	\\
    	& 
    	\hspace{5cm}
    	-
    	\dfrac{\mu\varepsilon}{2}h^2|\nabla_X \overline{V}|^2 =(\mu^2\ve + \mu \ve \beta + \mu^2\beta^2)R.
    \end{align*}
    Then we apply \eqref{Approximation nabla psi in terms of Vbar}, neglecting $\mathrm{F}_2$ whenever there are terms with $\mu\ve$ and together with the observation
    \begin{equation*}
    	\mathcal{T}_2^{\mu}[\beta b, \ve \zeta]\bullet =  \mathrm{Id} + \frac{\mu}{3h}\nabla_X \Big(h^3 \nabla_X\cdot \bullet  \Big) + \mu \beta R,
    \end{equation*}
	to deduce that
	\begin{align}\label{intermediar second equation Wh GN}
		\partial_t\overline{V} 
		+ 
		& 
		\mathcal{T}_2^{\mu}[\beta b, \ve \zeta]\nabla_X \zeta 
		+
		\frac{\ve}{2}\nabla_X (|\overline{V}|^2)
		- \frac{\mu\ve}{3h^2}\big{(}\nabla_X\cdot(h\overline{V})\big{)}\nabla_X(h^3\nabla_X\cdot\overline{V})
		\\
		& \notag
		\hspace{0cm}
		+
		\frac{\mu \ve}{h} \nabla_X \big( h^2\big{(}\nabla_X\cdot(h\overline{V})\big{)}\nabla_X\cdot \overline{V}\big)
		+
		\frac{\mu \ve}{3h}\nabla_X \big{(} h^3 \Delta_X(|\overline{V}|^2)\big{)}
		\\
		& \notag
		\hspace{0cm}
		-
		 \frac{\mu\ve }{3} \nabla_X\Big{(}\frac{1}{h}\big{(}\nabla_X(h^3 \nabla_X \cdot \overline{V})\big{)} \cdot \overline{V}\Big{)}
		-
		\frac{\mu \ve}{2}\nabla_X\big{(} h^2 |\nabla_X \overline{V}|^2\big{)} 
        \\ 
        &  =  (\mu^2\ve + \mu \ve \beta + \mu \beta^2)R.\notag
    \end{align}
    Now, using $\sqrt{\mathrm{F}_2} = 1 + \mu R$ and $\partial_t h = \ve \nabla_X\cdot(h\overline{V})$ remark that
        \begin{align*}
            & \partial_t(\overline{V} -\frac{\mu}{3h}\sqrt{\mathrm{F}_2}\nabla_X(h^3\sqrt{\mathrm{F}_2}\nabla_X\cdot\overline{V})) \\
            & = \partial_t \overline{V} - \frac{\mu}{3h}\sqrt{\mathrm{F}_2}\nabla_X(h^3\sqrt{\mathrm{F}_2}\nabla_X\cdot \partial_t\overline{V}) - \frac{\mu\ve}{3h^2}\nabla_X\cdot(h\overline{V})\nabla_X(h^3\nabla_X\cdot\overline{V}) 
            \\
            &\hspace{0.5cm}+ \frac{\mu\ve}{h} \nabla_X \big( h^2\nabla_X\cdot(h\overline{V}) F_2 \nabla_X\cdot \overline{V}\big)
        \end{align*}
        So that from \eqref{intermediar second equation Wh GN}, and discarding all the terms of order $ O(\mu^2\ve + \mu \ve \beta + \mu \beta^2)$, we get
        \begin{align*}
            &\partial_t(\overline{V} -\frac{\mu}{3h}\sqrt{\mathrm{F}_2}\nabla_X(h^3\sqrt{\mathrm{F}_2}\nabla_X\cdot\overline{V})) \\
            = &-\mathcal{T}_2^{\mu}[\beta b, h] \nabla_X \zeta - \frac{\ve}{2}  \nabla_X \big( |\overline{V}|^2\big) - \frac{\mu\ve}{3h} \nabla_X \Big( h^3 \Delta_X (|\overline{V}|^2) \Big) + \frac{\mu\ve}{2} \nabla_X(h^2 (\nabla_X \cdot \overline{V})^2)\\
            & + \frac{\mu\ve }{3} \nabla_X\Big{(}\frac{1}{h}\big{(}\nabla_X(h^3 \nabla_X \cdot \overline{V})\big{)} \cdot \overline{V}\Big{)}+ \frac{\mu}{3h}\nabla_X(h^3\nabla_X\cdot(\mathcal{T}_2[\beta b, h]\nabla_X\zeta)) 
            \\ 
            & 
            + 
            \frac{\mu\ve}{2}\nabla_X(h^3\Delta_X(|\overline{V}|^2))
             + (\mu^2\ve + \mu \ve \beta + \mu \beta^2)R,
        \end{align*}
        which at the end gives:
        \begin{align*}
            &\partial_t(\mathcal{I}^{\mu}[h]\overline{V}) + \mathcal{I}[h]\mathcal{T}_2^{\mu}[\beta b, h]\nabla_X\zeta + \frac{\ve}{2}\nabla_X \big( |\overline{V}|^2\big) + \mu\ve \nabla_X\mathcal{R}_1^{\mu}[\beta b, h, \overline{V}] = (\mu^2\ve + \mu \ve \beta + \mu \beta^2)R.
        \end{align*}

    \end{proof} 

    \subsubsection{Hamiltonian structure} We end this section by briefly commenting on the Hamiltonian structure of the Green-Naghdi type systems with bathymetry. Starting from the Hamiltonian of the water waves equations \eqref{Hamiltonian} and replacing the Dirichlet-Neumann operator by the approximation \eqref{second approximation of the DN operator}, using also \eqref{from F2 to sqrt F2}, we get
    \begin{align*}
        H(\zeta,\psi) = \frac{1}{2} \int_{\mathbb{R}^d} \zeta^2 \: \mathrm{d}X + \frac{1}{2}\int_{\mathbb{R}^d} h |\nabla_X\psi|^2 \: \mathrm{d}X - \frac{\mu}{6} \int_{\mathbb{R}^d}  \psi \Delta_X\sqrt{\mathrm{F}_2}\big(\frac{h^3}{h_b^3}\sqrt{\mathrm{F}_2}\Delta_X\psi\big) \: \mathrm{d}X \\
        - \frac{\mu\beta}{2}\int_{\mathbb{R}^d} (\Delta_X\psi) \mathcal{L}_2^{\mu}[\beta b]\Delta_X\psi \: \mathrm{d}X + \frac{\mu\beta}{4}\int_{\mathbb{R}^d} \nabla_X\psi \cdot \mathrm{F}_4 \Delta_X\big(\mathcal{L}_1^{\mu}[\beta b]\nabla_X\psi\big) \: \mathrm{d}X \\
        + \frac{\mu\beta^2}{4} \int_{\mathbb{R}^d} \psi \nabla_X\cdot\big(\mathcal{B}[\beta b]\nabla_X\psi) \: \mathrm{d}X + O(\mu^2\varepsilon + \mu\ve\beta + \mu^2\beta^2).
    \end{align*}
    Then, we make use of the two expansions given by \eqref{L1 approx next order} and \eqref{L2 next order}
    \begin{align*}
        \mathcal{L}_1^{\mu}[\beta b] = -b \mathrm{F_3} + O(\mu\beta^2), \quad \mathcal{L}_2^{\mu}[\beta b] = -\frac{1}{2}b\mathrm{F}_4 + \frac{\beta^2}{6}b^3 \mathrm{F}_3 + O(\mu\beta^4),
    \end{align*}
    and write
    \begin{align*}
        \mathrm{RHS}_1 := &- \frac{\mu\beta}{2}\int_{\mathbb{R}^d} (\Delta_X\psi) \mathcal{L}_2^{\mu}[\beta b]\Delta_X\psi \: \mathrm{d}X \\
        = \hspace{0.2cm} &\frac{\mu\beta}{4}\int_{\mathbb{R}^d}b (\Delta_X\psi) \mathrm{F}_3\Delta_X\psi \: \mathrm{d}X - \frac{\mu\beta^3}{12}\int_{\mathbb{R}^d}b^3 (\Delta_X\psi) \mathrm{F}_3\Delta_X\psi \: \mathrm{d}X +O(\mu^2\beta^5)\\
        = \hspace{0.2cm} &\frac{\mu\beta}{4}\int_{\mathbb{R}^d}b (\Delta_X\psi) \mathrm{F}_3\Delta_X\psi \: \mathrm{d}X - \frac{\mu\beta^3}{12}\int_{\mathbb{R}^d}b^3 \big(\sqrt{\mathrm{F}_3}\Delta_X\psi\big)^2 \: \mathrm{d}X +O(\mu^2\beta^5),
    \end{align*}
    where for the last equality, we used $\mathrm{F}_2 = \sqrt{\mathrm{F}_2} + O(\mu)$, and 
    \begin{align*}
        \mathrm{RHS}_2 := & \frac{\mu\beta}{4}\int_{\mathbb{R}^d} \nabla_X\psi \cdot \mathrm{F}_4 \Delta_X\big(\mathcal{L}_1^{\mu}[\beta b]\nabla_X\psi\big) \: \mathrm{d}X\\
        = \hspace{0.2cm} & - \frac{\mu\beta}{4}\int_{\mathbb{R}^d} \nabla_X\psi \cdot \mathrm{F}_4 \Delta_X\big(b\nabla_X\psi\big) \: \mathrm{d}X + O(\mu^2\beta^3).
        \end{align*}
    
    Deriving the equations associated to this approximated Hamiltonian thus obtained, we get the Green-Naghdi type system
    \begin{align*}
        \begin{cases}
            \partial_t \zeta 
                +
                \nabla_X\cdot(h\nabla_X\psi) + \frac{\mu}{3}\Delta_X\sqrt{\mathrm{F}_2}\big(\frac{h^3}{h_b^3} \sqrt{\mathrm{F}_2}\Delta_X\psi \big) - \frac{\mu\beta}{4}\nabla_X\cdot\big(\mathcal{Q}^{\mu}[b]\nabla_X\psi\big) \\
                \hspace{3cm}- \frac{\mu\beta^2}{4}\nabla_X\cdot\Big(\big(\mathcal{B}[\beta b] + \mathcal{B}[\beta b]^*\big)\nabla_X\psi\Big)+ \frac{\mu\beta^3}{6}\Delta_X\sqrt{\mathrm{F}_3}(b^3\sqrt{\mathrm{F}_3}\Delta_X\psi) = 0
                \\
                \partial_t \psi + \zeta + \dfrac{\varepsilon}{2}|\nabla_X\psi|^2 - \dfrac{\mu\varepsilon}{2}h^2(\sqrt{\mathrm{F}_2}\Delta_X\psi)^2 = 0
        \end{cases}
    \end{align*}
    where
    \begin{align*}
        \mathcal{Q}^{\mu}[b]\bullet = \Big(\mathrm{F}_3\nabla_X(b\nabla_X\cdot\bullet) + \nabla_X(b\mathrm{F}_3\nabla_X\cdot\bullet)\Big) + \Big(\mathrm{F}_4\nabla_X\nabla_X\cdot\big(b\bullet\big) + b \mathrm{F}_4\Delta_X\bullet\Big),
    \end{align*}
    and
    \begin{align*}
        \mathcal{B}[\beta b]\bullet &= 
            b \mathrm{F}_4\nabla_X(\nabla_X\cdot(b\bullet))
            +
            h_b
             \nabla_X \big{(}b\mathrm{F}_4\nabla_X\cdot(b\bullet)\big{)}
             +
             2h_b(\nabla_Xb)\mathrm{F}_1\nabla_X\cdot(b\bullet).
    \end{align*}

    \appendix
	
    \section{}

    \subsection{On the properties of  pseudo-differential operators}\label{A1} 
    In this section, we will give a rigorous meaning to the pseudo-differential operators given in Proposition \ref{Prop op}. Before turning to the proof, we recall the definition of a symbol.
    \begin{Def}\label{Pseudo diff}
        Let $d=1,2$ and $m\in\R$. We say $L \in S^m$ is a symbol of order $m$ if $L(X,\xi)$ is $C^{\infty}(\R^d \times \R^d)$ and satisfies
        \begin{equation*}
            \quad\forall \alpha \in \mathbb{N}^d, \quad \forall \gamma \in \mathbb{N}^d, \quad \langle \xi \rangle^{-(m -|\gamma|)}|\partial_X^{\alpha} \partial_{\xi}^{\gamma} L(X,\xi)| < \infty.
        \end{equation*}
        We also introduce the seminorm
        \begin{equation}\label{Semi norm}
            \mathcal{M}_m(L)  
            =
            \sup\limits_{|\alpha|\leq \lceil \frac{d}{2} \rceil + 1}\sup \limits_{|\gamma|\leq \lceil \frac{d}{2} \rceil + 1}   \sup \limits_{(X,\xi)\in \R^d\times \R^d } \Big{\{} \langle \xi \rangle^{-(m -|\gamma|)}|\partial_X^{\alpha} \partial_{\xi}^{\gamma} L(X,\xi)| \Big{\}}.
        \end{equation}
    \end{Def}
    \noindent
    Moreover, we recall the main tool we will use to justify the pseudo-differential operators in Sobolev spaces:
    %
    %
    \begin{thm}\label{C-V thm}
        Let $d = 1,2$, $s \geq 0$, and $L \in S^m$. Then formula \eqref{pseudo diff op} defines a bounded pseudo-differential operator from $H^{s+m}(\R^d)$ to $H^{s}(\R)$ and satisfies
        \begin{equation}\label{est Stein}
            |\mathcal{L}[X,D]u|_{H^{s}} \leq \mathcal{M}_m(L)  |u|_{H^{s+m}}.
        \end{equation}
    \end{thm}
    \noindent
    With this Theorem at hand, we can now give the proof.

    \begin{proof}[Proof of Proposition \ref{Prop op}]

    \noindent
    We will first prove that for $s \geq 0$ the operators $\mathcal{L}_i^{\mu}$ are a uniformly bounded on $H^s(\R^d)$. To prove this point we need to verify that the symbols: 
    \begin{align*}
        L_1^{\mu}(\beta b(X),\xi) & = -
        \frac{1}{\beta}\sinh{(\beta b(X) \sqrt{\mu}|\xi|)}\mathrm{sech}(\sqrt{\mu}|\xi|)\dfrac{1}{\sqrt{\mu}|\xi|}
        \\
        L_2^{\mu}(\beta b(X),\xi)  
        & = 
        \frac{1}{\beta}(\sinh{(\beta b(X) \sqrt{\mu}|\xi|)}\mathrm{sech}(\sqrt{\mu}|\xi|)\dfrac{1}{\sqrt{\mu}|\xi|} 
        -
        \beta b) \frac{1}{\mu |\xi|^2}
        \\
        L_3^{\mu}(\beta b(X),\xi) 
        & = - \big{(}\cosh(\beta b(X)\sqrt{\mu}|\xi|)\mathrm{sech}(\sqrt{\mu}|\xi|)-1\big{)}\frac{1}{\mu |\xi|^2}
    \end{align*}
    are elements of $S^0$ where the constants $\mathcal{M}_0(L_{i})$ are independent of $\mu$ and $\beta$. We treat each symbol separately.

    We start by proving that the symbol $L^{\mu}_1$ is in $S^0$. To do so, we will split the frequency domain into three regions. First, let $\beta \sqrt{\mu} |\xi|\leq 1$ and $\sqrt{\mu}|\xi| \leq 1$. Then since $L_1^{\mu}(\beta b(X), \xi) \in C^{\infty}(\R^d \times \R^d)$ and the Taylor expansion around $(X,0)$ gives us
    \begin{align*}
        |\partial_X^{\alpha} \partial_{\xi}^{\gamma}L_1^{\mu}(\beta b(X),\xi)| \lesssim 1.
    \end{align*}
    Next, consider the region $\beta \sqrt{\mu} |\xi| > 1$. Then we also have that $|\xi|\geq \sqrt{\mu} |\xi|>1$. With this in mind, we can prove the   necessary decay estimate. Indeed, since $b\in C^{\infty}_c(\R^d)$ satisfies \eqref{non-cavitation 2}, i.e. for $h_{b,\max}\in (0,1)$:
    $$0<h_{b,\min}\leq 1 - \beta  b(X),$$
    combined with $\mathrm{sech}(x)\sim e^{-x}$ and $\sinh(x)\sim e^x$ for $x\in \R$, we have that
    \begin{align*}
        \Big{|} \partial_X^{\alpha} \partial_{\xi}^{\gamma}\Big{(} \frac{\sinh(\beta b(X) \sqrt{\mu} |\xi|)}{\cosh(\sqrt{\mu}|\xi|)}\Big{)}\Big{|} 
        & \lesssim
        \mu^{\frac{|\gamma|}{2}}(1+ \sqrt{\mu}|\xi|)^{|\alpha|}e^{-h_{b,\min}\sqrt{\mu}|\xi|} 
        \\
        & \lesssim (\frac{2}{h_{b,\min}})^{|\alpha|}\mu^{\frac{|\gamma|}{2}}e^{-\frac{1}{2}h_{b,\min}\sqrt{\mu}|\xi|}.
    \end{align*}
    Additionally, there holds
    \begin{equation*}
        \mu^{\frac{|\gamma|}{2}} (1+|\xi|)^{|\gamma|}e^{-\sqrt{\mu}|\xi|} \lesssim 1,
    \end{equation*}
    and so we obtain the estimate
    \begin{align}\label{usefull stuff}
       \Big{|} \partial_X^{\alpha} \partial_{\xi}^{\gamma}\Big{(} \frac{\sinh(\beta b(X) \sqrt{\mu} |\xi|)}{\cosh(\sqrt{\mu}|\xi|)}\Big{)}\Big{|}  
        & \lesssim (1+|\xi|)^{-|\gamma|}.
    \end{align}
    Then combining this estimate with the Leibniz rule we obtain that
    \begin{align*}
        |\partial_X^{\alpha} \partial_{\xi}^{\gamma} L_1^{\mu}(\beta b(X),\xi)| 
        & \lesssim 
        \frac{1}{\beta}
        \underset{{\gamma_1\in \N^d \: : \: \gamma_1\leq \gamma}}{\sum}
        \Big{|} \partial_X^{\alpha} \partial_{\xi}^{\gamma_1}\Big{(} \frac{\sinh(\beta b(X) \sqrt{\mu} |\xi|)}{\cosh(\sqrt{\mu}|\xi|)}\Big{)}\Big{|} 
        \: 
        \big{|}  \partial_{\xi}^{\gamma-\gamma_1}\Big{(}(\sqrt{\mu}|\xi|)^{-1}\Big{)}\big{|}
        \\
        & \lesssim (\beta\sqrt{\mu}|\xi|)^{-1}
        \underset{{\gamma_1\in \N^d \: : \: \gamma_1\leq \gamma}}{\sum}
        (1+|\xi|)^{-|\gamma_1|}|\xi|^{|\gamma_1|-|\gamma|}
        \\
        & 
        \lesssim
        |\xi|^{-|\gamma|}.
    \end{align*}
    However, since we have that $|\xi|\geq \sqrt{\mu}|\xi|>1$, we obtain the desired result
    \begin{equation}\label{decay L1}
        |\partial_X^{\alpha} \partial_{\xi}^{\gamma} L_1^{\mu}(\beta b(X),\xi)| \lesssim \langle \xi\rangle^{-|\gamma|}.
    \end{equation}
    Lastly, let $\beta \sqrt{\mu} |\xi|\leq 1$ and $\sqrt{\mu}|\xi| \geq 1$. In this case, we expand $x\mapsto\sinh(x)$ to obtain
    \begin{align*}
        L_1^{\mu}(\beta b(X),\xi) & =  - b(X)\mathrm{sech}(\sqrt{\mu}|\xi|)  \\
       &
       -
       \frac{1}{6\beta} \Big(\beta^3 b(X)^3 \int_0^1  \cosh(t\beta b(X) \sqrt{\mu}|\xi|)  (1-t)^2 \: \mathrm{d}t\Big) (\sqrt{\mu}|\xi|)^2\mathrm{sech}(\sqrt{\mu}|\xi|). 
    \end{align*}
    To conclude, we observe that
    \begin{align*}
        |\partial_X^{\alpha} \partial_{\xi}^{\gamma}\big{(} b(X)\mathrm{sech}(\sqrt{\mu}|\xi|) \big{)}| 
        & \lesssim  
        \mu^{\frac{|\gamma|}{2}}e^{-\sqrt{\mu}|\xi|} 
        \lesssim \langle  \xi\rangle^{-|\gamma|} .
    \end{align*}
    For the second term, we let $t\in [0,1]$ and observe that we only need to consider  
    \begin{align*}
        \frac{\cosh(t \beta b(X) \sqrt{\mu}|\xi|)}{\cosh(\sqrt{\mu}|\xi|)}(\sqrt{\mu}|\xi|)^2,
    \end{align*}
    for which the decay estimate follows similarly to \eqref{usefull stuff}. Indeed, we first observe that
    \begin{align*}
        \Big{|} \partial_X^{\alpha} \partial_{\xi}^{\gamma}\Big{(} \frac{\cosh(t\beta b(X) \sqrt{\mu} |\xi|)}{\cosh(\sqrt{\mu}|\xi|)}\Big{)}\Big{|} 
        & \lesssim
        \mu^{\frac{|\gamma|}{2}}t^{|\alpha|}(1+ \sqrt{\mu}t|\xi|)^{|\alpha|}e^{-h_{b,\min}\sqrt{\mu}|\xi|} 
        \\
        & 
        \lesssim \mu^{\frac{|\gamma|}{2}} e^{-\frac{h_{b,\min}}{2}\sqrt{\mu}|\xi|},
    \end{align*}
    since $t\in [0,1]$. Then by the Leibniz rule, we get
    \begin{align*}
        \Big{|} \partial_X^{\alpha} \partial_{\xi}^{\gamma}\Big{(} \frac{\cosh(t\beta b(X) \sqrt{\mu} |\xi|)}{\cosh(\sqrt{\mu}|\xi|)} \mu|\xi|^2\Big{)}\Big{|} 
        & 
        \lesssim \underset{{\gamma_1\in \N^d \: : \: \gamma_1\leq \gamma}}{\sum}
        \mu^{\frac{|\gamma_1|}{2} +1}e^{-\frac{h_{b,\min}}{2}\sqrt{\mu}|\xi|} |\xi|^{2+|\gamma_1|-|\gamma|}
        \\
         & 
         \lesssim   |\xi|^{-|\gamma|},
    \end{align*}
    for $|\xi| \geq \sqrt{\mu}|\xi|>1$. Consequently, we can conclude this case. By Theorem \eqref{C-V thm} there holds,
    \begin{align*}
        |\mathcal{L}_1^{\mu}[\beta b]u|_{H^{s}} \leq M(s)|u|_{H^{s}}.
    \end{align*}

    For the symbol $L^{\mu}_2$,  we observe for $\beta \sqrt{\mu}|\xi|\leq 1$ and $\sqrt{\mu}|\xi|\leq 1$ and a Taylor expansion that it is smooth and bounded.   Moreover, for frequencies such that $\beta \sqrt{\mu}|\xi|> 1$, we can argue as we did for $L^{\mu}_1$ to get sufficient decay in the frequency variable at infinity. Lastly, in the case $\beta \sqrt{\mu} |\xi|\leq 1$ and $\sqrt{\mu}|\xi| \geq 1$ we use the following expansion
    \begin{align*}
        L_2^{\mu}(\beta b(X),\xi) & =  b(X) \big{(} \mathrm{sech}(\sqrt{\mu}|\xi|) -1 \big{)}\frac{1}{\mu |\xi|^2} \\
       &
       +
       \frac{1}{6\beta} \Big(\beta^3 b(X)^3 \int_0^1  \cosh(t\beta b(X) \sqrt{\mu}|\xi|)  (1-t)^2\mathrm{d}t\Big) \mathrm{sech}(\sqrt{\mu}|\xi|),
    \end{align*}
    and again argue as we did for $L^{\mu}_1$.  

    The estimates on $L_3^{\mu}$ is simpler since it does not depend on $\frac{1}{\beta}$. Thus, using similar arguments we can prove the necessary decay at infinity, and a Taylor series to prove the boundedness for small frequencies.

    The estimate \eqref{L approx} follows directly from the boundedness on $H^s(\mathbb{R}^d)$ of $\mathcal{L}_2^{\mu}[\beta b]$ since its symbol is in $S^0$ and that 
    $$L_1^{\mu}(\beta b(X), \xi) +   b(X) =- \mu L_2^{\mu}(\beta b(X),\xi)|\xi|^2.$$
    \noindent
    Indeed, the symbol $r_1(X,\xi) = L_2^{\mu}(\beta b(X),\xi)|\xi|^2$ is an element of $S^2$ and by Theorem \ref{C-V thm} we deduce that $\mathcal{R}_1[X, \mathrm{D}]u(X) = \mathcal{F}^{-1}(r_1(X,\xi)\hat{u}(\xi))(X)$ satisfies
    \begin{equation*}
        |\mathcal{R}_1[X, \mathrm{D}]u|_{H^s} \lesssim |u|_{H^{s+2}},
    \end{equation*}
    so that
    \begin{align*}
        |\mathcal{L}_1^{\mu}[\beta b] u + b u|_{H^s} 
        = \mu |\mathcal{R}[X,\mathrm{D}]u|_{H^s}
        \lesssim \mu |u|_{H^{s+2}}.
    \end{align*}

    The next estimate, given by \eqref{L1 approx next order}, is deduced from the Taylor expansion of the symbol $L_1^{\mu}$ given by:
    \begin{align*}
        L_1^{\mu}(\beta b(X),\xi) & =  - \Big{(}b(X) 
        +
       \frac{\mu \beta^2}{6}  b(X)^3 |\xi|^2 \Big{)}\mathrm{sech}(\sqrt{\mu}|\xi|) -
       \frac{\mu^2 \beta^4}{120}b(X)^5 |\xi|^4 r_2(X,\xi), 
    \end{align*}
    where the rest $r_2$ is given by
    \begin{equation*}
       r_2(X,\xi) = \int_0^1\cosh(t\beta b(X) \sqrt{\mu}|\xi|)  (1-t)^4 \mathrm{d}t\: \mathrm{sech}(\sqrt{\mu}|\xi|),
    \end{equation*}
    and is an element of $S^0$ by arguing as above. By extension, the symbol $b(X)^5|\xi|^4r_2(X,\xi) \in S^4$ and we conclude by Theorem \ref{C-V thm}.

    Lastly, we consider estimate \eqref{L2 next order}. Again by a Taylor series expansion, we observe that
    \begin{align*}
        L_2^{\mu}(\beta b(X),\xi) & =   \Big{(}b(X)(\mathrm{sech}(\sqrt{\mu}|\xi|)-1)\frac{1}{\mu|\xi|^2} 
        +
       \frac{ \beta^2}{6}  b(X)^3  \Big{)}\mathrm{sech}(\sqrt{\mu}|\xi|) 
       \\ 
       & 
       \hspace{0.5cm}
       +
       \frac{\mu \beta^4}{120}b(X)^5 |\xi|^2r_2(X,\xi),
    \end{align*}
    where $b(X)^5|\xi|^2r_2(X,\xi) \in S^2$, allowing us to conclude by Theorem \ref{C-V thm}.

    \end{proof}
    \begin{remark}
        We note that we could improve the estimates in the proof above. For instance, we can get $L_1 \in S^{-1}$. However, the constant $\mathcal{M}_{-1}(L_1)$ would be singular with respect to $\beta$ and $\mu$. 
    \end{remark}

    \subsection{Technical estimates}\label{A2} In this section we give a series of multiplier estimates. To start, we recall the Fourier multiplier depending on the transverse variable:
    \begin{equation*}
        \mathrm{F}_0u(X) = \mathcal{F}^{-1}\Big( \frac{\cosh((z+1)\sqrt{\mu}|\xi|)}{\cosh(\sqrt{\mu}|\xi|)} \hat{u}(\xi)\Big)(X).
    \end{equation*}
    Then the first result reads:
    \begin{prop}\label{Prop F0}
        Let $s\in \R$ and take $u \in \mathscr{S}(\R^d)$, then there holds
        \begin{align*}
            |\mathrm{F}_0u|_{H^s} & \lesssim |u|_{H^s}\hspace{0.3cm}
            \\
            |\partial_z\mathrm{F}_0u|_{H^s}  & \lesssim \mu|\nabla_Xu|_{H^{s+1}}
            \\
            |\partial_z^2\mathrm{F}_0u|  & \lesssim \mu|\nabla_Xu|_{H^{s+1}}.
        \end{align*}
        Moreover, for $k\in \N $ and under condition \eqref{non-cavitation 2} we have similar estimates on the domain $\mathcal{S}_b = \R^d \times[-1+ \beta b, 0]$:
        \begin{align*}
         \|\mathrm{F}_0u - u\|_{H^{k,0}(\mathcal{S}_b)} & \lesssim \mu | \nabla_X u |_{H^{k+1}}
        \\
         \|\partial_z \mathrm{F}_0u\|_{H^{k,0}(\mathcal{S}_b)}  & \lesssim \mu | \nabla_X u |_{H^{k+1}}
        \\
         \|\partial^2_{z}\mathrm{F}_0u\|_{H^{k,0}(\mathcal{S}_b)} & \lesssim \mu | \nabla_X u|_{H^{k+1}}.
    \end{align*}
    \end{prop}
    \begin{proof}
        The estimates on $H^s(\R^d)$ are a direct consequence of Plancherel's identity and the Taylor expansion formula for $x \in \R$:
        \begin{align*}
            \cosh(x) 
            & =
            1 + \frac{x^2}{2} \int_0^1\cosh(tx)(1-t) \:\mathrm{d}t.
        \end{align*}

        For the estimates on $\mathcal{S}_b$, we use that $-h_b(X)>-2$, by assumption \eqref{non-cavitation 2}, then extend the definition of $\mathrm{F}_0$ to the domain $\mathcal{S} := \R^d \times [-2,0]$. The first estimate on $\mathcal{S}_b$ is a consequence of 
        \begin{align*}
            \|\mathrm{F}_0u - u\|_{H^{k,0}(\mathcal{S}_b)} \leq \|\mathrm{F}_0u - u\|_{H^{k,0}(\mathcal{S})} & = \Big{\|}\dfrac{\cosh{((z+1)\sqrt{\mu}|\mathrm{D}|)}}{\cosh{(\sqrt{\mu}|\mathrm{D}|)}}u - u\Big{\|}_{H^{k,0}(\mathcal{S})} \lesssim \mu | \nabla_Xu |_{H^{k+1}}.
        \end{align*}
        The remaining estimates are proved similarly. 
    \end{proof}

    The next result concerns the following operators:
    \begin{equation*}
        T_1(z)[X,\mathrm{D}] u (X)= \mathcal{F}^{-1}\Big{(} \frac{\sinh(\frac{z}{h_b(X)}\sqrt{\mu}|\xi|)}{\cosh(\sqrt{\mu}|\xi|)} \hat{u}(\xi)\Big{)}(X),
    \end{equation*}
    and
    \begin{equation*}
        T_2(z)[X,\mathrm{D}] u(X) = \mathcal{F}^{-1}\Big{(} \frac{\cosh(\frac{z}{h_b(X)}\sqrt{\mu}|\xi|)}{\cosh(\sqrt{\mu}|\xi|)} \hat{u}(\xi)\Big{)}(X).
    \end{equation*}
    We should note that we will apply these operators to functions depending only on $X$, which makes the dependence in $z \in [-h_b,0]$ easier to deal with.

    \begin{prop}\label{Prop T1 and T2}
        Let $k \in \N$ and take $u \in \mathscr{S}(\R^d)$, then under condition
        \eqref{non-cavitation 2} we have 
        \begin{align*}
         \|T_1u \|_{H^{k,0}(\mathcal{S}_b)} & \leq  M(k)|  u |_{H^{k}}
        \\
         \|T_2u\|_{H^{k,0}(\mathcal{S}_b)}   & \leq  M(k) |   u |_{H^{k}}.
    \end{align*}
    \end{prop}
    \begin{proof}
        We first observe that $T_1$ is well-defined on $\mathscr{S}(\R^d)$. Indeed, for $t_0> \frac{d}{2}$ there holds
        \begin{align*}
            |T_1u(X)| 
            & \leq
            \sup\limits_{z \in [-h_b(X),0]} 
            \Big{|} \mathcal{F}^{-1}\Big{(} \frac{\sinh(\frac{z}{h_b(X)}\sqrt{\mu}|\xi|)}{\cosh(\sqrt{\mu}|\xi|)} \hat{u}(\xi)\Big{)}(X)\Big{|}
            \\
            & \leq  
            \sup\limits_{\xi \in \R^d}\:  \tanh(\sqrt{\mu}|\xi|)\langle \xi \rangle^{t_0}  
            |\hat{u}(\xi)|
            \int_{\R^d} \langle\xi\rangle^{-t_0} \: \mathrm{d}\xi
            \\
            &<\infty.           
         \end{align*}
         Moreover, using similar arguments one can prove $T_1u \in \mathscr{S}(\R^d)$. The same is true for $T_2$. 
         
         Next, we prove the estimates. To do so, we first let $k=0$ and use a change of variable, Hölder's inequality, the Sobolev embedding, and Plancherel's identity to make the observation:
         \begin{align*}
             \|T_1 u\|_{L^2(\mathcal{S}_b)}^2
             & = 
             \int_{\R^d}h_b(X)\int_{-1}^0 |\mathcal{F}^{-1}\Big{(}\frac{\sinh(z\sqrt{\mu}|\xi|)}{\cosh(\sqrt{\mu}|\xi|)} \hat{u}(\xi) \Big{)}(X)|^2\: \mathrm{d}z \mathrm{d}X
             \\
             & 
             \leq |h_b|_{L^{\infty}}
             \int_{-1}^0\int_{\R^d} |\mathcal{F}^{-1}\Big{(} \frac{\sinh(z\sqrt{\mu}|\xi|)}{\cosh(\sqrt{\mu}|\xi|)} \hat{u}(\xi) \Big{)}(X)|^2\:  \mathrm{d}X \mathrm{d}z
             \\
             & 
             \leq M_0|u|_{L^2}^2.
         \end{align*}
         For higher derivatives, the proof is the same after an application of the chain rule. The same is true for $T_2$.
         
    \end{proof}

    The next result is on the Dirichlet-Neumann operator (Theorem $3.15$ in \cite{WWP}): 

    \begin{prop}\label{Prop: D-N op}
        Let $s \geq 0$. Let $\zeta \in H^{s+3}(\R^d)$ be such that \eqref{non-cavitation 1} is satisfied, and take $\psi \in \dot{H}^{s+3}(\R^d)$. Then one has
        \begin{equation}\label{G est 1}
            \frac{1}{\mu}|\mathcal{G}^{\mu} \psi|_{H^{s+1}} \leq M(s+3)|\nabla_X \psi|_{H^{s+2}}.
        \end{equation}
        %
        %
        %
        %
        %
        %
        
    \end{prop} 
    
    Lastly, we have the following estimates on the multipliers:
    \begin{align*}
            \mathrm{F}_1 = \dfrac{\tanh{(\sqrt{\mu}|\mathrm{D}|)}}{\sqrt{\mu}|\mathrm{D}|},\quad \mathrm{F}_2 = \frac{3}{\mu |\mathrm{D}|^2}(1- \mathrm{F}_1), \quad  \mathrm{F}_3 =\mathrm{sech}(\sqrt{\mu}|D|), \quad \mathrm{F}_4 =  \frac{2}{\mu |\mathrm{D}|^2}(1- \mathrm{F}_3).
    \end{align*}
    \begin{prop}\label{Simple est}
        Let $s\in \R$ and take $u \in \mathscr{S}(\R^d)$, then for $i\in \{1,2,3,4\}$ there holds
        \begin{align*}
            |(\mathrm{F}_i -1)u|_{H^s}  & \lesssim \mu |\nabla_Xu|_{H^{s+1}}.
        \end{align*}
    \end{prop}
    \begin{proof}
        The estimates are a direct consequence of Plancherel's identity and the Taylor expansion formulas:
        \begin{align*}
            \cosh(x) 
            & =
            1 + \frac{x^2}{2} \int_0^1\cosh(tx)(1-t) \:\mathrm{d}t
            \\
            \sinh(x) 
            & =
            x + \frac{x^3}{6} \int_0^1\cosh(tx)(1-t)^2 \:\mathrm{d}t,
            \\ 
            \frac{1}{\cosh(x)} 
            & =
            1 - \frac{x^2}{2}  
            + 
            \frac{x^4}{24}
            \int_0^1 \Big(\mathrm{sech}(tx) - 20 \mathrm{sech}^3(tx) + 24 \mathrm{sech}^5(tx)\Big)(1-t)^3\:\mathrm{d}t
        \end{align*}
        for $0\leq x \leq 1$.

        \end{proof}
 
    \subsection{Classical estimates}\label{A3}
    In this section, we recall some classical estimates that will be used throughout the paper. Finally, we end the section with the proof of Proposition \ref{Prop elliptic est}.

    \begin{lemma} Let $\beta \in [0,1]$, $b \in C^{\infty}_c(\R^d)$, $h_b = 1-\beta b$,  $\mathcal{S}_b= (-h_b,0)\times \R^d$, and assume \eqref{non-cavitation 2} holds true. Then for $u \in H^1(\mathcal{S}_b)$ satisfying $u|_{z=0} = 0$, there holds
    \begin{equation}\label{Poincare 1}
        \|u\|_{L^2(\mathcal{S}_b)} \lesssim \|\nabla_{X,z}^{\mu} u\|_{L^2(\mathcal{S}_b)},
    \end{equation}
    and
    \begin{equation}\label{Poincare 2}
        |u|_{z=-h_b}|_{L^2} \lesssim \|\nabla_{X,z}^{\mu} u\|_{L^2(\mathcal{S}_b)}.
    \end{equation}
    Moreover, if we further suppose $u\in H^{k,k}(\mathcal{S}_b)$ then
    \begin{align}\label{estimate for V bar}
            \Big{|}\int_{-1+ \beta b(\cdot)}^0u(\cdot,z)\: \mathrm{d}z \Big{|}_{H^k}^2
            \leq M(k)
            \big{(}\|\nabla_{X,z}^{\mu}u\|^2_{H^{k,0}(\mathcal{S}_b)} + \sum\limits_{j=1}^k\|\partial_z^{j}u\|_{H^{k-j,0}(\mathcal{S}_b)}^2\big{)}.
    \end{align}
    \end{lemma}
    \begin{proof}
        For the proof of \eqref{Poincare 1} we use assumption $u|_{z=0} = 0$ and the Fundamental Theorem of Calculus combined with Cauchy-Schwarz inequality  we get that
        \begin{align*}
           \int_{\R^d}\int_{-1+\beta b(X)}^0 |u(X,z)|^2\: \mathrm{d}z \mathrm{d}X
            & = \int_{\R^d}\int_{-1+\beta b(X)}^0 |\int_{z}^0 (\partial_z u)(X,z') \: \mathrm{d}z'|^2\:\mathrm{d}z \mathrm{d}X
            \\
            &\leq (1+ \beta |b|_{L^{\infty}})^2 \int_{\R^d}\int_{-1 + \beta b(X)}^0 |(\partial_z u)(X,z')|^2 \: \mathrm{d}z'\mathrm{d}X.
        \end{align*}

        For the proof of \eqref{Poincare 2}, we first use the assumption $u|_{z=0} = 0$ with the Fundamental Theorem of Calculus and Young's inequality to get that
        \begin{align*}
            \int_{\R} u(X,-h_b(X))^2 \: \mathrm{d}X 
            & =
            \int_{\R} \int_{-1 + \beta b(X)}^0\partial_z\big(u(X,z)^2\big) \: \mathrm{d}z\mathrm{d}X 
            \\
            & \leq
            \int_{\R} \int_{-1+ \beta b(X)}^0 \partial_zu(X,z)^2 \: \mathrm{d}z\mathrm{d}X 
            +
            \int_{\R} \int_{-1 + \beta b(X)}^0 u(X,z)^2\: \mathrm{d}z\mathrm{d}X.
        \end{align*}
        Then by \eqref{Poincare 1} we conclude that 
        \begin{equation*}
            |u|_{-h_b}|_{L^2} \lesssim \|\nabla_{X,z}^{\mu} u\|_{L^2(\mathcal{S}_b)}.
        \end{equation*}

       For the proof \eqref{estimate for V bar}, we first consider the estimate with  one derivative to fix the idea. In particular, we perform a change of variable and then use the chain rule and Hölder's inequality to get
        \begin{align*}
            \Big{|}\nabla_X\int_{-1 + \beta b(\cdot)}^0u(\cdot,z)\: \mathrm{d}z \Big{|}_{L^2}^2
            & = 
            \int_{\R^d} \big{|}\nabla_X \int_{-1}^0u(X,zh_b(X))h_b(X) \: \mathrm{d}z \big{|}^2 \: \mathrm{d}X
            \\
            & 
            \leq
            \int_{\R^d} \big{(}\int_{-1}^0|(\nabla_X u)(X,zh_b(X))| h_b(X)\: \mathrm{d}z \big{)}^2 \: \mathrm{d}X
            \\
            &
            \hspace{0.5cm}
            +
            \int_{\R^d} \big{(} \int_{-1}^0|z|\beta|\nabla_Xb|\:|(\partial_zu)(X,zh_b(X))| h_b(X) \: \mathrm{d}z \big{)}^2 \: \mathrm{d}X
            \\
            & 
            \hspace{0.5cm} 
            +
            \beta |\nabla_X b|_{L^{\infty}}
            \int_{\R^d} \big{(}\int_{-1}^0| u(X,zh_b(X))|\: \mathrm{d}z \big{)}^2 \: \mathrm{d}X.
        \end{align*}
        Next, we can transform the integral back to its original domain using \eqref{non-cavitation 2}, and then apply Cauchy-Schwarz and Hölder's inequality to obtain 
        \begin{align*}
            \Big{|}\nabla_X\int_{-1 + \beta b(\cdot)}^0u(\cdot,z)\: \mathrm{d}z \Big{|}_{L^2}^2
            \leq M(k+1)
            (\|u\|_{H^{1,0}(\mathcal{S}_b)}^2 
            +
            \|\partial_z u\|_{L^2(\mathcal{S}_b)}^2 )
        \end{align*}
        Repeating this process for any $k\in \N$, using the Leibniz rule, gives us
        \begin{align*}
            \Big{|}\int_{-1 + \beta b(\cdot)}^0u(\cdot,z)\: \mathrm{d}z \Big{|}_{H^k}^2
            & =
            \sum\limits_{\gamma \in \N^d \: : \: |\gamma|\leq k}
            \int_{\R^d} \big{|}\partial_X^{\gamma}\int_{-1}^0 \big{(}u(X,zh_b(X)) h_b(X)\big{)} \: \mathrm{d}z\big{|}^2  \: \mathrm{d}X
            \\
            & \leq
            M(k)(\|u\|_{H^{k,0}(\mathcal{S}_b)}^2
            +
            \sum\limits_{j=0}^k\|\partial_z^{j}u\|_{H^{k-j,0}(\mathcal{S}_b)}^2).
        \end{align*}
        To conclude our observation, we use the assumption $u|_{z=0} =0$ to apply the Poincaré inequality \eqref{Poincare 1} on the first terms.
    \end{proof}
    Before proving the main result, we need some classical estimates (see Proposition $B.2$ and Proposition $B.4$ in \cite{WWP}).
    \begin{lemma} Let $t_0 \geq \frac{d}{2}$, $s\geq - t_0$, , $f \in H^{\max\{t_0,s\}}(\R^d)$, and take $g\in H^{s}(\R^d)$ then
    \begin{equation}\label{Classical prod est}
        |fg|_{H^s} \lesssim |f|_{H^{\max\{t_0,s\}}}|g|_{H^s}.
    \end{equation}
    Moreover, if there exist $c_0>0$ and $1+g \geq c_0$ then
    \begin{equation}\label{prod est division}
        \Big{|}\frac{f}{1+g}\Big{|}_{H^s} \lesssim C(c_0,|g|_{L^{\infty}})(1+|f|_{H^s})|g|_{H^s}.
    \end{equation}
    \end{lemma}
    \noindent

    Lastly, we will prove the main result of this section:

    \begin{proof}[Proof of Proposition \ref{Prop elliptic est}]
    We first establish the existence and uniqueness of variational solutions to \eqref{Elliptic est pb}. Here the variational formulation associated with \eqref{Elliptic est pb} is given by
    \begin{align}\label{var form}
        \int_{\mathcal{S}_b}  P (\Sigma_b) \nabla^{\mu}_{X,z} u \cdot  \nabla_{X,z}^{\mu}\varphi \: \mathrm{d}z \mathrm{d}X 
        & =
        \int_{\mathcal{S}_b} f  \varphi \: \mathrm{d}z \mathrm{d}X
        +
        \int_{\R^d}  g \: \varphi|_{z=-h_b} \:\mathrm{d}X,
    \end{align}
    for $\varphi \in H^1(\mathcal{S}_b)$. Then using the coercivity estimate \eqref{Coercivity} and the Poincaré inequality \eqref{Poincare 1} to get that
    \begin{align*}
        c\|\varphi \|_{H^1}\leq \int_{\mathcal{S}_b}  P (\Sigma_b) \nabla^{\mu}_{X,z} \varphi \cdot  \nabla_{X,z}^{\mu} \varphi\: \mathrm{d}z \mathrm{d}X.
    \end{align*}
    While the right-hand side of \eqref{var form} is continuous by Cauchy-Schwarz and the trace inequality \eqref{Poincare 2}. As a result,  by Riesz representation Theorem, there exists a unique variational solution  $u \in H^{1,0}(\mathcal{S}_b)$.

    Next, we will prove that $u \in H^{k,0}(\mathcal{S}_b)$ by considering the problem on the fixed strip $\mathcal{S} = \R^d \times [-1,0]$, where we define
    \begin{align*}
        \Sigma(X,z) = (X, hz + \ve \zeta).
    \end{align*}
    Then we have that
    \begin{align*}
        (u\circ \Sigma_b^{-1})\circ \Sigma(X,z) = u(X,zh_b) : = \tilde u(X,z),
    \end{align*}
    and through a change of variable, we obtain the equation
     \begin{align}
        \int_{\mathcal{S}}   \tilde{P}(\Sigma) \nabla^{\mu}_{X,z} \tilde u \cdot  \nabla_{X,z}^{\mu}\tilde \varphi \: \mathrm{d}z \mathrm{d}X 
        & =
        \int_{\mathcal{S}} \tilde f  \tilde \varphi \: \mathrm{d}z \mathrm{d}X
        +
        \int_{\R^d}  g \: \tilde \varphi|_{z=-1} \:\mathrm{d}X,
    \end{align}
    where $\tilde{f}(X,z) = f(X,zh_b(X))$, $\tilde{\varphi}(X,z) =\varphi(X,zh_b(X))$, and $\tilde{P}(\Sigma)$ is an elliptic matrix given by
    \begin{equation*}
        \tilde{P}(\Sigma) 
	=
	\begin{pmatrix} 
		(1+\partial_z \theta )\mathrm{Id} & -\sqrt{\mu} \nabla_X \theta 
		\\
		-\sqrt{\mu}(\nabla_X \theta)^T & \dfrac{1+ \mu |\nabla_X\theta|^2}{1+\partial_z \theta} 
	\end{pmatrix}.
    \end{equation*}
    with $\theta(X,z) = (\ve \zeta - \beta b)z + \ve \zeta $. At this point, the problem is classical, and we refer to Proposition $4.5$ in \cite{DucheneMMWW21} to deduce that  $\nabla_{X,z}^{\mu}\tilde u  \in H^{k,0}(\mathcal{S})$ for $k \in \N$ and satisfying
    \begin{align*}
        \| \nabla_{X,z}^{\mu} \tilde u \|_{H^{k,0}(\mathcal{S})} 
        & \leq
        M(k+1) ( |g|_{H^k}+\|\tilde{f}\|_{H^{k,0}})
        \\
        &
        \leq 
        M(k+1) (|g|_{H^k} + \sum\limits_{j=0}^k\|\partial_z^j f\|_{H^{k-j,0}}).
    \end{align*}
    In the last inequality, we used the chain rule and the product estimate \eqref{Classical prod est}. Moreover, for $k\geq 1 + t_0$ we have $\tilde u  \in C^2(\mathcal{S})$ and is a classical solution of 
    \begin{equation*}
        \begin{cases}
            \nabla_{X,z}^{\mu} \tilde{P}(\Sigma) \nabla_{X,z}^{\mu} \tilde u  = \tilde{f} 
            \\ 
            v|_{z=0} = 0 , \quad \partial_n^{\tilde{P}}\tilde u |_{z = -1} = g.
        \end{cases}
    \end{equation*}
    Then using the equation, we can control the partial derivatives in $z$ by the derivatives in $X$ through
    \begin{align*}
        \frac{1 + |\nabla_X \theta|^2 }{h} \partial_z^2 \tilde u  =  \tilde{f} - \mu \nabla_X\cdot(h \nabla_X \tilde u ) + \mu \nabla_X\cdot (\nabla_X \theta \partial_z \tilde u )+ \mu \partial_z(\nabla_X \theta \cdot \nabla_X \tilde u ) - \frac{ (\partial_z|\nabla_X \theta|^2)}{h} \partial_z \tilde u ,
    \end{align*}
    and the regularity and positivity of $\frac{1 + |\nabla_X \theta|^2 }{h}$. Indeed, there holds,
    \begin{align*}
         \|  \partial_z^k \tilde u  \|_{L^2(\mathcal{S})} 
         \leq 
         M(k+1) (\| \tilde u  \|_{H^{k,0}(\mathcal{S})} +  \| \partial_z \tilde u  \|_{H^{k,0}(\mathcal{S})} + \| \partial_z^k\tilde{f}\|_{L^2(\mathcal{S}_b)}).
    \end{align*}
    Having the desired regularity, we may relate these observations with the original problem $u$ on $\mathcal{S}_b$. In particular, by \eqref{non-cavitation 2} we have that 
    \begin{equation*}
        \nabla_{X,z}^{\mu}u(X,z) = \nabla_{X,z}^{\mu }\big(\tilde u (X, \frac{z}{h_b})\big) \in H^{k,0}(\mathcal{S}_b),
    \end{equation*}
    using the chain rule, the regularity of $h_b$, and a change of variable to get that
    \begin{align*}
        \|  \nabla_{X,z}^{\mu}u \|_{H^{k,0}(\mathcal{S}_b)} 
        & \leq M(k+1)(\| \nabla_{X,z}^{\mu} \tilde u \|_{H^{k,0}(\mathcal{S})} + \sum \limits_{j=0}^{k} \| \partial_z^j \tilde u\|_{H^{k,0}(\mathcal{S})})
        \\
        & 
        \leq M(k+1)(|g|_{H^k} +  \sum \limits_{j=0}^{k} \|\partial_z^jf\|_{H^{k-j,0}(\mathcal{S}_b)}).
    \end{align*}

    \end{proof}

    \section*{Acknowledgements}
 
    This research was supported by a Trond Mohn Foundation grant. It was also supported by the Faculty Development Competitive Research Grants Program 2022-2024 of Nazarbayev University: Nonlinear Partial Differential Equations in Material Science, Ref. 11022021FD2929. 

    The authors would also like to thank Vincent Duchêne and David Lannes for providing helpful remarks.

    \bibliographystyle{plain}
    \bibliography{Bibli.bib}

\def\cprime{$'$}
\begin{thebibliography}{10}

\bibitem{Alazard21}
Thomas Alazard.
\newblock Free surface flows in fluid dynamics.
\newblock Lecture note, Ecole Normale Supérieure Paris-Saclay, 2021.

\bibitem{AlinhacGerard07}
Serge Alinhac and Patrick G\'{e}rard.
\newblock {\em Pseudo-differential operators and the {N}ash-{M}oser theorem},
  volume~82 of {\em Graduate Studies in Mathematics}.
\newblock American Mathematical Society, Providence, RI, 2007.
\newblock Translated from the 1991 French original by Stephen S. Wilson.

\bibitem{Alvarez-SamaniegoLannes08b}
Borys Alvarez-Samaniego and David Lannes.
\newblock Large time existence for 3{D} water-waves and asymptotics.
\newblock {\em Invent. Math.}, 171(3):485--541, 2008.

\bibitem{Benoit_BP_17}
M\'{e}sognon-Gireau Beno\^{i}t.
\newblock The {C}auchy problem on large time for a {B}oussinesq-{P}eregrine
  equation with large topography variations.
\newblock {\em Adv. Differential Equations}, 22(7-8):457--504, 2017.

\bibitem{Benoit_WW_17}
M\'{e}sognon-Gireau Beno\^{i}t.
\newblock The {C}auchy problem on large time for the water waves equations with
  large topography variations.
\newblock {\em Ann. Inst. H. Poincar\'{e} C Anal. Non Lin\'{e}aire},
  34(1):89--118, 2017.

\bibitem{CarterDinvayKalish21}
John~D. Carter, Evgueni Dinvay, and Henrik Kalisch.
\newblock Fully dispersive {B}oussinesq models with uneven bathymetry.
\newblock {\em J. Engrg. Math.}, 127:Paper No. 10, 14, 2021.

\bibitem{Chazel07}
Florent Chazel.
\newblock Influence of bottom topography on long water waves.
\newblock {\em M2AN Math. Model. Numer. Anal.}, 41(4):771--799, 2007.

\bibitem{CraigSulemGuyenneNicholls05}
Walter Craig, Philippe Guyenne, David~P. Nicholls, and Catherine Sulem.
\newblock Hamiltonian long-wave expansions for water waves over a rough bottom.
\newblock {\em Proc. R. Soc. Lond. Ser. A Math. Phys. Eng. Sci.},
  461(2055):839--873, 2005.

\bibitem{Dingemans94}
M.~W. Dingemans.
\newblock Comparison of computations with {B}oussinesq-like models and
  laboratory measurements.
\newblock Technical report, Deltares, Delft, 1994.

\bibitem{DucheneMMWW21}
Vincent Duch\^{e}ne.
\newblock A unified theoretical approach. analysis of pdes.
\newblock [math.AP]. Université de Rennes 1, 2021. ⟨tel-03282212⟩, 2021.

\bibitem{Emerald21}
Louis Emerald.
\newblock Rigorous derivation from the water waves equations of some full
  dispersion shallow water models.
\newblock {\em SIAM J. Math. Anal.}, 53(4):3772--3800, 2021.

\bibitem{Emerald22}
Louis Emerald.
\newblock Local well-posedness result for a class of non-local quasi-linear
  systems and its application to the justification of whitham-boussinesq
  systems, 2022.

\bibitem{KleinLinaresPilodEtAl18}
Christian Klein, Felipe Linares, Didier Pilod, and Jean-Claude Saut.
\newblock On {W}hitham and related equations.
\newblock {\em Stud. Appl. Math.}, 140(2):133--177, 2018.

\bibitem{WWP}
Lannes.
\newblock {\em The water waves problem: mathematical analysis and asymptotics}.
\newblock Mathematical surveys and monographs; volume 188. American
  Mathematical Society, Rhode Island, United-States, 2013.

\bibitem{Metivier2008}
Guy M\'{e}tivier.
\newblock {\em Para-differential calculus and applications to the {C}auchy
  problem for nonlinear systems}, volume~5 of {\em Centro di Ricerca Matematica
  Ennio De Giorgi (CRM) Series}.
\newblock Edizioni della Normale, Pisa, 2008.

\bibitem{Paulsen22}
Martin~Oen Paulsen.
\newblock Long time well-posedness of {W}hitham-{B}oussinesq systems.
\newblock {\em Nonlinearity}, 35(12):6284--6348, 2022.

\bibitem{WeiKirbyGrilliEtAl95}
Ge~Wei, James~T. Kirby, Stephan~T. Grilli, and Ravishankar Subramanya.
\newblock A fully nonlinear {B}oussinesq model for surface waves. {I}. {H}ighly
  nonlinear unsteady waves.
\newblock {\em J. Fluid Mech.}, 294:71--92, 1995.

\bibitem{Zakharov68}
V.~E. Zakharov.
\newblock Stability of periodic waves of finite amplitude on the surface of a
  deep fluid.
\newblock {\em J. Appl. Mech. Tech. Phys.}, 9:190--194, 1968.

\end{thebibliography}

\end{document}